\documentclass[letterpaper,10pt,reqno]{article}
\usepackage[bb=fourier,cal=pxtx]{mathalfa}
\usepackage{amssymb}
\usepackage{mathtools}		
\usepackage{mathabx}		
\usepackage{nicefrac}
\usepackage{accents} 
\newcommand*\underdot[1]{\underaccent{\dot}{#1}}

\usepackage{indentfirst}	
\usepackage{amsthm}
\usepackage{thmtools}
\usepackage{tikz}			
\usetikzlibrary{arrows}
\usepackage{tikz-cd}		
\usepackage{enumitem}		
\usepackage[backref=page,colorlinks=true]{hyperref}	
\usepackage[font=small]{caption}
\usepackage[capitalise]{cleveref} 

\usepackage{xcolor}
\hypersetup{
    colorlinks,
    linkcolor={red!50!black},
    citecolor={blue!50!black},
    urlcolor={blue!80!black}
}

\renewcommand*{\backref}[1]{}
\renewcommand*{\backrefalt}[4]{\tiny
  \ifcase #1 (\textbf{NOT CITED.})%
  \or    (Cited on p.~{#2}.)%
  \else   (Cited on pp.~{#2}.)%
  \fi}

\makeatletter

\def\MRbibitem{\@ifnextchar[\my@lbibitem\my@bibitem}

\def\mybiblabel#1#2{\@biblabel{{\hyperref{http://www.ams.org/mathscinet-getitem?mr=#1}{}{}{#2}}}}

\def\myhyperanchor#1{\Hy@raisedlink{\hyper@anchorstart{cite.#1}\hyper@anchorend}}

\def\my@lbibitem[#1]#2#3#4\par{%
  \item[\mybiblabel{#2}{#1}\myhyperanchor{#3}\hfill]#4%
  \@ifundefined{ifbackrefparscan}{}{\BR@backref{#3}}%
  \if@filesw{\let\protect\noexpand\immediate
    \write\@auxout{\string\bibcite{#3}{#1}}}\fi\ignorespaces%
}

\def\my@bibitem#1#2#3\par{%
  \refstepcounter\@listctr
  \item[\mybiblabel{#1}{\the\value\@listctr}\myhyperanchor{#2}\hfill]#3%
  \@ifundefined{ifbackrefparscan}{}{\BR@backref{#2}}%
  \if@filesw\immediate\write\@auxout
    {\string\bibcite{#2}{\the\value\@listctr}}\fi\ignorespaces%
}

\makeatother

\let\OLDthebibliography\thebibliography
\renewcommand\thebibliography[1]{
  \OLDthebibliography{#1}
  \setlength{\parskip}{0pt}
  \setlength{\itemsep}{0pt plus 0.3ex}
}

\declaretheoremstyle[
headfont=\small\itshape,
bodyfont=\small
]{myremark}

\declaretheorem[numberwithin=section]{theorem}
\declaretheorem[sibling=theorem]{lemma}
\declaretheorem[sibling=theorem]{corollary}
\declaretheorem[sibling=theorem]{proposition}
\declaretheorem[sibling=theorem,style=myremark]{example}
\declaretheorem[sibling=theorem,style=myremark]{remark}

\declaretheorem[sibling=theorem]{question}
\declaretheorem[name=Acknowledgements, style=definition, numbered=no]{ack}

\hypersetup{bookmarksdepth = 3} 
\numberwithin{equation}{section}     

\Crefname{section}{Section}{Sections}
\Crefname{subsection}{Subsection}{Subsections}

\setlist[enumerate,1]{label={\upshape(\alph*)},ref=\alph*}
\setlist[enumerate,2]{label={\upshape(\arabic*)},ref=\arabic*}

\newcommand{\tribar}[1]{\mathopen{| {\kern -1.5pt} | {\kern -1.5pt} |} {#1}
\mathclose{| {\kern -1.5pt} | {\kern -1.5pt} |}}

\newcommand{\R}{\mathbb{R}}
\newcommand{\Z}{\mathbb{Z}}
\newcommand{\N}{\mathbb{N}}
\newcommand{\E}{\mathbb{E}}
\newcommand{\F}{\mathbb{F}}
\newcommand{\G}{\mathbb{G}}
\newcommand{\K}{\mathbb{K}}
\newcommand{\M}{\mathbb{M}}
\renewcommand{\P}{\mathbb{P}}
\newcommand{\W}{\mathbb{W}}
\newcommand{\id}{\mathrm{id}}
\newcommand{\cB}{\mathcal{B}}
\newcommand{\cE}{\mathcal{E}}
\newcommand{\cG}{\mathcal{G}}\newcommand{\cI}{\mathcal{I}}

\newcommand{\cM}{\mathcal{M}}
\newcommand{\cR}{\mathcal{R}}
\newcommand{\cS}{\mathcal{S}}\newcommand{\cU}{\mathcal{U}}

\newcommand{\st}{\;\mathord{;}\;}

\DeclareMathOperator{\supp}{supp}

\newcommand*\circled[1]{\tikz[baseline=(char.base)]{
    \node[shape=circle,draw,inner sep=1pt] (char) {\footnotesize{#1}};}}

\newcommand{\GL}{\mathit{GL}}
\newcommand{\SL}{\mathit{SL}}
\DeclareMathOperator{\vol}{vol}
\newcommand{\WEDGE}{\mathsf{\Lambda}}  
\newcommand{\dd}{\mathrm{d}}   
\newcommand{\uu}{\mathrm{u}}   
\renewcommand{\ss}{\mathrm{s}}   

\newcommand{\Aut}{\mathrm{Aut}}
\newcommand{\End}{\mathrm{End}}
\DeclareMathOperator{\bol}{bol}
\newcommand{\Gr}{\mathcal{G}}   
\DeclareMathOperator*{\linf}{linf}

\newcommand{\arxiv}[1]{Preprint \href{http://arxiv.org/abs/#1}{arXiv:{#1}}}
\newcommand{\doi}[1]{\href{http://dx.doi.org/#1}{DOI:{#1}}}

\renewcommand{\epsilon}{\varepsilon}
\renewcommand{\phi}{\varphi}
\renewcommand{\setminus}{\smallsetminus}
\renewcommand{\emptyset}{\varnothing}

\makeatletter
\newcommand{\subjclass}[2][1991]{%
  \let\@oldtitle\@title%
  \gdef\@title{\@oldtitle\footnotetext{#1 \emph{Mathematics subject classification.} #2}}%
}
\newcommand{\keywords}[1]{%
  \let\@@oldtitle\@title%
  \gdef\@title{\@@oldtitle\footnotetext{\emph{Key words and phrases.} #1.}}%
}
\newcommand{\ackno}[1]{%
  \let\@@oldtitle\@title%
  \gdef\@title{\@@oldtitle\footnotetext{\emph{Acknowledgements.} #1.}}%
}
\makeatother

\begin{document}

\title{Extremal Norms for Fiber-Bunched Cocycles}
\date{October 4, 2019}



\author{Jairo Bochi \and Eduardo Garibaldi}



\ackno{Bochi was partially supported by projects Fondecyt 1180371 and Conicyt PIA ACT172001. Garibaldi was partially supported by FAPESP's Thematic Project 2012/18780-0}

\maketitle

\begin{abstract}
In traditional Ergodic Optimization, one seeks to maximize Birkhoff averages. The most useful tool in this area is the celebrated Ma\~n\'e Lemma, in its various forms. In this paper, we prove a non-commutative Ma\~n\'e Lemma, suited to the problem of maximization of Lyapunov exponents of linear cocycles or, more generally, vector bundle automorphisms. More precisely, we provide conditions that ensure the existence of an extremal norm, that is, a Finsler norm with respect to which no vector can be expanded in a single iterate by a factor bigger than the maximal asymptotic expansion rate. These conditions are essentially irreducibility and sufficiently strong fiber-bunching. Therefore we extend the classic concept of Barabanov norm, which is used in the study of the joint spectral radius. We obtain several consequences, including sufficient conditions for the existence of Lyapunov maximizing sets. 
\end{abstract}

\section{Introduction}

\subsection{Extremal norms}\label{ss.intro_extremal}

Let $\E$ be a $d$-dimensional real vector bundle over a compact metric space $X$, with projection map $\pi$.
Let $T \colon X \to X$ be a homeomorphism.
We say that $\Phi$ is an \emph{automorphism of $\E$ covering $T$} if
the diagram
$$
\begin{tikzcd}
\E	\arrow[r,"\Phi"] \arrow[d,swap,"\pi"]	& \E \arrow[d,"\pi"] \\
X	\arrow[r,swap,"T"]						& X
\end{tikzcd}
$$
commutes and moreover the restriction of $\Phi$ to each fiber $\E_x \coloneqq \pi^{-1}(x)$ is a linear isomorphism $\Phi_x$ onto the fiber $\E_{Tx}$.
The set of such automorphisms is denoted $\Aut(\E,T)$.
The simplest situation is when the vector bundle is trivial, say $\E = X \times \R^d$.
Then $\Phi$ takes the form
\begin{equation}\label{e.cocycle}
\Phi(x,u) = (T(x),F(x)u) \, ,
\end{equation}
for some continuous map $F \colon X \to \GL(d,\R)$.
The pair $(T,F)$ is called a \emph{(linear) cocycle}.

\medskip

A \emph{Finsler norm}\footnote{Beware that other definitions of Finsler norms appear in the literature; here the main point is that the norm is not necessarily induced by inner products (i.e.\ ``Riemannian'').}
on $\E$ is a continuous map $\| \mathord{\cdot} \| \colon \E \to \R$ whose restriction to each fiber $\E_x$ is a norm.
If $L$ is a linear map from a fiber $\E_x$ to another fiber $\E_y$, then we define the operator norm:
\begin{equation}\label{e.def_operator_norm}
\|L\|_{y \gets x} \coloneqq \sup_{\substack{u \in \E_x \\ u\neq 0}} \frac{\|L(u)\|}{\|u\|} \, .
\end{equation}
When no confusion is likely to arise we denote this simply by $\|L\|$.

Fix an automorphism $\Phi$ covering $T$ and a Finsler norm $\| \mathord{\cdot} \|$.
Given $x \in X$, the limit
$$
\chi_1(\Phi, x) \coloneqq \lim_{n \to +\infty} \frac{1}{n} \log \|\Phi^n_x\|
= \lim_{n \to +\infty} \frac{1}{n} \log \| \Phi_{T^{n-1}x} \circ \cdots \circ \Phi_{Tx} \circ \Phi_x\| \, ,
$$
if it exists, is called the \emph{(first) Lyapunov exponent} of $\Phi$ at the point $x$.
The Lyapunov exponent is obviously independent of the choice of the Finsler norm.
If $\mu$ is a $T$-invariant Borel probability measure for $T$, then the Lyapunov exponent $\chi_1(\Phi, x)$
exists for $\mu$-almost every $x\in X$; this is a well-known consequence of Kingman's subadditive ergodic theorem; see e.g.\ \cite{Krengel}.
Let us denote $\chi_1(\Phi, \mu) \coloneqq \int \chi_1(\Phi,\mathord{\cdot}) \, d\mu$.
If the measure $\mu$ is ergodic then $\chi_1(\Phi, x)  =  \chi_1(\Phi,\mu)$ for $\mu$-almost every $x\in X$.

In this paper we are interested in the \emph{maximal Lyapunov exponent}, defined as:
\begin{equation}\label{e.def_beta}
\beta(\Phi) \coloneqq \sup_{\mu \in \cM_T} \chi_1(\Phi, \mu) \, ,
\end{equation}
where $\cM_T$ denotes the set of all $T$-invariant Borel probability measures.
The supremum is always attained by an ergodic measure -- this follows from upper semicontinuity of $\chi_1(\Phi, \mathord{\cdot})$ with respect to the weak-star topology, and the fact that $\cM_T$ is a compact convex set whose extreme points are exactly the ergodic measures.
Let us mention that the maximal Lyapunov exponent can also be characterized in more elementary terms as follows:
\begin{equation}\label{e.beta_other}
\beta(\Phi) = \linf_{n \to \infty} \frac{1}{n} \sup_{x \in X} \log \| \Phi^n_x \|
= \sup_{x \in X} \limsup_{n \to \infty} \frac{1}{n} \log \| \Phi^n_x \| \, .
\end{equation}
(We use ``$\linf$'' to denote a limit that is also an infimum.)
These equalities follow from general results on ``subadditive ergodic optimization'': see \cite[Appendix~A]{Morris_Mather}.

A trivial upper bound for the maximal Lyapunov exponent, which depends on the chosen Finsler norm, is given by:
\begin{equation} \label{e.starting_point}
\beta(\Phi) \le  \log  \sup_{x \in X} \| \Phi_x \| \, .
\end{equation}
If equality holds then $\| \mathord{\cdot} \|$ is called an \emph{extremal norm} for $\Phi$.
More precisely, the norm is so ``tight'' that
there is \emph{no} vector $u \neq 0$ in $\E$ whose expansion factor $\|\Phi(u)\| / \|u\|$ exceeds the maximum asymptotic expansion rate $e^{\beta(\Phi)}$.
In particular, if $\beta(\Phi) \le 0$ then the extremal norm is a (non-strict) Lyapunov function for $\Phi$.

Extremal norms first appeared in the 1960 paper \cite{RS} by Rota and Strang, who considered the particular setting of one-step cocycles (details are given below), but apparently were not considered in our level of generality before.

The existence of an extremal norm is far from automatic\footnote{On the other hand, one can always construct ``almost-extremal'' norms, i.e., norms for which the inequality \eqref{e.starting_point} is an approximate equality, and such norms can be taken Riemannian. Furthermore, it is possible to find a Riemannian norm with respect to which all the singular values of the linear maps $\Phi_x$ (and not only the first) are suitably controlled: see \cite[Prop.~4.1]{Bochi_ICM}.}, and has strong consequences.
In this paper we construct extremal norms for a large and natural 
class of vector bundle automorphisms.

\subsection{Previous results}\label{ss.known}

Consider the case of a $1$-dimensional vector bundle $\E$, with an arbitrary Finsler norm $\| \mathord{\cdot} \|$.
Given $\Phi \in \Aut(\E,T)$, there exists a unique continuous function $f \colon X \to \R$ such that
\begin{equation}\label{e.operator}
u \in \E_x \quad \Rightarrow \quad \|\Phi(u)\|_{Tx} = e^{f(x)} \| u\|_x \, .
\end{equation}
Note that in this case the maximal Lyapunov exponent $\beta(\Phi)$ equals:
\begin{equation}\label{e.beta_f}
\beta (f) \coloneqq \sup_{\mu \in \cM_T} \int f \, d\mu \, .
\end{equation}
Any other Finsler norm $\tribar{\cdot}$ is of the form:
$$
\tribar{u}_x = e^{h(x)} \| u\|_x \, ,
$$
for some continuous function $h \colon X \to \R$.
Then $\tribar{\cdot}$ is a extremal norm if and only if $h$ satisfies the ``cohomological inequality'':
$$
f + h \circ T  - h \le \beta(f) \, .
$$
Such a function $h$ is called a \emph{subaction} for $(T,f)$.
Existence of subactions can fail dramatically: see e.g.\ \cite[{\S}3]{Bousch_Jenkinson} and \cite[Appendix]{Garibaldi_book}.
However, if the dynamics $T$ is in some sense hyperbolic (e.g., a shift) and the function $f$ is regular enough (e.g., H\"older) then subactions $h$ do exist. Results of this type are sometimes called \emph{Ma\~n\'e lemmas}; see \cite{CG,Sav,CLT,Bousch_Walters,Bousch_amphi} for various versions and approaches, and see \cite[Prop.~2.1]{Bochi_ICM} for a negative result.
Important applications include \cite{Bousch_Mairesse,Contreras}. The study of invariant measures that attain that supremum in \eqref{e.beta_f} is called \emph{ergodic optimization}; we refer the reader to  \cite{Jenkinson_survey,Jenkinson_survey_new,Garibaldi_book} for much more information. For a discussion of ergodic optimization in a more general context, including optimization of Lyapunov exponents, see \cite{Bochi_ICM}.

\medskip

When $\dim \E > 1$, commutativity is lost and much less is known.
The most studied situation is the following one.
Let $T \colon X \to X$ be the full shift on $N$ symbols, defined on the space $X \coloneqq \{0,1,\dots,N-1\}^\Z$.
Given a $N$-tuple $(A_0,\dots,A_{N-1})$ of invertible $d \times d$ matrices, let
$F \colon X \to \GL(d,\R)$ be given by $F(x) = A_{x_0}$.
We say that $(T,F)$ is a \emph{one-step cocycle}.
Let $\Phi$ the associated automorphism \eqref{e.cocycle}.
In that case, the quantity $e^{\beta(\Phi)}$ is known as the \emph{joint spectral radius} of the set $\{A_0,\dots,A_{N-1}\}$.\footnote{More generally, one could consider (possibly infinite) bounded sets of (possibly non-invertible) square matrices.}
It was introduced by Rota and Strang \cite{RS}.

If, for example, $N=1$ and $A_0 = \left( \begin{smallmatrix} 1 & 1 \\ 0 & 1 \end{smallmatrix} \right)$, then no extremal norm exists. However, if the set $\{A_0,\dots,A_{N-1}\}$ is \emph{irreducible}, in the sense that there is no common invariant non-trivial subspace, then extremal norms $\tribar{\cdot}$ do exist, and can be taken so that
$\tribar{u}_x$ is independent of $x \in X$. 
Actually, Barabanov~\cite{Barabanov} proved that there exists a norm $\tribar{\cdot}$ on $\R^d$ with the following stronger property:
\begin{equation}\label{e.Barabanov}
\forall u \in \R^d, \quad
\max_{i \in \{0,\dots,N-1\}} \tribar{A_i u} = e^{\beta(\Phi)} \tribar{u} \, .
\end{equation}
For more information on the joint spectral radius and Barabanov norms, see \cite{Wirth,Jungers}.
Further applications of extremal norms were obtained by Morris \cite{Morris_rapidly,Morris_Mather}.

Still in the setting of one-step cocycles, a modification of the concept of Barabanov norm was used in \cite{BR,BM} to study Lyapunov-maximizing and also Lyapunov-minimizing measures.

Extremal norms for certain locally constant cocycles over sofic shifts have been studied in the papers \cite{PEDJ,CGP}.

The main purpose of this paper is to establish existence of extremal norms in a far more general setting.

\subsection{The main result}

We now describe the hypotheses on the automorphism $\Phi$ and the underlying dynamics $T$ from which we will prove the existence of extremal norms. We first describe them informally, leaving the precise definitions for later sections. 

First, we assume that $T \colon X \to X$ is a transitive \emph{hyperbolic homeomorphism} of a compact metric space $X$. Hyperbolicity means that $T$ has local stable and unstable sets with uniform exponential bounds, which satisfy a local product property. Examples include subshifts of finite type and Anosov diffeomorphisms.

Second, we assume that the vector bundle $\E$ has a H\"older structure, and that the automorphism $\Phi$ respects this structure. In the case of trivial vector bundles, this means that the matrix function $F$ in formula \eqref{e.cocycle} is H\"older continuous.

Third, we assume that the automorphism $\Phi$ is \emph{fiber-bunched}.
In crude terms, this means that the non-conformality of the linear maps $\Phi_x$ is small when compared to the hyperbolicity rates of $T$. The precise condition involves the H\"older exponent of the automorphism, so that more regular automorphisms are allowed to be less conformal. In the case that $T$ and $\Phi$ are differentiable, fiber-bunching means that the projectivization of $\Phi$ is a \emph{partially hyperbolic} diffeomorphism.

Actually, for $d \ge 3$ we need to assume a stronger form of fiber-bunching.

Our last assumption is \emph{irreducibility}, meaning that $\Phi$ admits no nontrivial regular subbundle, where regular means as regular as the automorphism itself.
We remark that this condition is satisfied for typical fiber-bunched automorphisms: it holds on an open and dense subset of infinite codimension.

The main result of this paper is that \emph{under the conditions above, extremal norms exist}. See \cref{c.extremal} for a precise statement.

In the case where the base dynamics $T$ is a subshift of finite type, we are able to improve our main result and obtain an extremal norm with a further property akin to the Barabanov property: see \cref{ss.Barabanov}.

Classical Barabanov norms are usually non-Riemannian (that is, they do not come from inner products), and it is easy to produce examples\footnote{The pair of matrices \eqref{e.two_matrices} is one such example.}. On the other hand, in our setting, there is more flexibility as the norm is allowed to depend on the basepoint. So one could wonder if the Finsler extremal norms in our main result could be taken Riemannian. Unfortunately, that is not the case: we construct an explicit example in \cref{ss.Riemann}.

\subsection{Consequences}

As a consequence of our result on the existence of extremal norms, we can show that the maximal Lyapunov exponent is a locally Lipschitz function on the space of strongly bunched irreducible automorphisms (see \cref{p.Wirth} for a more precise statement), thus extending a result of Wirth~\cite{Wirth}.

We are also able to obtain several general properties of strongly bunched automorphisms $\Phi$ (not necessarily irreducible):
\begin{itemize}

\item Their growth obeys certain uniform bounds: see \cref{t.polynomial}.

\item They obey the \emph{subordination principle}: if $\mu$ and $\nu$ are invariant probability measures such that $\nu$ is Lyapunov maximizing in the sense that $\chi_1(\Phi,\nu) = \beta(\Phi)$, and $\supp \mu \subseteq \supp \nu$, then $\mu$ is Lyapunov maximizing as well: see \cref{t.subordination}.
This property is far from being tautological, even in the commutative setting; in fact it was introduced in this setting by Bousch \cite{Bousch_Walters}.

\item The maximal Lyapunov exponent $\beta(\Phi)$ can be approximated by Lyapunov exponents of measures supported on periodic orbits, and moreover the quality of this approximation is superpolynomial with respect to the period: see \cref{t.super_pol}.
This extends a result of Morris \cite{Morris_rapidly}, who gave a quantitative version of the celebrated theorem of Berger--Wang \cite{BWang}.

\end{itemize}

We also introduce \emph{Mather sets} in our context; these sets are the habitat of Lyapunov maximizing measures. We prove an important structural result on the existence of \emph{dominated splittings} on the Mather sets, namely \cref{t.dom}, which is an essential ingredient in the proof of the aforementioned \cref{t.super_pol}.

\subsection{Organization of the paper}

In \cref{s.setting} we introduce the setting for our results, providing the definitions and properties of fiber-bunched automorphisms and related concepts. 

In \cref{s.subbundles} we study irreducibility and related concepts. 

In \cref{s.bounded} we provide sufficient conditions for \emph{relative product boundedness}, an intermediate property which is required for the existence of extremal norms.

The construction of extremal norms is given in \cref{s.norms}, together with the construction of Barabanov-like norms for shifts and an application to the regularity of $\beta (\mathord{\cdot})$.

In \cref{s.Mather} we introduce Mather sets in a very general setting and, under the assumption of existence of an extremal norm, establish fine properties about them.

In \cref{s.app} we collect several applications of our results.

\cref{s.technical} contains the proofs of several subsidiary results, therefore making the paper self-contained.

In \cref{s.examples} we exhibit some ``pathological'' examples, including an example that fits in the setting of our main results, but where no Riemannian extremal norm exists.

\section{The fiber-bunched setting}\label{s.setting}

In this \lcnamecref{s.setting}, we fix the basic setting for our theorems. Namely, we define and state the basic properties of H\"older vector bundles, intrinsically hyperbolic homeomorphisms, fiber-bunching, holonomies, and irreducibility. Our approach is influenced by \cite{BGV,Viana,KalSad}, and we tried to make it as general as possible. 
We also obtain some new regularity results that are essential for the main theorems of the paper.
However, to make the presentation more fluid, we postpone most proofs to \cref{s.technical}.

\subsection{The H\"older exponent}\label{ss.theta}

From now on, assume that $(X,\dd)$  is a compact metric space.
We also fix $\theta > 0$ such that
the algebra of $\theta$-H\"older functions on $X$ is \emph{normal},
that is, given any two disjoint compact subsets of $X$, there exists a $\theta$-H\"older function that takes values in the interval $[0,1]$ and equals $0$ on one set and $1$ on the other. This assumption is automatically satisfied if $\theta \le 1$. If $X$ is a Cantor set, then the assumption holds for any $\theta>0$. 
Normality implies the existence of $\theta$-H\"older partitions of unity: see e.g.\ \cite[p.~221]{Katz}.

\subsection{H\"older vector bundles}

Let $\E$ be a $d$-dimensional vector bundle over $X$.
We recall the definition and fix the terminology. 
$\E$ is a topological space endowed with a continuous map ${\pi \colon \E \to X}$ (called the \emph{projection}),
a cover of $X$ by open sets $U_i$ (called \emph{coordinate neighborhoods}),
and a family of homeomorphisms (called \emph{coordinate maps})
$$
\psi_i \colon U_i \times \R^d \to \pi^{-1}(U_i) \quad
\text{such that $\pi( \psi_i(x, u) ) = x$ for all $(x,u) \in  U_i \times \R^d$,}
$$
which is required to have the following compatibility property:
whenever $x \in U_i \cap U_j$, the map
$$
g_{j \gets i}(x) \coloneqq \big[ \psi_j(x, \mathord{\cdot})\big]^{-1} \circ \psi_i(x, \mathord{\cdot}) \colon \R^d \to \R^d
$$
is linear.
Therefore we obtain a family of continuous maps:
\begin{equation}\label{e.g}
g_{j \gets i} \colon U_i \cap U_j \to \GL(d,\R),
\end{equation}
which are called \emph{coordinate transformations}.
Moreover, each \emph{fiber} $\E_x \coloneqq \pi^{-1}(x)$ has a unique structure of $d$-dimensional vector space
such that the maps
\begin{equation}\label{e.h_i}
h_i(x) \coloneqq \psi_i(x, \mathord{\cdot}) \colon \R^d \to \E_x
\end{equation}
become isomorphisms.
Since $X$ is assumed to be compact, we will from now on assume that the cover $\{U_i\}$ is finite.

\medskip

We say that $\E$ is a \emph{$\theta$-H\"older vector bundle} if the coordinate transformations \eqref{e.g} are locally $\theta$-H\"older. By compactness, we can reduce the coordinate neighborhoods so that the  coordinate transformations become (uniformly) $\theta$-H\"older.

As mentioned in \cref{ss.intro_extremal}, a \emph{Finsler norm}
is a continuous function $\| \mathord{\cdot} \|$ on $\E$ that restricts to a norm $\| \mathord{\cdot} \|_x$ on each fiber $\E_x$.
A Finsler norm $\| \mathord{\cdot} \|$ is called \emph{Riemannian} if each $\| \mathord{\cdot} \|_x$ is induced by an inner product $\langle \mathord{\cdot}, \mathord{\cdot} \rangle_x$.
A Finsler norm  $\| \mathord{\cdot} \|$  is called \emph{$\theta$-H\"older} if for every $u \in \R^d$ and every coordinate neighborhood, the function $x \in U_i \mapsto \| h_i(x) u \| $ is $\theta$-H\"older.
Every $\theta$-H\"older vector bundle $\E$ admits a $\theta$-H\"older Riemannian norm; the proof is straightforward using a $\theta$-H\"older partition of unity. 

\medskip

We will also need a way of ``transporting'' vectors from one fiber to another:

\begin{proposition}[Transport maps]\label{p.transport}
Let $\E$ be a $\theta$-H\"older vector bundle.
There exists a family of linear maps $I_{y \gets x} \colon \E_x \to \E_y$ with the following properties:
\begin{enumerate}
\item
For every point $x \in X$, the linear map $I_{x \gets x}$ equals the identity.
\item
For every pair of indices $i$, $j$, the matrix-valued map
$$
(x,y)  \in U_i \times U_j \mapsto [h_j(y)]^{-1} \circ I_{y \gets x} \circ h_i(x)
$$
is $\theta$-H\"older.
\end{enumerate}
\end{proposition}

See \cref{ss.basic} for the proof of \cref{p.transport}. The next propositions, also proved in \cref{ss.basic}, give additional quantitative properties of the transport maps that will be useful in subsequent calculations.
Recall that we agree to denote a norm and its induced operator norm by the same symbol, as in \eqref{e.def_operator_norm}.

\begin{proposition}\label{p.transport_groupoid}
Let $\E$ be a $\theta$-H\"older vector bundle, endowed with a Finsler norm.
Let $\{I_{y\gets x}\}$ be the family of transport maps provided by \cref{p.transport}.
Then there is $C>0$ such that for all points $x$, $y$, $z \in X$,
$$
\| I_{y \gets z} \circ I_{z \gets x} - I_{y \gets x} \| \le C \max\{ \dd(x,z)^\theta, \dd(y,z)^\theta \} \, ,
$$
\end{proposition}

\begin{proposition}\label{p.norm_Holder}
Let $\E$ be a $\theta$-H\"older vector bundle, endowed with a Finsler norm $\| \mathord{\cdot}\|$.
Let $\{I_{y\gets x}\}$ be the family of transport maps provided by \cref{p.transport}.
Then the Finsler norm $\|\mathord{\cdot}\|$ is $\theta$-H\"older if and only if
there exists $C>0$ such that for all points $x$, $y \in X$,
$$
\big| \|I_{y \gets x}\| - 1 \big| \le C \dd(x,y) ^\theta \, .
$$
\end{proposition}

\subsection{\texorpdfstring{$\theta$}{theta}-H\"older bundle automorphisms}\label{ss.auto}

Assume that $\E$ is a $\theta$-H\"older vector bundle over the compact metric space~$X$.
Fix a $\theta$-H\"older Riemannian norm on $\E$.

A map $\Phi \colon \E \to \E$ is called an \emph{endomorphism} of $\E$ if
there exists a continuous map $T \colon X \to X$ such that $\pi \circ \Phi = T \circ \pi$
(we say that $\Phi$ \emph{covers} $T$) and for each $x \in X$, the restriction of $\Phi$ to the fiber $\E_x$ is a linear map $\Phi_x$ to the fiber $\E_{Tx}$. If $T$ is a homeomorphism and each $\Phi_x$ is a isomorphism then we say that $\Phi$ is an \emph{automorphism}.

We say that the endomorphism $\Phi$ covering $T$ is \emph{$\theta$-H\"older} if $T$ is Lipschitz and the maps
$$
x \in U_i \cap T^{-1}(U_j) \mapsto [h_j(Tx)]^{-1} \circ \Phi_x \circ h_i(x) \in \GL(d,\R)
$$
are $\theta$-H\"older.\footnote{This is similar to the definition of \emph{$\theta$-bounded vertical shear} in \cite{PSW2}.}
As an immediate consequence, the function $x \in X \mapsto \|\Phi_x\|$ is $\theta$-H\"older.

We can characterize $\theta$-H\"older automorphisms in terms of the transport maps from \cref{p.transport}:
\begin{proposition}\label{p.endo_Holder}
An endomorphism $\Phi \colon \E \to \E$ covering a Lipschitz map $T$ is $\theta$-H\"older if and only if there exists $K>0$ such that for all 
$x$, $y \in X$, we have 
$$
\big\| I_{Ty \gets Tx} \circ \Phi_x - \Phi_y \circ I_{y \gets x} \big\| \le K \dd(x,y)^\theta \, .
$$
\end{proposition}

A proof is provided in \cref{ss.basic}.

\medskip

Next, we want to topologize the set of $\theta$-H\"older automorphisms.

Let $\End^\theta(\E,T)$ denote the vector space of $\theta$-H\"older endomorphisms $\Phi \colon \E \to \E$ covering $T$. 
Define the \emph{$C^0$ norm}:
\begin{equation}\label{e.C0_norm}
\|\Phi\|_0 \coloneqq \sup_{x \in X} \|\Phi_x\| \, .
\end{equation}
The stronger \emph{$\theta$-H\"older norm} makes $\End^\theta(\E,T)$ a Banach space:
\begin{equation}\label{e.Holder_norm}
\|\Phi\|_\theta \coloneqq \max \left\{ \|\Phi\|_0 , \ \sup_{x \neq y} \frac{\|I_{Ty \gets Tx} \circ \Phi_x - \Phi_y \circ I_{y \gets x}\|}{\dd(x,y)^\theta} \right\} \, .
\end{equation}
The set $\Aut^\theta(\E,T)$ of $\theta$-H\"older automorphisms is a $C^0$-open subset of $\End^\theta(\E,T)$.
Given $K\ge 1$, let:
\begin{equation}\label{e.bounded_set}
\Aut^\theta_K(\E,T) \coloneqq \big\{ \Phi \in \Aut^\theta(\E,T) \st \|\Phi\|_\theta \le K, \ \|\Phi^{-1}\|_\theta \le K \big\} \, .
\end{equation}

\subsection{Hyperbolic homeomorphisms}\label{ss.hyp_homeo}

The concept of hyperbolicity in differentiable dynamical systems was introduced by Anosov \cite{Anosov} and Smale \cite{Smale}. Even without recourse to a differentiable structure, it is possible to define hyperbolicity (and to prove interesting theorems); this has been done in various ways: \cite{Bowen,Ruelle,AY,Akin,AH}. In this paper, we will use a minor variation of the definition of hyperbolic homeomorphism given by Sakai \cite{Sakai} (see \cref{r.hyperb} below).

\medskip

Recall that $X$ is a compact metric space.
Let $T \colon X \to X$ be a homeomorphism.
Given $x \in X$ and $\epsilon>0$, we define the following sets:
\begin{itemize}
\item \emph{local unstable set} $W^\uu_\epsilon(x) \coloneqq \big\{ y \in X \st \dd(T^{-n} y, T^{-n} x) \le \epsilon \text{ for all } n\ge 0\big\}$;
\item \emph{local stable set} $W^\ss_\epsilon(x) \coloneqq \big\{ y \in X \st \dd(T^n y, T^n x) \le \epsilon \text{ for all } n\ge 0\big\}$.
\end{itemize}
We say that $T$ is a \emph{hyperbolic homeomorphism} if the following axioms hold: 
\begin{enumerate}

\item\label{i.biLip}
$T$ is bi-Lipschitz, i.e., both $T$ and $T^{-1}$ are Lipschitz;

\item\label{i.lambdas} 
there exist a constant $\epsilon_0 > 0$ and a pair of continuous positive functions $\lambda_\uu$, $\lambda_\ss$ (called the \emph{hyperbolicity exponents}) such that:
\begin{alignat}{4}
\label{e.lambda_u}
x &\in X, &\  x', x'' &\in W^\uu_{\epsilon_0}(x) &\quad &\Rightarrow &\quad
\dd(T^{-1} x', T^{-1} x'') &\le e^{-\lambda_\uu(x)} \, \dd(x',x'') \, , \\
\label{e.lambda_s}
y &\in X, &\  y', y'' &\in W^\ss_{\epsilon_0}(y) &\quad &\Rightarrow &\quad
\dd(T y', T y'') &\le e^{-\lambda_\ss(y)} \, \dd(y', y'') \, ;
\end{alignat}

\item\label{i.bracket} 
there exists a constant $\epsilon_1 \in (0,\epsilon_0)$ such that for any pair of points $x$, $y \in X$ with $\dd(x,y) \le 2\epsilon_1$, the intersection $W^\uu_{\epsilon_0}(x) \cap W^\ss_{\epsilon_0}(y)$ contains exactly one point, denoted by $[x,y]$ and called \emph{the bracket of $x$ and $y$}, which depends continuously on $x$ and $y$;

\item\label{i.bounded_angles}  
there exists a constant $C>0$ such that: 
\begin{equation}\label{e.bounded_angles} 
x, y \in X, \ \dd(x,y) \le 2\epsilon_1 \ \Rightarrow \ 
\max \big\{ \dd([x,y],x) , \dd([x,y],y) \big\} \le C \dd(x,y) \, .
\end{equation}
\end{enumerate}

\begin{remark}\label{r.hyperb}
Sakai \cite{Sakai} uses the terminology \emph{$\mathcal{L}$-hyperbolic homeomorphism}, while Ruelle \cite{Ruelle} uses \emph{Smale spaces}. Modulo a change of metric, both definitions are equivalent to ours, and also to expansivity plus the shadowing property: see \cite{Sakai} and references cited there. 
\end{remark}

Let us also define other sets associated with $T$:
\begin{itemize}
\item \emph{unstable set} $W^\uu(x) \coloneqq \left\{ y \in X \st \dd(T^{-n} y, T^{-n} x) \to 0 \text{ as } n \to +\infty \right\}$;
\item \emph{stable set} $W^\ss(x) \coloneqq \left\{ y \in X \st \dd(T^{n} y, T^{n} x) \to 0 \text{ as } n \to +\infty \right\}$;
\end{itemize}
If $T$ is a hyperbolic homeomorphism then, as an immediate consequence of part~(\ref{i.lambdas}) of the definition, for every $\epsilon \in (0,\epsilon_0]$ we have the following set relations:
\begin{equation}\label{e.longW}
W^\uu(x) = \bigcup_{n\ge 0} T^n(W^\uu_\epsilon(T^{-n} x)) \, , \qquad
W^\ss(x) = \bigcup_{n\ge 0} T^{-n}(W^\ss_\epsilon(T^n x)) \, .
\end{equation}

\medskip

The transverse regularity of the unstable and stable sets is a classical subject, and fine results about hyperbolicity rely on it: see \cite[Chapter~19]{KH}.
Nevertheless, we could not find a reference for the following property for hyperbolic homeomorphisms:

\begin{proposition}\label{p.regularity_base}
Let $T$ be a hyperbolic homeomorphism.  
There exist constants $0 < \kappa_\ss \le 1$ and $C>0$ such that
if $x$, $x'$, $y$, $y' \in X$ satisfy  (see \cref{f.rectangle}):
\begin{equation}\label{e.rectangle}
x' \in W^\uu_{\epsilon_0}(x), \
y' \in W^\uu_{\epsilon_0}(y), \
y  \in W^\ss_{\epsilon_0}(x), \
y' \in W^\uu_{\epsilon_0}(x'),
\end{equation}
then:
$$
\dd(y,y') \le C \, \dd(x, x')^{\kappa_\ss} \, .
$$
\end{proposition}

\begin{figure}[hbt]
	\begin{center}
		\begin{tikzpicture}[scale=.6]
			\draw(-1,0)--(2,0) node[right]{$W^\uu$};
			\draw(-1,5)--(2,5) node[right]{$W^\uu$};
			\draw(0,-1)--(0,6.5) node[midway,left]{$W^\ss$};
			\draw(1,-1)--(1,6.5) node[midway,right]{$W^\ss$};
			\fill(0,0) circle[radius=2pt] node[below left]{$x$};
			\fill(1,0) circle[radius=2pt] node[below right]{$x'$};  
			\fill(0,5) circle[radius=2pt] node[above left]{$y$}; 
			\fill(1,5) circle[radius=2pt] node[above right]{$y'$};  
		\end{tikzpicture}
	\caption{Four points in the configuration \eqref{e.rectangle}.}
	\label{f.rectangle}
	\end{center}
\end{figure}
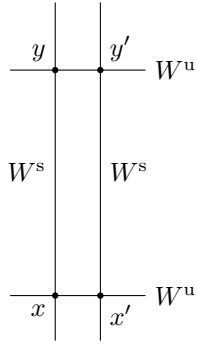

The proof, which includes an estimate for the constant $\kappa_\ss$, is given in \cref{s.technical}. 
If $T$ is the restriction of a $C^2$ diffeomorphism to a basic hyperbolic basic set then a better estimate for $\kappa_\ss$ is given in \cite{SS}. Even better regularity estimates can be obtained under various types of extra assumptions: see \cite{PintoRand} and references therein.

\subsection{Fiber-bunched automorphisms and their holonomies}

We now discuss the notion of \emph{fiber-bunching}. 
It was introduced in a setting very similar to ours by Bonatti, G\'omez-Mont, and Viana \cite{BGV}, though related concepts can be traced back to Brin and Pesin \cite{BrinP} and Hirsch, Pugh, and Shub \cite{HPS}.
Earlier papers \cite{BGV,BV} use a different terminology (``dominated cocycles''), but subsequently the term ``fiber-bunched cocycles'' prevailed: \cite{AV_Portugalia,AV_Inventiones,KalSad}.

\medskip

If $L$ is a linear isomorphism between inner product spaces, we define its \emph{bolicity}\footnote{The term and the notation come from \cite{Pugh,PSW1}. In numerical analysis, the bolicity is called \emph{condition number}.} as
\begin{equation} \label{e.def_bol}
\bol(L) \coloneqq \|L\| \, \|L^{-1}\| \, ,
\end{equation}
which measures the lack of conformality of $L$ (see \cref{p.Lip_bol}).

Let $\E$ be a $\theta$-H\"older $d$-dimensional vector bundle over $X$.
Assume that $T$ is a hyperbolic homeomorphism, and that $\Phi$ is a $\theta$-H\"older automorphism of $\E$ covering $T$.
We say that $\E$ is \emph{fiber-bunched} 
if there exists a Riemannian norm (sometimes called an \emph{adapted norm}) 
such that for all $x \in X$,
\begin{equation}\label{e.def_fiber_bunched}
\log \bol(\Phi_x)
<  \min \left\{ \theta \lambda_\uu(x), \theta \lambda_\ss(x)\right\} \, ,
\end{equation}
where $\lambda_\uu$, $\lambda_\ss$ are the hyperbolicity rates of $T$.
By perturbing the adapted norm if necessary, we can assume it is also $\theta$-H\"older.

Consider the subset of fiber-bunched automorphisms in the space $\Aut^\theta(\E,T)$ of $\theta$-H\"older automorphisms; then this set is open with respect to the $C^0$ norm \eqref{e.C0_norm}, and therefore also open with respect to the stronger $\theta$-H\"older norm \eqref{e.Holder_norm}.

Sometimes we need stronger bunching: we say that it is \emph{$(\eta_\uu,\eta_\ss)$-bunched} for certain constants $\eta_\uu$, $\eta_s \in (0,\theta]$ if, for some adapted norm, and all $x\in X$,
\begin{equation}\label{e.def_strongly_bunched}
\log \bol(\Phi_x)
<  \min \left\{ \eta_\uu \lambda_\uu(x), \eta_\ss \lambda_\ss(x)\right\} \, .
\end{equation}

\begin{remark}\label{r.pointwise_vs_absolute}
We have used the \emph{pointwise} definition of fiber-bunching;
the more stringent notion of \emph{absolute} fiber-bunching requires the same condition
with \emph{constant} hyperbolicity exponents $\lambda_\uu$, $\lambda_\ss$.
Furthermore, our definition of fiber-bunching is \emph{immediate} in the sense that it manifests itself in a single iterate; one can also define a notion of \emph{eventual} fiber-bunching.  
\end{remark}

The most basic and fruitful consequence of fiber-bunching is the existence of certain unstable and stable holonomy maps.
Like the transport maps from \cref{p.transport}, unstable and stable holonomy maps provide a way of linearly transporting vectors from a fiber $\E_x$ to another fiber $\E_y$ (as long as the points $x$, $y$ belong to the same unstable or stable set), but with several extra properties:

\begin{proposition}[Holonomy maps]\label{p.holonomies}
Let	$\Phi \in \Aut^\theta_K(\E,T)$ be a fiber-bunched automorphism.
For each $\star \in \{\uu,\ss\}$, there exist a unique family of linear maps $H^\star_{y \gets x} \colon \E_x \to \E_y$,
defined whenever $y \in W^\star(x)$,
such that the following properties hold:
\begin{enumerate}
\item\label{i.groupoid_1}
$H^\star_{x \gets x} = \id$.
\item\label{i.groupoid_2}
$H^\star_{z \gets y} \circ H^\star_{y \gets x} = H^\star_{z \gets x}$.
\item\label{i.equivariance} 
$\Phi_y \circ H^\star_{y \gets x} = H^\star_{Ty \gets Tx} \circ \Phi_x$.
\item\label{i.holonomy_Holder}
There exists a constant $C>0$ such that:
\begin{equation}\label{e.holonomy_Holder}
y \in W^\star_{\epsilon_0}(x) \quad \Rightarrow \quad
\| H^\star_{y \gets x} - I_{y \gets x} \| \le C \dd(x,y)^\theta \, .
\end{equation}
\item\label{i.holonomy_cont}
The following map is continuous:
\begin{equation}\label{e.joint_holonomy}
\begin{aligned}
\big\{ (u,y) \in \E \times X \st y \in W^\star_{\epsilon_0}(\pi(u)) \big\} &\to \E \\
(u,y) &\mapsto H^\star_{y \gets \pi(u)} (u)
\end{aligned}
\end{equation}
\end{enumerate}
Furthermore, the constant $C$ in \eqref{e.holonomy_Holder} works for all automorphisms in a $C^0$-neighborhood of $\Phi$ in $\Aut^\theta_K(\E,T)$, and the the right-hand side in \eqref{e.joint_holonomy} depends continuously on the automorphism in that neighborhood.
\end{proposition}

The maps $H^\uu$ and $H^\ss$ are called \emph{unstable} and \emph{stable} holonomies, respectively.
Properties (\ref{i.groupoid_1}) and (\ref{i.groupoid_2}) are called \emph{groupoid properties}, and property (\ref{i.equivariance}) is called \emph{equivariance}.

The stable holonomies are actually defined by the following formula:
\begin{equation*} 
H^\ss_{y \gets x} \coloneqq \lim_{n \to + \infty} (\Phi_y^n)^{-1} \circ I_{T^n y \gets T^n x} \circ \Phi_x^n \, ,
\end{equation*}
and unstable holonomies are defined likewise, taking $n \to -\infty$ instead.
The proof of \cref{p.holonomies} consists essentially in proving uniform convergence in these formulas, and it turns out that fiber-bunching is the precise condition for this to work. 
Except for minor adjustments, the argument is the same as in \cite[{\S}1.4]{BGV}, \cite[{\S}4.1]{KalSad}, but for completeness and convenience of the reader we present the proof in \cref{ss.holonomies}.

\begin{remark}
Fiber-bunched automorphisms satisfy a non-commutative version of Walters' condition \cite{Bousch_Walters}, namely:
$$
\forall \epsilon>0 \ \exists \delta> 0 \ \text{such that } \sup_{i \in \ldbrack 0,n \rdbrack} \dd(T^i x, T^i y) < \delta  \ \Rightarrow \ \big\| \Phi_y^n \circ I_{y \gets x} - I_{T^n y \gets T^n x} \circ \Phi_x^n \big\| < \epsilon \, .
$$
Indeed, consider $z \coloneqq [x,y]$ and note the following identity:
$$
\Phi_y^n = H^\ss_{T^n x \gets T^n z} \circ H^\uu_{T^n z \gets T^n x} \circ \Phi^n_x \circ H^\uu_{x \gets z} \circ H^\ss_{z \gets y} \, .
$$
Then, using the continuity of the bracket and the regularity of the holonomies, it is straightforward to obtain the non-commutative Walters' condition.
\end{remark}

We use the holonomies to define certain subsets of $\E$.
For $\epsilon>0$, $u \in \E$, and $x = \pi(u)$, let:
\begin{align*}
\W^\uu_\epsilon(u)	&\coloneqq \big\{ H^\uu_{y \gets x}(u) \st y \in W^\uu_\epsilon(x) \big\} \, ,  \\
\W^\uu(u)			&\coloneqq \big\{ H^\uu_{y \gets x}(u) \st y \in W^\uu(x)      \big\}
= \bigcup_{n\ge 0} \Phi^n(\W^\uu_{\epsilon_0}(\Phi^{-n}(u)))\, ,
\end{align*}
Analogously we define $\W^\ss_\epsilon(u)$ and $\W^\ss(u)$. 
The sets $\W^\uu$ (resp.\ $\W^\ss$) form a $\Phi$-invariant partition of $\E$ and project by $\pi$ onto the sets $W^\uu$ (resp.\ $W^\ss$).

Part (\ref{i.holonomy_Holder}) of \cref{p.holonomies} basically says that the ``leaves'' $\W^\uu$, $\W^\ss$ are H\"older-continuous.
We will need the transverse regularity of the holonomies:

\begin{proposition}\label{p.regularity_above}
Let	$\Phi \in \Aut^\theta_K(\E,T)$ be a fiber-bunched automorphism.
There exist $\theta_\ss \in (0,\theta \kappa_\ss]$ and $C>0$ such that
if $x$, $x'$, $y$, $y' \in X$ satisfy conditions \eqref{e.rectangle} as in \cref{f.rectangle}
then:
$$
\left\| H^\uu_{y' \gets y} \circ H^\ss_{y \gets x} - H^\ss_{y' \gets x'} \circ H^\uu_{x'\gets x} \right\|
\le C \dd(x,x')^{\theta_\ss} \, .
$$
Furthermore, the same constants $\theta_\ss$ and $C$ work for every automorphism in a $C^0$-neighborhood of $\Phi$ in $\Aut^\theta_K(\E,T)$.
\end{proposition}

We were not able to find such a statement in the literature,
so we provide a proof in \cref{ss.regularity}.

\section{Invariant subbundles}\label{s.subbundles}

\subsection{Subbundles and rigidity}\label{ss.rigidity}

Let $E$ be an inner product space of dimension $d$, and let $p \in \ldbrack 1, d \rdbrack$.
We denote by $\Gr_p(E)$ the \emph{$p$-th Grassmannian} of $E$, i.e., the set of all $p$-dimensional subspaces of $E$.
There are many metrics on this set that are ``natural'' in the sense that they are preserved by the action of orthogonal linear maps: see \cite{QZL}.
As shown in \cref{ss.Grass}, we can find one such metric $\dd$ with the useful properties stated in the following \lcnamecrefs{p.Lip_bol}:

\begin{proposition}\label{p.Lip_bol}
If $L \colon E \to F$ is a linear isomorphism between $d$-dimensional inner product spaces and $p<d$,
then the induced map $\Gr_p(E) \to \Gr_p(F)$ is Lipschitz with a constant equal to the bolicity of $L$ \eqref{e.def_bol}. 
\end{proposition}

\begin{proposition}\label{p.close_id}
If $L \colon E \to E$ is a linear isomorphism of a inner product space such that $\|L-\id\| \le \delta$ then the induced map on $\Gr_p(E)$ is $O(\delta)$-close to the identity, provided $\delta$ is sufficiently small.
\end{proposition}

\begin{proposition}\label{p.span_Lip}
The map that associates a $p$-tuple of linearly independent vectors to its span is locally Lipschitz.
\end{proposition}

In particular, the metric $\dd$ induces the usual topology on the Grassmannian.

\medskip

Now consider a $\theta$-H\"older $d$-dimensional vector bundle $\E$ over $X$.
For each $p \in \ldbrack 1 , d-1 \rdbrack$, let $\Gr_p(\E)$ denote the fiber bundle whose fiber over $x\in X$ is $\Gr_p(\E_x)$.
As just explained, the fixed Riemannian norm on $\E$ induces a distance on each fiber of this bundle.

Let $\F$ be a continuous 
$p$-dimensional subbundle of $\E$.
We say that $\F$ is \emph{$\eta$-H\"older}, for some $\eta \in (0,\theta]$, if
there exists $C>0$ such that for all sufficiently close points $x$, $y \in X$ we have:
\begin{equation}\label{e.Holder_subbundle}
\dd \big( \F_y, I_{y \gets x} (\F_x) \big) \le C \, \dd(x,y)^{\eta} \, ,
\end{equation}
where the $\dd$ in the left hand side is the distance in $\cG_p(\E_y)$.
(Recall that $I_{y \gets x}$ is an  isomorphism when $x$ and $y$ are close enough.)

Let $\Phi$ be a fiber-bunched automorphism of $\E$. 
We say that a subbundle $\F \subseteq \E$ is \emph{$\Phi$-invariant} if for all $x \in X$, we have 
$$
\Phi_x (\F_x) = \F_{Tx} \, .
$$
We say that $\F$ is \emph{$H^\uu$-invariant} (or \emph{$\W^\uu$-saturated})
if for all $x \in X$ and all $y \in W^\uu(x)$, we have
$$
H_{y \gets x}^\uu (\F_x) = \F_y \, .
$$
We say that $\F$ is \emph{$\eta$-H\"older along unstable sets}, for some $\eta \in (0,\theta]$, if
there exists $C>0$ and $\epsilon>0$ such that the estimate \eqref{e.Holder_subbundle} holds whenever $y \in W^\uu_{\epsilon}(x)$.
Equivalently, 
there exists $C \ge 0$ such that for all $x \in X$ and all $y \in W^\uu_{\epsilon_0}(x)$, we have:
\begin{equation}\label{e.Holder_section}
\dd \big( \F_y, H^\uu_{y \gets x} (\F_x) \big) \le C \, \dd(x,y)^{\eta} \, ;
\end{equation}
to see the equivalence, use $\theta$-H\"olderness of the holonomy \eqref{e.holonomy_Holder} and \cref{p.close_id}.
$H^\ss$-invariance and $\eta$-H\"olderness along stable sets are defined analogously.

\begin{proposition}[Rigidity] \label{p.rigidity}
Let $\eta \in (0,\theta]$.
Suppose that $\Phi$ is $(\eta,\theta)$-bunched.
Let $\F \subseteq \E$ be a continuous $\Phi$-invariant subbundle.
If $\F$ is $\eta$-H\"older along unstable sets then
$\F$ is $H^\uu$-invariant, and in particular
$\F$ is actually $\theta$-H\"older along unstable sets.
\end{proposition}

\begin{proof}
Since $\Phi$ is $(\eta,\theta)$-bunched, there is a constant $r \in (0,1)$ such that $\bol(\Phi_x) < r e^{\eta \lambda_\uu(x)}$ for every $x \in X$.
Now fix $x \in X$ and $y \in W^\uu_{\epsilon_0}(x)$.
For each $n \ge 0$, let $x_n \coloneqq T^{-n} x$ and $y_n \coloneqq T^{-n} y$.
Then:
\begin{alignat*}{2}
\dd \big (\F_y, \, H^\uu_{y \gets x} (\F_x) \big)
&= \dd \big (\Phi^n_{y_n}(\F_{y_n}), \, \Phi^n_{y_n} (H^\uu_{y_n \gets x_n} (\F_{x_n}) \big) &\quad&\text{(by $\Phi$-invariance of $\F$)}\\
&\le \bol(\Phi^n_{y_n}) \, \dd \big (\F_{y_n}, \, H^\uu_{y_n \gets x_n} (\F_{x_n}) \big) &\quad&\text{(by \cref{p.Lip_bol})}\\
&\le C \, \bol(\Phi^n_{y_n})  \, \dd(x_n,y_n)^{\eta}
&\quad&\text{(by $\eta$-H\"olderness of $\F$).}
\end{alignat*}
On one hand, by submultiplicativity of bolicity, $\bol(\Phi^n_{y_n}) \le \bol(\Phi_{y_1}) \cdots \bol(\Phi_{y_n})$.
On the other hand, using \eqref{e.lambda_u} recursively,
\begin{multline*}
\dd(x_n,y_n) = \dd(T^{-1} x_{n-1}, T^{-1} y_{n-1}) \le e^{-\lambda_\uu(y_{n-1})} \dd(x_{n-1}, y_{n-1}) \le \cdots  \\ \le e^{-\lambda_\uu(y_0) - \cdots -\lambda_\uu(y_{n-1})} \dd(x,y) \, .
\end{multline*}
Combining these estimates, we have:
\begin{align*}
\dd \big (\F_y, \, H^\uu_{y \gets x} (\F_x) \big)
&\le C \, \left[ \prod_{j=1}^{n} \bol(\Phi_{y_j})\right] \left[ \prod_{j=0}^{n-1} e^{-\eta \lambda_\uu(y_j)} \right] \dd(x,y)^{\eta} \\
&\le  C B^2 r^{n-2} \, \dd(x,y)^{\eta} \, ,
\end{align*}
where $B \ge 1$ is the maximal bolicity.
As $n \to \infty$, the right hand side tends to zero.
So $\F_y = H^\uu_{y \gets x} (\F_x)$, proving that the subbundle $\F$ is $H^\uu$-invariant.
Since \eqref{e.Holder_section} holds with $C=0$, the subbundle $\F$ is $\theta$-H\"older along unstable sets.
\end{proof}

\begin{corollary}\label{c.irred}
Let $\Phi$ be a fiber-bunched automorphism of $\E$. 
Let $\F \subseteq \E$ be a $\Phi$-invariant continuous subbundle.
Then the following conditions are equivalent:
\begin{enumerate}
\item\label{i.irred1} 
$\F$ is a $\theta$-H\"older subbundle;
\item\label{i.irred2} 
$\F$ is both $H^\uu$- and $H^\ss$-invariant.
\end{enumerate}
\end{corollary}

\begin{proof}
If condition (\ref{i.irred1}) holds then $\F$ is $\theta$-H\"older along unstable sets, and so \cref{p.rigidity} (with $\eta = \theta$) guarantees that $\F$ is $H^\uu$-invariant. By symmetry, $\F$ is also $H^\ss$-invariant. That is, condition (\ref{i.irred2}) holds.
Conversely, assume that condition (\ref{i.irred2}) holds, and consider a pair $x$, $y$ of nearby points. 
Then the bracket $z \coloneqq [x,y]$ is well-defined, and by property \eqref{e.bounded_angles}, it is $O(\dd(x,y))$-close to either $x$ or $y$.
By hypothesis, $\F_y = H^\ss_{y\gets z} \circ H^\uu_{z \gets x} (\F_x)$.
Using \cref{p.transport_groupoid} and $\theta$-H\"olderness of the holonomies \eqref{e.holonomy_Holder}, we see that $\|I_{y \gets x} - H^\ss_{y\gets z} \circ H^\uu_{z \gets x}\| = O(\dd(x,y)^\theta)$.
It follows from \cref{p.close_id} that  
$\dd ( \F_y, I_{y \gets x} (\F_x)) = O(\dd(x,y)^\theta)$, i.e., condition (\ref{i.irred1}) holds.
\end{proof}

\subsection{Irreducibility}

The \emph{trivial} subbundles of $\E$ are the zero section and $\E$ itself.
A fiber-bunched automorphism $\Phi$ is called \emph{reducible} if it has a nontrivial $\Phi$-invariant subbundle $\F$ satisfying either of the equivalent conditions of \cref{c.irred}, and \emph{irreducible} otherwise.

While the existence of continuous $\Phi$-invariant subbundles is common, the existence of $\theta$-H\"older ones is not. For example, if the automorphism admits a dominated splitting\footnote{See \cref{ss.Mather_dom} for the definition and properties of dominated splittings}, then the subbundles that form the splitting are $\Phi$-invariant, continuous, and actually H\"older, but usually with smaller H\"older exponent. Actually, the dominating bundle is $H^\uu$-invariant and so $\theta$-H\"older along unstable sets, but usually not so well behaved along stable sets.

A precise formulation of the fact that reducibility is uncommon among fiber-bunched automorphisms is provided by \cref{p.irred_typical}.


\subsection{The strong bunching hypothesis}

If $d=2$ then ordinary fiber-bunching suffices for our main results, while if $d\ge 3$ we need $\Phi$ to be not only fiber-bunched, but $(\eta_0,\theta)$-bunched, where $\eta_0$ is given by the following:

\begin{lemma}\label{l.needed_strength}
There exists $\eta_0 \in (0,\theta]$ that depends only on the hyperbolic homeomorphism $T$ (or, more precisely, on its hyperbolicity exponents) and on the H\"older exponent $\theta$
such that if $\Phi$ is a $(\eta_0,\theta)$-bunched automorphism
then the associated regularity exponent $\theta_\ss$ from \cref{p.regularity_above} satisfies:
$$
\theta_\ss \ge \eta_0 \, .
$$
\end{lemma}


For the proof (and an explicit value for $\eta_0$), see \cref{ss.regularity}.

\medskip

Let us say that a $\theta$-H\"older automorphism $\Phi \colon \E \to \E$ covering $T$ is \emph{strongly bunched} if:
\begin{itemize}
	\item the vector bundle $\E$ has fibers of dimension $d \le 2$, and $\Phi$ is fiber-bunched; \textbf{or}
	\item $\Phi$ is a $(\eta_0,\theta)$-bunched automorphism. 
\end{itemize}

The precise point of our proofs where we need strong bunching is for the validity of \cref{t.irr_to_span}, explained in the next \lcnamecref{ss.span}.

\subsection{Spannability} \label{ss.span}

The following concept of \emph{spannability} will play an important role in this paper; it is vaguely similar to the concept of accessibility in partially hyperbolic dynamics (see e.g.\ \cite[{\S}8.1]{Pesin_book}).

Let us say that a fiber-bunched automorphism $\Phi$ is \emph{spannable} if for all $x$, $y \in X$, and all nonzero $u \in \E_x$, there exist:
\begin{itemize}
	\item points $x_1, \dots, x_d \in W^\uu(x)$;
	\item integers $n_1, \dots, n_d \ge 0$ such that the points $y_i \coloneqq T^{n_i} x_i$ all belong to $W^\ss(y)$;
\end{itemize}
with the property that the vectors $v_1, \dots, v_d \in \E_y$ defined by
\begin{equation}\label{e.spanners}
v_i \coloneqq H^\ss_{y \gets y_i} \circ \Phi^{n_i}_{x_i} \circ H^\uu_{x_i \gets x} (u)
\end{equation}
form a basis for $\E_y$.

It is clear that every spannable automorphism is irreducible. 
The following important result provides a converse under extra assumptions:

\begin{theorem}[Sufficient conditions for spannability]\label{t.irr_to_span}
Let $T$ be a transitive hyperbolic homeomorphism.
Let $\Phi$ be a strongly bunched irreducible automorphism covering $T$.
Then $\Phi$ is spannable.
\end{theorem}

In particular (see \cref{p.irred_typical}), typical strongly bunched automorphisms are spannable (provided $T$ is transitive). 

It would be interesting to know whether or not strong bunching is really necessary for the validity of \cref{t.irr_to_span}; see \cref{r.Clark} below for a possible approach to this question.

In order to prove the \lcnamecref{t.irr_to_span}, we need the following easy property of the unstable and stable sets for the base dynamics:

\begin{lemma}\label{l.transitive}
For every $x \in X$, the sets $\bigcup_{n \ge 0} W^\uu(T^n x)$ and $\bigcup_{n \ge 0} W^\ss(T^{-n} x)$ are dense in $X$.
\end{lemma}

\begin{proof}
Let $D$ be the set of points whose forward orbits are dense.
Since $T$ is transitive, $D$ is itself dense.
Moreover, $D$ is $W^\ss$-saturated (i.e., it is a union of stable sets).
By definition of hyperbolic homeomorphism, local stable and unstable sets whose basepoints are sufficiently close always intersect.
It follows that $D$ intersects all unstable sets.
This implies that for every $x \in X$, the set $\bigcup_{n \ge 0} W^\uu(T^n x)$ is dense.
Applying this to $T^{-1}$ we obtain that $\bigcup_{n \ge 0} W^\ss(T^{-n} x)$ is also dense.
\end{proof}

\begin{proof}[Proof of \cref{t.irr_to_span}]
Fix a point $x \in X$ and a nonzero vector $u \in \E_x$.
Let $\Lambda \coloneqq \bigcup_{n \ge 0} W^\uu(T^n x)$, which by \cref{l.transitive} is a dense subset of $X$.
Define the following subsets of the vector bundle $\E$:
$$
\mathbb{U} \coloneqq \bigcup_{n \ge 0} \W^\uu(\Phi^n (u)) \, ,
\qquad
\mathbb{S} \coloneqq \bigcup_{v \in \mathbb{U}} \W^\ss(v) \, ,
\qquad
\F \coloneqq \mathrm{span}(\mathbb{S}) \, ,
$$
where the latter equation means that for each $y \in X$, the fiber $\F_y \coloneqq \E_y \cap \F$ is the vector space spanned by $\mathbb{S}_y \coloneqq \E_y \cap \mathbb{S}$.
In order to prove the \lcnamecref{t.irr_to_span}, we need to show that $\F = \E$.
Clearly,
\begin{itemize}
\item $\mathbb{U}$ projects onto $\Lambda$, and is both forward-$\Phi$-invariant (i.e., $\Phi(\mathbb{U}) \subseteq \mathbb{U}$) and $\W^\uu$-saturated (i.e., it is a union of $\W^\uu$ sets);
\item $\mathbb{S}$ projects onto $X$, and is both forward-$\Phi$-invariant and $\W^\ss$-saturated; therefore $\F$ has the same properties.
\end{itemize}

We claim that the function $y \in X \mapsto \dim \F_y$ has the following properties:
\begin{enumerate}
\item\label{i.dim1}
it is non-decreasing along orbits of $T$ (i.e., $\dim \F_{Ty} \ge \dim \F_y$); 
\item\label{i.dim2}
it is constant along $W^\ss$ sets;
\item\label{i.dim3} 
it is lower semicontinuous. 
\end{enumerate}
Indeed, properties (\ref{i.dim1}) and (\ref{i.dim2}) follow from the facts that $\F$ is forward-invariant and $\W^\ss$-saturated, respectively.
In order to check property (\ref{i.dim3}), fix an arbitrary point $y \in X$ and let $p \coloneqq \dim \F_y$.
Then there exist points $x_1, \dots, x_p \in W^\uu(x)$ and integers $n_1, \dots, n_p \ge 0$ such that the points $y_i \coloneqq T^{n_i} x_i$ all belong to $W^\ss(y)$, and the vectors $v_i$ given by formula \eqref{e.spanners} span $\F_y$.
If $y'$ is sufficiently close to $y$, then for each $i$ we can find $y_i' \in W^\uu(y_i)\cap W^\ss(y')$ such that the holonomies $H^\uu_{y_i' \gets y_i}$ and $H^\ss_{y' \gets y_i'}$ are respectively close to the identity and $H^\ss_{y \gets y_i}$. Then each vector $v_i' \coloneqq H^\ss_{y' \gets y'_i} \circ H^\uu_{y'_i \gets y_i} \circ H^\ss_{y_i \gets y} (v_i)$ is close to $v_i$, and so the span of $\{v_1', \dots, v_p'\}$ has dimension $p$. Since each $v_i'$ belongs to $\F_{y'}$, we conclude that $\dim \F_{y'} \ge p$, therefore proving the semicontinuity property (\ref{i.dim3}).

Let $C$ be the set of points $y \in X$ where $\dim \F_y$ attains its minimum.
By the properties (\ref{i.dim1}), (\ref{i.dim2}), and (\ref{i.dim3}) that we have just proved, 
the set $C$ is nonempty, backwards-invariant (i.e., $T^{-1}(C) \subseteq C$), $W^\ss$-saturated, and closed.
It follows from \cref{l.transitive} that $C=X$.
In other words, $\F$ has constant dimension, say $p$.
So $\F$ is not only forward-$\Phi$-invariant, but $\Phi$-invariant.

\medskip

Let $\theta_\ss$ be given by \cref{p.regularity_above}.
We claim that $\F$ is $\theta_\ss$-H\"older along unstable sets, in the sense defined in \cref{ss.rigidity}.
By compactness, it suffices to prove this claim on a neighborhood of each point $y \in X$.
Take points $x_1, \dots, x_p \in W^\uu(x)$ and integers $n_1, \dots, n_p \ge 0$ such that the points $y_i \coloneqq T^{n_i} x_i$ all belong to $W^\ss(y)$, and the vectors $v_i$ given by formula \eqref{e.spanners} span $\F_y$.
Take $k \ge 0$ large enough so that the points $T^k y_i$ all belong to $W^\ss_{\epsilon_1}(T^k y)$, where $\epsilon_1$ is constant from condition (\ref{i.bracket}) in the definition of hyperbolic homeomorphism.
If we prove that $\F$ is $\theta_\ss$-H\"older along unstable sets on a neighborhood of $T^k y$ then, by invariance, it will follow that $\F$ is $\theta_\ss$-H\"older along unstable sets on a neighborhood of $y$. So let us assume that $k=0$, for simplicity of notation.
Let $y' \in W^\uu_{\epsilon_1}(y)$ be close to $y$.
Then the brackets $y'_i \coloneqq [y_i, y']$ are well-defined; see \cref{f.spanning}.
We need to compare the following two subspaces of~$\E_{y'}$:
\begin{align*}
\F_{y'} &= \mathrm{span} \big\{ H^\ss_{y' \gets y'_i} \circ \underbrace{H^\uu_{y'_i \gets y_i} \circ H^\ss_{y_i \gets y}}_{\circled{1}} (v_i) \big\}_{i=1}^p \, , \\
H^\uu_{y' \gets y} (\F_y) &= \mathrm{span} \big\{H^\ss_{y' \gets y'_i} \circ \underbrace{H^\ss_{y'_i \gets y'} \circ H^\uu_{y' \gets y}}_{\circled{2}} (v_i) \big\}_{i=1}^p \, .
\end{align*}
By \cref{p.regularity_above}, $\|\circled{1} - \circled{2}\| = O(\dd(y,y')^{\theta_\ss})$;
moreover $\| H^\ss_{y' \gets y'_i}\| = O(1)$.
So, by \cref{p.span_Lip}, we conclude that $\dd \big( \F_{y'} , H^\uu_{y' \gets y} (\F_y) \big) = O(\dd(y,y')^{\theta_\ss})$.
This concludes the proof that $\F$ is $\theta_\ss$-H\"older along unstable sets.
A fortiori, $\F$ is continuous (since it is invariant under stable holonomies).

\begin{figure}[hbt]
	\begin{center}
		\begin{tikzpicture}[scale=.6]
			\draw(-7,0) node[left]{$W^\uu$}--(-3,0);
			\fill(-6,0) circle[radius=2pt] node[below]{$x_1$};
			\fill(-5,0) circle[radius=2pt] node[below]{$x$};
			\fill(-4,0) circle[radius=2pt] node[below]{$x_2$}; 		
			\draw[->](-6,.5) arc(180:90:4.5) node[midway, left]{$T^{n_1}$};	
			\draw[->](-4,.5) arc(180:90:2.5) node[midway, left]{$T^{n_2}$};	
			\draw(-1,0)--(2.5,0) node[right]{$W^\uu$};
			\draw(-1,3)--(2.5,3);
			\draw(-1,5)--(2.5,5);
			\draw(0,-1)--(0,6.5) node[above]{$W^\ss$};
			\draw(1,-1)--(1,6.5);
			\fill(0,0) circle[radius=2pt] node[below left]{$y$};
			\fill(1,0) circle[radius=2pt] node[below right]{$y'$};
			\fill(0,3) circle[radius=2pt] node[above left]{$y_2$};
			\fill(1,3) circle[radius=2pt] node[above right]{$y_2'$};
			\fill(0,5) circle[radius=2pt] node[above left]{$y_1$};
			\fill(1,5) circle[radius=2pt] node[above right]{$y_1'$};
		\end{tikzpicture}
	\caption{Proof of \cref{t.irr_to_span}}\label{f.spanning}
	\end{center}
\end{figure}

\medskip

The proof ends differently according to the dimension of $\E$.
If $d=1$ then $\F = \E$ and we are done.

Next, consider the case $d=2$.
Assume for a contradiction that $\F \neq \E$, i.e., that $\F$ has $1$-dimensional fibers.
For each $y \in \Lambda$, the set $\mathbb{U}_y$ contains a nonzero vector and therefore spans $\F_y$.
Since $\Lambda$ is dense in $X$ and $\F$ is continuous, we conclude that $\F$ is the closure of $\mathrm{span}(\mathbb{U})$.
In particular, $\F$ is $\W^\uu$-saturated.
Recalling that  $\F$ is also $\W^\ss$-saturated, we contradict irreducibility.
This concludes the proof in the case $d=2$.

\medskip

Now consider the case $d\ge 3$.
Then, by definition of strong bunching, $\Phi$ is $(\eta_0,\theta)$-bunched.
Recall from \cref{l.needed_strength} that $\eta_0 \le \theta_\ss$.
So \cref{p.rigidity} (rigidity) applies and the regularity of the subbundle is upgraded: it is actually $\theta$-H\"older along unstable sets.
Irreducibility implies that $\F = \E$, thus concluding the proof.
\end{proof}

We will use an apparently stronger, but equivalent form of spannability.
Recall that $\epsilon_1>0$ is one of the constants that appear in the definition of hyperbolic homeomorphism (\cref{ss.hyp_homeo}).

\begin{proposition}[Uniform spannability]\label{p.unif_span}
Suppose $\Phi$ is a spannable automorphism.
Then there exist constants $\bar{n} \ge 0$ and $C_0 > 0$ with the following properties:
For all points $x$, $y \in X$, and all unit vectors $u \in \E_x$, there exist:
\begin{itemize}
	\item points $x_1, \dots, x_d \in W^\uu_{\epsilon_1}(x)$;
	\item integers $n_1, \dots, n_d \in \ldbrack 0, \bar{n} \rdbrack$ such that the points $y_i \coloneqq T^{n_i} x_i$ all belong to $W^\ss_{\epsilon_1}(y)$;
\end{itemize}
with the property that the vectors $v_1, \dots, v_d \in \E_y$ defined by \eqref{e.spanners} form a basis for $\E_y$;
moreover, if $L \colon \E_y \to \E_y$ is a linear map that sends this basis to an orthonormal basis then $\|L\| < C_0$.
\end{proposition}

\begin{proof}
If $u \in \E$ is a nonzero vector, let $[u]$ denote its class in the projective bundle $\hat{\E} \coloneqq  \Gr_1(\E)$.
Let $\Phi$ be a spannable automorphism.
Given $([u], y) \in \hat{\E} \times X$, consider $x = \pi(u)$ and let $x_i$, $n_i$, $y_i$, and $v_i$, where $i \in \ldbrack 1,d \rdbrack$, be as in the definition of spannability. 
Note that if $([\tilde u], \tilde y)$ belongs to a sufficiently small neighborhood of $([u],y)$ then we can find the corresponding data $(\tilde x_i, \tilde n_i, \tilde y_i, [\tilde v_i])$ close to $(x_i, n_i, y_i, [v_i])$ (so $\tilde n_i = n_i$) and actually depending continuously on $([\tilde u], \tilde y)$.
Since the space $\hat{\E} \times X$ is compact, we can cover it by finitely many such neighborhoods $U_j$.
We can also assume that the sets $U_j$ are compact.

Fix any set $U_j$ and an element $([u],y) \in U_j$. Let  $x = \pi(u)$ and let $(x_i, n_i, y_i, [v_i])$, $i \in \ldbrack 1,d \rdbrack$ be the corresponding spannability data. 
For each $k \ge 0$, the pair $([\Phi^{-k}(u)], T^k y) \in \hat{\E} \times X$ has $(T^{-k} x_i, n_i+2k, T^k y_i, [\Phi^k(v_i)])$, $i \in \ldbrack 1,d \rdbrack$ as valid spannability data. 
By \eqref{e.longW}, if $k$ is large enough then 
$$
T^{-k} x_i \in W^\uu_{\epsilon_1}(T^{-k} x) \quad\text{and}\quad T^k y_i \in W^\ss_{\epsilon_1}(T^k y) \quad \text{for each $i \in \ldbrack 1,d \rdbrack$.}
$$
By continuity of the spannability data on the compact set $U_j$, this conclusion holds provided $k$ is bigger than some $k_j$.
There are finitely many indices $j$ to consider, so let us fix a definitive $k$ bigger than all $k_j$'s.
The compact sets $V_j \coloneqq \{([\Phi^{-k}(u)], T^k y) \st ([u],y) \in U_j\}$ also cover the space $\hat{\E} \times X$.
They provide the spannability data with the required uniformity properties.
This proves the \lcnamecref{p.unif_span}.
\end{proof}

\begin{corollary}\label{c.open_span}
Given a spannable automorphism $\Phi \in \Aut^\theta_K(\E,T)$, we can choose $\bar{n} \ge 0$ and $C_0 > 0$ such that the the statement of \cref{p.unif_span} holds for all automorphisms in a $C^0$-neighborhood of $\Phi$ in the space $\Aut^\theta_K(\E,T)$.

In particular, spannable automorphisms form a $C^0$-open subset of $\Aut^\theta_K(\E,T)$.
\end{corollary}

\begin{proof}
By part~(\ref{i.holonomy_cont}) of \cref{p.holonomies}, holonomies depend continuously on the fiber-bunched automorphism $\Phi$, with respect to the $C^0$-norm.
So, in the situation of \cref{p.unif_span}, if we make a $C^0$-perturbation of $\Phi$ (among $\theta$-H\"older automorphisms) then the vectors $v_1$, \dots, $v_d$  change little and therefore stay linearly independent.
\end{proof}

\begin{remark}
Let us say that a automorphism is \emph{topologically irreducible} if it admits no \emph{continuous} proper invariant subbundle. As the proof of \cref{t.irr_to_span} shows, if a fiber-bunched automorphism over a transitive hyperbolic homeomorphism is topologically irreducible then it is spannable.
\end{remark}

\begin{remark}\label{r.Clark}
As Clark Butler has pointed out to us, if a fiber-bunched automorphism satisfies the \emph{pinching-and-twisting} condition from \cite[Def.~1.3]{BV},  \cite[Def.~1.2]{AV_Portugalia} then it is spannable.
In other words, one can remove the strong bunching hypothesis from \cref{t.irr_to_span}, provided one replaces irreducibility with the (strictly stronger) pinching-and-twisting condition.

Let us sketch the proof.
Let $\mathbb{U} \subseteq \F$ be as in proof of \cref{t.irr_to_span}.
Let $\mathbb{V}$ be the closure of the span of $\bigcup_{n\ge 0} \Phi^{-n}(\mathbb{U})$; then $\mathbb{V}$ is $\Phi$-invariant, $\W^\uu$-saturated,  projects down on $X$, and is contained in $\F$.
Let $\mu$ be the $T$-invariant probability measure on $X$ with maximal entropy (other choices are possible). 
Let $\P\Phi$ be the projectivization of the automorphism $\Phi$.
Then $\P\Phi$ admits an \emph{invariant $\mathrm{u}$-state}, that is, an invariant measure $\hat m$ that projects on $\mu$ and whose disintegration w.r.t.\ to this projection is $\mu$-a.e.\  
invariant under unstable holonomies: see \cite[\S~4.1]{AV_Portugalia}.
By adapting the construction, we can ensure that the invariant $\mathrm{u}$-state $\hat m$ gives full weight to $\P \mathbb{V}$, and in particular to $\P\F$, which is a continuous $\W^\ss$-saturated invariant subbundle of $\P\E$.
Under the pinching-and-twisting assumption, \cite[Prop.~5.1]{BV} or  \cite[Prop.~5.1]{AV_Portugalia} say that such a situation is impossible unless $\F = \E$. (Actually in these papers $T$ is a shift, but the proofs can be adapted to the general situation, or we can use a Markov partition.) Therefore $\Phi$ is spannable.

It is not clear how to relax the pinching-and-twisting hypothesis in the arguments from \cite{BV,AV_Portugalia}. 
Therefore we still lack an optimal criterion for spannability.
\end{remark}

\section{Bounding the growth}\label{s.bounded}

\subsection{Relative product boundedness}\label{ss.RPB}

A vector bundle automorphism $\Phi$ is called \emph{product bounded} if
$$
\sup_{n \ge 0} \, \sup_{x\in X} \|\Phi^n_x\| < \infty
$$
for some and hence any Finsler norm on $\E$.
This condition evidently implies that $\beta(\Phi)\le 0$, i.e.\ the maximal Lyapunov exponent \eqref{e.def_beta} is nonpositive.
On the other hand, we say that $\Phi$ is \emph{relatively product bounded} if $e^{-\beta(\Phi)} \Phi$ is product bounded, that is,
$$
\sup_{n \ge 0} e^{-\beta(\Phi) n} \sup_{x\in X} \|(\Phi^n)_x\| < \infty \, .
$$
Of course, if $\Phi$ has an extremal norm then it is relatively product bounded.
The converse is true in the $1$-step case, as noted by Rota and Strang \cite{RS}.
But the converse is not true in general\footnote{The naive attempt of defining an extremal norm by $\tribar{u} \coloneqq \sup_{n\ge 0} e^{-\beta(\Phi) n}\|\Phi^n(u)\|$ does not necessarily work because continuity may fail.}; in fact it may fail even in dimension $1$, as shown by Morris \cite[Proposition 2]{Morris_sufficient}.
In Morris' example, the dynamics is uniformly hyperbolic (actually a full shift), but the function is not H\"older.

In this paper, we need to prove relative product boundedness as an essential preliminary step in the construction of extremal norms.
We will show the following:

\begin{proposition}\label{p.bounded}
Every spannable automorphism is relatively product bounded.
\end{proposition}

The proof, which will occupy \cref{ss.local_RPB,ss.bounded_proof},
is roughly as follows:
first, we find pieces of $\W^\uu$ sets of uniform size that stay relatively product bounded for a long time (\cref{l.anti_claim}), then we use a compactness argument to find small pieces of $\W^\uu$ sets that stay relatively product bounded forever (\cref{l.input}), and finally we use spannability to spread this property to the whole bundle (\cref{l.key}).

\medskip

Let us close this \lcnamecref{ss.RPB} with some remarks about product boundedness and relative product boundedness.

It was shown by Blondel and Tsitsiklis \cite{BT} that the product boundedness of a pair of rational matrices is algorithmically undecidable.   

A result of Coronel, Navas, and Ponce \cite{CNP} states if $T$ is a \emph{minimal} homeomorphism (i.e.\ all its orbits are dense) and  $\Phi$ and $\Phi^{-1}$ are both product bounded then there exists an \emph{invariant} Riemannian norm.

It is easy to give examples of regular (e.g.\ H\"older) automorphisms that are not relatively product bounded: any cocycle constant equal to $\left( \begin{smallmatrix} 1 & 1 \\ 0 & 1 \end{smallmatrix} \right)$ will do. 
Here is a more interesting example: 

\begin{example}\label{ex.Herman}
Let the base dynamics $T$ be an irrational rotation 
of the circle $\R/\Z$, and consider the matrix-valued map $F(x) \coloneqq \left( \begin{smallmatrix} 2 & 0 \\ 0 & 1/2 \end{smallmatrix} \right) R_{2\pi x}$, where $R_\theta$ denotes the rotation matrix by angle~$\theta$. As shown by Herman \cite[p.~471--473]{Herman}, the $\SL(2,\R)$-cocycle $(T,F)$ has a positive Lyapunov exponent, but it is not uniformly hyperbolic. Therefore it cannot be relatively product bounded, because otherwise it would contradict a result of Morris \cite[Theorem~2.1]{Morris_rapidly}.
\end{example}

\subsection{Existence of local unstable sets with relatively bounded orbits}\label{ss.local_RPB}

Let $\Phi$ be a fiber-bunched automorphism in the set $\Aut^\theta_K(\E,T)$.
By the definition \eqref{e.bounded_set} of this set, 
\begin{equation}\label{e.bound_Phi}
\|\Phi^{\pm 1}_x\| \le K  \quad \text{for all } x \in X.
\end{equation}
By equicontinuity of local holonomies, there exists a constant $C_1 > 1$ such that:
\begin{equation}\label{e.bound_H}
\|H^\star_{y \gets x}\| < C_1 
\quad \text{for all } x \in X, \ \star \in \{\uu,\ss\}, \  y \in W^\star_{\epsilon_0}(x).
\end{equation}
Moreover, by \cref{p.holonomies}, it is possible to choose a constant $C_1$ that works for all automorphisms in a $C^0$-neighborhood of $\Phi$ in $\Aut^\theta_K(\E,T)$.

Let $\E^\times$ denote the complement of the zero section in $\E$.
Recall that $\epsilon_0$ comes from the definition of hyperbolic homeomorphism.

\begin{lemma}\label{l.anti_claim}
Let $\Phi$ 
be a fiber-bunched automorphism.
Then there exists $\epsilon_2 \in (0,\epsilon_0)$, depending only on $T$, such that for every integer $m > 0$ there exists $u \in \E^\times$ with the following property:
$$
\sup_{n \in \ldbrack 1, m \rdbrack} \, \sup_{v \in \W^\uu_{\epsilon_2} (u)}  e^{-\beta(\Phi) n} \| \Phi^n (v) \| \le 2  \|u\| \, .
$$
\end{lemma}

\begin{proof}
Multiplying $\Phi$ by a nonzero constant, we can assume that $\beta(\Phi) = 0$.
Let $\lambda_\uu$ be the hyperbolicity exponent of $T$ along unstable sets, and let
\begin{equation}\label{e.hyp_rate}
a \coloneqq \sup_{x \in X} e^{-\lambda_\uu(x)}  < 1 \, .
\end{equation}
Hyperbolicity implies:
\begin{equation}\label{e.W_contraction}
\forall x \in X , \ \forall \epsilon \in (0, \epsilon_0], \quad
T^{-1}(W^\uu_\epsilon(x)) \subseteq W^\uu_{a\epsilon}(T^{-1} x) \, .
\end{equation}
Let $\epsilon_2 \coloneqq (1-a)\epsilon_0$.
In order to show that the conclusion of the \lcnamecref{l.anti_claim} holds for this $\epsilon_2$,
let us assume for a contradiction that there exists an integer $m > 0$
such that:
\begin{equation}\label{e.old_claim}
\forall u \in \E^\times \
\exists n=n(u) \in \ldbrack 1, m \rdbrack \
\exists v=v(u) \in \W^\uu_{\epsilon_2} (u) \text{ s.t.\ }
\| \Phi^n (v) \| > 2 \|u\| \, .
\end{equation}

We recursively define sequences $(u_k)$, $(v_k)$ in $\E^\times$ and $(n_k)$ in $\ldbrack 1, m \rdbrack$ as follows:
We choose $u_0 \in \E^\times$ arbitrarily.
Assuming $u_k$ was already defined, we let $n_k \coloneqq n(u_k)$ and $v_k \coloneqq v(u_k)$ be given by \eqref{e.old_claim}, and let $u_{k+1} \coloneqq \Phi^{n_k}(v_k)$.
Note that for each $k\ge 0$ we have
$\|u_{k+1}\| > 2 \|u_k\|$ and so $\| u_k \| \ge 2^k \|u_0\|$.

Now let $\ell_k \coloneqq n_0 + n_1 + \cdots + n_{k-1}$ (so $\ell_0 \coloneqq 0$), and $w_k \coloneqq \Phi^{-\ell_k}(v_k)$.
We claim that each $w_k$ belongs to $\W^\uu_{\epsilon_0} (u_0)$;
indeed:
\begin{alignat*}{2}
v_k &\in \W^\uu_{\epsilon_2}(u_k) &&\Rightarrow \\
\Phi^{-n_{k-1}}(v_k) &\in \W^\uu_{a \epsilon_2}(v_{k-1}) \subseteq \W^\uu_{(1+a) \epsilon_2}(u_{k-1}) &&\Rightarrow \\
\Phi^{-n_{k-2}-n_{k-1}}(v_k) &\in \W^\uu_{(a+a^2) \epsilon_2}(v_{k-2}) \subseteq \W^\uu_{(1 + a+a^2) \epsilon_2}(u_{k-2}) &&\Rightarrow \\
&\ \vdots \\
w_k = \Phi^{-n_{0} - \cdots -n_{k-1}}(v_k) &\in \W^\uu_{(a+\cdots+a^k) \epsilon_2}(v_0) \subseteq \W^\uu_{(1 + a + \cdots + a^k) \epsilon_2}(u_0) \, ,
\end{alignat*}
proving the claim.
In particular, by \eqref{e.bound_H} we obtain $\|w_k\| \le C_1 \|u_0\|$. 
Since $v_k \in \W^\uu_{\epsilon_2}(u_k)$, using \eqref{e.bound_H} again we have $\|u_k\| \le C_1 \|v_k\|$. 
Therefore:
$$
\frac{\|\Phi^{\ell_k} (w_k)\|}{\|w_k\|} 
= \frac{\|v_k\|}{\|w_k\|} 
\ge \frac{C_1^{-1} \|u_k\|}{C_1 \|u_0\|}
\ge C_1^{-2} 2^k 
\, .
$$
Since $\ell_k \le mk$, using \eqref{e.beta_other} we obtain 
$$
\beta(\Phi) = \lim_{k \to \infty} \sup_{x\in X} \frac{ \log \|\Phi^{\ell_k}_x \|}{\ell_k}
\ge \limsup_{k \to \infty} \frac{ \log (\|\Phi^{\ell_k} w_k \|/\|w_k\|)}{\ell_k} \ge \frac{\log 2}{m} > 0 \, .
$$
This contradiction concludes the proof.
\end{proof}

The next \lcnamecref{l.input} supersedes the previous one:

\begin{lemma}\label{l.input}
Let $\Phi \in \Aut^\theta_K(\E,T)$ be a fiber-bunched automorphism. 
Then there exist a constant $C_2>1$ and a vector $u_* \in \E^\times$ such that
$$
\sup_{n \ge 0} \, \sup_{v \in \W^\uu_{\epsilon_1}(u_*)} e^{-\beta(\Phi) n} \| \Phi^n (v) \| \le C_2 \|u_*\|  \, .
$$
Moreover, the same constant $C_2$ works for all automorphisms in a $C^0$-neighborhood of $\Phi$ in $\Aut^\theta_K(\E,T)$.
\end{lemma}

\begin{proof}
Again, multiplying $\Phi$ by a nonzero constant (and increasing $K$ if necessary), we can assume that $\beta(\Phi) = 0$.

Let $\epsilon_2$ be given by \cref{l.anti_claim}.
By the continuity of the bracket, there exists  $\epsilon_3 \in (0,\epsilon_1)$ such that:
$$
z_1,z_2 \in X, \ \dd(z_1,z_2) < 2\epsilon_3 \quad \Rightarrow \quad \dd([z_1,z_2], z_i) \le \epsilon_2 \, .
$$

For each integer $m \ge 1$, \cref{l.input} provides $u_m \in \E^\times$, say with $\|u_m\|=1$, such that for every $n \in \ldbrack 1, m \rdbrack$ and every $v \in \W^\uu_{\epsilon_2}(u_m)$
we have $\| \Phi^n(v) \| \le 2$.
Passing to a subsequence if necessary, we assume that $(u_m)$ converges to some $\bar{u}$, which has $\|\bar{u}\|=1$.
Let $x_m \coloneqq \pi(u_m)$ and $\bar{x} \coloneqq \pi(\bar{u})$.

We claim that
\begin{equation}\label{e.almost_there}
\sup_{n \ge 0} \, \sup_{\bar{v} \in \W^\uu_{\epsilon_3}(\bar{u})} \| \Phi^n (\bar{v}) \|  \le 2C_1 \, .
\end{equation}
Indeed, given $\bar{v} \in \W^\uu_{\epsilon_3}(\bar{u})$ and $n \ge 0$,
consider $\bar{y} \coloneqq \pi(\bar{v})$.
Since $x_m \to \bar{x}$, for every sufficiently large $m \ge n$ we have $\dd(x_m,\bar{y}) < 2\epsilon_3$, and in particular $y_m \coloneqq [x_m,\bar{y}]$ is well-defined and belongs to $W^\uu_{\epsilon_2}(x_m)$.
Let $v_m \coloneqq H^\uu_{y_m \gets x_m}(u_m)$ and $w_m \coloneqq H^\ss_{\bar{y} \gets y_m}(v_m)$
(see \cref{f.input}).

\begin{figure}[hbt]
	\begin{center}
		\begin{tikzpicture}[scale=.8]
			\draw(0,0)--(8,0) node[right]{$W^\uu$};
			\fill(2,0) circle[radius=2pt] node[above left]{$x_m$};
			\fill(6,0) circle[radius=2pt] node[above left]{$y_m$};
			\draw plot [smooth, tension=1] coordinates { (0,4.3) (2,4) (4,3.8) (6,4) (8,3.65)} node[right]{$\W^\uu$};
			\draw[dashed] (2,0)--(2,4);
			\fill(2,4) circle[radius=2pt] node[below left]{$u_m$};
			\fill(6,4) circle[radius=2pt] node[above left]{$v_m$};
			\draw[dashed] (6,0)--(6,4);

			\draw(1,1)--(9,1) node[right]{$W^\uu$};
			\fill(3.25,1) circle[radius=2pt] node[below right]{$\bar{x}$};
			\fill(7,1) circle[radius=2pt] node[below right]{$\bar y$};
			\draw plot [smooth, tension=1] coordinates { (1,5.4) (3.25,5) (5,4.8) (7,4.9) (9,4.75)} node[right]{$\W^\uu$};
			\draw[dashed] (3.25,1)--(3.25,5);
			\fill(3.25,5) circle[radius=2pt] node[below left]{$\bar{u}$};

			\draw(1.5,-.5)--(3.5,1.5);
			\draw(1.75,-.5)--(3.75,1.5) node[right]{$W^\ss$};
			\draw(5.5,-.5)--(7.5,1.5) node[right]{$W^\ss$};
			\draw plot [smooth, tension=1] coordinates { (5.5,3.4) (6,4) (7,5.5) (7.5,6.1)} node[right]{$\W^\ss$};

			\fill(7,4.9) circle[radius=2pt] node[below right]{$\bar{v}$};
			\fill(7,5.5) circle[radius=2pt] node[right]{$w_m$};
			\draw[dashed] (7,1)--(7,5.5);
		\end{tikzpicture}
	\caption{Proof of \cref{l.input}.}\label{f.input}
	\end{center}
\end{figure}

Then:
$$
\|\Phi^n(w_m) \|
=   \left\| H^\ss_{T^m \bar{y} \gets T^m y_m} (\Phi^n(v_m)) \right\|
\le C_1 \left\| \Phi^n(v_m) \right\|
\le 2C_1 \, .
$$
As $m \to \infty$ (recall that $n$ is fixed), we have $y_m \to [\bar{x},\bar{y}] = \bar{y}$
and so:
$$
w_m = H^\ss_{\bar{y} \gets y_m} \circ H^\uu_{y_m \gets x_m}(u_m) \to H^\uu_{\bar{y} \gets \bar{x}}(\bar{u}) = \bar{v} \, ,
$$
by continuity of holonomies.
It follows that $\|\Phi^n(\bar{v}) \| \le 2 C_1$, completing the proof of the claim \eqref{e.almost_there}.

Fix a constant $\ell > 0$ depending only on $T$ such that $W^\uu_{\epsilon_1}(T^\ell \bar{x}) \subseteq T^\ell(W^\uu_{\epsilon_3}(\bar{x}))$.
Let $u_* \coloneqq \Phi^\ell(\bar{u})$.
Then
$$
\sup_{n \ge 0} \, \sup_{v \in \W^\uu_{\epsilon_1}(u_*)} \| \Phi^n (v) \| \le 
\sup_{n \ge 0} \, \sup_{\bar{v} \in \W^\uu_{\epsilon_3}(\bar{u})} \| \Phi^{n+\ell} (\bar{v}) \|  \le 2C_1 \, , 
$$
by \eqref{e.almost_there}.
On the other hand, recalling \eqref{e.bound_Phi}, we have $\|u_*\| \ge K^{-\ell} \|u\| = K^{-\ell}$.
So the vector $u_*$ has the desired property with $C_2 \coloneqq 2 K^\ell C_1$, completing the proof of the \lcnamecref{l.input}.
\end{proof}

\subsection{Proof of relative product boundedness}\label{ss.bounded_proof}

The next lemma uses spannability to spread local product boundedness from a local unstable set to the whole space:

\begin{lemma}\label{l.key}
Let $\Phi \in \Aut^\theta_K(\E,T)$ be a spannable automorphism.
There exists $C_3 > 1$ with the following properties. 
Suppose $u \in \E$ is a nonzero vector such the following quantity is finite:
$$
r \coloneqq \frac{1}{\|u\|}
\limsup_{n \to \infty} e^{-n \beta(\Phi)}\sup_{v \in \W^\uu_{\epsilon_1}(u)} \|\Phi^n(v)\| \, .
$$
Then 
$$
\limsup_{n \to \infty} e^{-n \beta(\Phi)} \sup_{y \in X} \|\Phi^n_y \| \le C_3 r \, .
$$
Furthermore, the same constant $C_3$ works for all automorphisms in a $C^0$-neighborhood of $\Phi$ in $\Aut^\theta_K(\E,T)$.
\end{lemma}

Note that the hypothesis of \cref{l.key} is non-void by \cref{l.input}, and that its conclusion implies that $\Phi$ is relatively product bounded. So \cref{l.key} implies \cref{p.bounded}. However, the more technical statement of \cref{l.key} is necessary for the construction of an extremal norm in the next section.

\begin{proof}[Proof of \cref{l.key}]
It is sufficient to consider $\beta(\Phi)=0$.
Let $\bar{n}$ and $C_0$ be the uniform spannability constants provided by \cref{p.unif_span}.
Fix a nonzero vector $u$ for which the associated quantity $r$ is finite, and without loss of generality, let us assume that $\|u\| = 1$.
Let $r'>r$ be arbitrary.
Then there exists $n_*$ such that 
$$
\sup_{n \ge n_*} \sup_{v \in \W^\uu_{\epsilon_1}(u)} \|\Phi^n(v)\| \le r'  \, .
$$
Consider arbitrary $y \in X$ and $w \in \E_y$.
Apply \cref{p.unif_span} to the points $x \coloneqq \pi(u)$ and $y$ and the vector $u$, obtaining
points $x_1, \dots, x_d \in W^\uu_{\epsilon_1}(x)$ and times $n_1, \dots, n_d \in \ldbrack 0, \bar{n} \rdbrack$ such that each point $y_i \coloneqq T^{n_i} x_i$ belongs to $W^\ss_{\epsilon_1}(y)$ and the vectors $v_i$ defined by \eqref{e.spanners} form a basis for $\E_y$.
Moreover, if we express $w$ as a linear combination $a_1 v_1 + \dots + a_d v_d$, then the \lcnamecref{p.unif_span} also yields that $(\sum_i a_i^2)^{1/2} \le C_0 \|w\|$. So each $|a_i| \le C_0\|w\|$.
For each $i$ and $n \ge n_*$,
we have
$$
\Phi^n(v_i) = \underbrace{H^\ss_{T^n y \gets T^n y_i}}_{\circled{1}} \circ \underbrace{\Phi^{n_i+n}_{x_i} \circ H^\uu_{x_i \gets x} (u)}_{\circled{2}} \, .
$$
We have $\| \circled{1} \| \le C_1$ by \eqref{e.bound_H}, and $\| \circled{2} \| \le r'$ by definition. 
Combining these estimates, we obtain:
$$
\|\Phi^n(w) \| \le d C_0 C_1 r' \|w\| \, ,
$$
that is, $\|\Phi^n_y\| \le C_3 r'$, where $C_3 \coloneqq d C_0 C_1$.
So $\limsup_{n \to \infty} \sup_{y \in X} \|\Phi^n_y \|$ is bounded by $C_3 r'$,
and actually by $C_3 r$, since $r'>r$ is arbitrary.
This proves the desired inequality.

Now consider a $C^0$-perturbation of $\Phi$ in the set $\Aut^\theta_K(\E,T)$. By \cref{c.open_span}, this perturbation is also spannable, and we can use the same constants $\bar{n}$ and $C_0$.
So the argument above applies verbatim for the perturbed automorphism.
\end{proof}

\subsection{Application: polynomial bounds}\label{ss.polynomial}

Let us give an application of what we have proved so far, namely that under the hypothesis of strong fiber-bunching, relative product boundedness fails at most by a polynomial factor.
The reader anxious to see extremal norms may skip this \lcnamecref{ss.polynomial}.

\begin{theorem}\label{t.polynomial}
Let $T$ be a transitive hyperbolic homeomorphism.
Let $\Phi \colon \E \to \E$ be strongly bunched automorphism covering $T$.
Then there exists an integer $d' \in \ldbrack 0, d-1 \rdbrack$ and $C>0$ such that 
$$
\|\Phi_x^n\| \le C n^{d'} e^{n \beta(\Phi)} \quad \text{for all $x\in X$ and $n \ge 0$.}
$$ 
\end{theorem}

For related results, see \cite[Theorem 3.10]{KalSad}, \cite[\S3.5--3.6]{Jungers}.

Before proving this \lcnamecref{t.polynomial}, let us fix some terminology.
Suppose $\E$ is a $\theta$-H\"older vector bundle over $X$, with a fixed $\theta$-H\"older Riemannian norm,
and that $\F \subseteq \E$ is a $\theta$-H\"older subbundle.
Let $\F^\perp \subseteq \E$ be the orthogonal complement subbundle, which is also $\theta$-H\"older.
Then the orthogonal projections 
\begin{equation}\label{e.projections}
P \colon \E \to \F \quad \text{and} \quad Q \colon \E \to \F^\perp
\end{equation}
are $\theta$-H\"older endomorphisms covering $\id_X$.
Now suppose $\Phi \colon \E \to \E$ is $\theta$-H\"older automorphism covering $T$ and that $\F$ is $\Phi$-invariant.
Then there are two induced $\theta$-H\"older automorphisms, both covering $T$, namely the obvious \emph{restricted automorphism} $\Phi|_\F \colon \F \to \F$, and the \emph{quotient automorphism} $\nicefrac{\Phi}{\F} \colon \F^\perp \to \F^\perp$ defined by $\nicefrac{\Phi}{\F} \coloneqq (Q \circ \Phi )|_{\F^\perp}$.
If the automorphism $\Phi$ is fiber-bunched (or strongly bunched) then so are $\Phi|_\F$ and $\nicefrac{\Phi}{\F}$.

\begin{proof}[Proof of \cref{t.polynomial}]
Let $\Phi \colon \E \to \E$ be a strongly bunched automorphism.
If $\Phi$ is irreducible then by \cref{t.irr_to_span} it is spannable, and by \cref{p.bounded} it is relatively product bounded, hence our claim holds with $d'=0$.
In particular, the \lcnamecref{t.polynomial} holds when $d=1$.
Now suppose $\Phi$ is reducible, that is, there exists a $\theta$-H\"older $\Phi$-invariant nontrivial subbundle $\F \subset \E$.
By induction on dimension, we can assume that the \lcnamecref{t.polynomial} holds for the restricted automorphism $\Phi|_{\F}$ and the quotient automorphism $\nicefrac{\Phi}{\F}$, that is, there are nonnegative integers $d_1 < \dim \F$ and $d_2 < d - \dim \F$ such that:
\begin{equation}\label{e.trouxa1}
\| (\Phi|_\F)_x^n \| = O \left(n^{d_1} e^{n \beta(\Phi|_\F)} \right)
\quad \text{and} \quad 
\| (\nicefrac{\Phi}{\F})_x^n \| = O \left(n^{d_2} e^{n \beta(\nicefrac{\Phi}{\F})} \right) \, .
\end{equation}
Note that, by the definitions of the automorphisms $\Phi|_\F$ and $\nicefrac{\Phi}{\F}$,
$$
\max \big\{ \| (\Phi|_\F)_x^n \| , \  \| (\nicefrac{\Phi}{\F})_x^n \| \big\} \le \| \Phi_x^n \|
\quad \text{for all $x\in X$ and $n \ge 0$,}
$$
and therefore 
\begin{equation}\label{e.trouxa2}
\max \big\{ \beta(\Phi|_\F) , \ \beta(\nicefrac{\Phi}{\F}) \big\} \le \beta(\Phi) \, .
\end{equation}
Letting $P$ and $Q$ be the orthogonal projections \eqref{e.projections}, note the identity:
$$
\Phi_x = (\Phi|_\F)_x \circ P_x + P_{Tx} \circ \Phi_x \circ Q_x + (\nicefrac{\Phi}{\F})_x \circ Q_x \, .
$$
More generally, for every $n \ge 1$, we have: 
\begin{equation}\label{e.trouxa3}
\Phi^n_x = (\Phi|_\F)^n_x \circ P_x + 
\left[ \sum_{j=0}^{n-1} (\Phi|_\F)^{n-j-1}_{T^{j+1}x} \circ P_{T^{j+1} x} \circ \Phi_{T^j x} \circ (\nicefrac{\Phi}{\F})^j_x \circ Q_x \right] + 
(\nicefrac{\Phi}{\F})^n_x \circ Q_x  \, ,
\end{equation}
which can be checked by induction.
Using the bounds \eqref{e.trouxa1}, it follows that:
$$
\| \Phi_x^n \| = O \left(n^{d_1+d_2+1} \, e^{n \max \{ \beta(\Phi|_\F) , \beta(\nicefrac{\Phi}{\F})\} }\right) \, .
$$
Noting that $d_1 + d_2 + 1 < d$ and recalling \eqref{e.trouxa2}, we obtain the desired polynomial bound. 
\end{proof}

Incidentally, note that \eqref{e.trouxa3} implies that \eqref{e.trouxa2} is an equality, that is:
\begin{equation}\label{e.split_beta}
\beta(\Phi)  = \max \big\{ \beta(\Phi|_\F) , \ \beta(\nicefrac{\Phi}{\F}) \big\} \, .
\end{equation}
Actually, a more general fact holds: for any $T$-invariant ergodic probability measure $\mu$,
\begin{equation}\label{e.split_chi}
\chi_1(\Phi, \mu) = \max \big\{ \chi_1(\Phi|_{\F}, \mu) , \chi_1(\nicefrac{\Phi}{\F}, \mu)  \big\} \, .
\end{equation}
We will use this fact  in \cref{ss.subordination}.
We were not able to find a precise reference for it, but it follows easily from the identity \eqref{e.trouxa3} together with an estimate such as \cite[Lemma~12]{B_Oseledets}.

\section{Construction of extremal norms}\label{s.norms}

\subsection{Extremal norms for spannable automorphisms}\label{ss.extremal_norms}

In this \lcnamecref{ss.extremal_norms} we state and prove the central result of this paper, \cref{t.extremal} below. 
Let us present a simple consequence first:

\begin{corollary}\label{c.extremal}
Let $T$ be a transitive hyperbolic homeomorphism.
Let $\Phi$ be a strongly bunched irreducible automorphism covering $T$.
Then $\Phi$ admits an extremal norm. 
\end{corollary}

Here is the full statement of our result on extremal norms.
Let $\theta_\uu$ be the exponent provided by applying \cref{p.regularity_above} to $\Phi^{-1}$.

\begin{theorem}\label{t.extremal}
Every spannable automorphism $\Phi \in \Aut^\theta_K(\E,T)$ admits an extremal norm $\tribar{\mathord{\cdot}}$,
which has the following additional properties: 
\begin{enumerate}

\item\label{i.norm_eccentricity} 
there exists $C_4 > 1$ such that for every $u \in \E$, 
\begin{equation}\label{e.eccentricity} 
C_4^{-1} \|u\| \le \tribar{u} \le C_4 \|u\| \, ;
\end{equation}

\item\label{i.norm_Holder}
$\tribar{\mathord{\cdot}}$ is $\theta_\uu$-H\"older, that is, there is a constant $C_5>0$ such that for all $x$, $x'\in X$,
\begin{equation}\label{e.norm_Holder}
\big| \tribar{I_{x' \gets x}} - 1 \big| \le C_5 \dd(x,x')^{\theta_\uu} \, .
\end{equation}

\item\label{i.norm_u_Holder}  
$\tribar{\mathord{\cdot}}$ is $\bar\theta$-H\"older along unstable sets with $\bar\theta \coloneqq \max\{\theta,1\}$, that is, there is a constant $C_6>0$ such that for all $x \in X$ and $x' \in W^\uu_{\epsilon_0}(x)$,
\begin{equation}\label{e.norm_u_Holder} 
\big| \tribar{H^\uu_{x' \gets x}} - 1 \big| \le C_6 \dd(x,x')^{\bar\theta} \, ;
\end{equation}

\end{enumerate}
Furthermore, for every sufficiently $C^0$-small perturbation of the automorphism $\Phi$ in the set $\Aut^\theta_K(\E,T)$, we can find an extremal norm that satisfies the properties above with the same constants $\theta_\uu$, $C_4$, $C_5$, $C_6$.

\end{theorem}

Combining the \lcnamecref{t.extremal} above with \cref{t.irr_to_span} we immediately obtain \cref{c.extremal}.

Note that part (\ref{i.norm_Holder}) of the statement of \cref{t.extremal} is compatible with the characterization of H\"olderness of a norm given by \cref{p.norm_Holder}.
In summary, our extremal norm is H\"older, but perhaps with a smaller H\"older exponent than the original $\Phi$.\footnote{A similar loss of exponent also appears in the first version of Ma\~n\'e Lemma for Anosov diffeomorphisms, obtained by Lopes and Thieullen \cite{LopesT}. Later, Bousch \cite{Bousch_amphi} obtained a stronger  Ma\~n\'e Lemma without loss of exponent. However, it is unclear whether Bousch's strategy can be applied in our setting.}
Nevertheless, part (\ref{i.norm_u_Holder}) says that the norm is more regular along unstable sets: there is no loss of exponent, and if $\theta<1$ there is a gain.

Concerning the final part of the statement of \cref{t.extremal}, recall from \cref{c.open_span} that the set of spannable automorphisms is a $C^0$-open subset of the set $\Aut^\theta_K(\E,T)$.
So the \lcnamecref{t.extremal} also says that our extremal norms vary in a bounded way if the automorphism is perturbed; this is useful to certain applications (see \cref{ss.Wirth}).

\medskip

Before commencing the actual proof, let us establish an auxiliary fact:

\begin{lemma}\label{l.bump}
Let $0<a<1$.
Let $\bar\theta \coloneqq \max\{\theta,1\}$.
Then there exists a $\bar{\theta}$-H\"older function $\zeta \colon X \times X \to [0,1]$ such that:
\begin{alignat*}{3}
\zeta(x,y) &= 1 &\quad &\text{if} &\quad \dd(x,y) &\le a\epsilon_1 \, ; \\
\zeta(x,y) &= 0 &\quad &\text{if} &\quad \dd(x,y) &\ge \epsilon_1 \, .
\end{alignat*}
\end{lemma}

\begin{proof}
If $\theta \le 1$, let $f \colon [0,+\infty) \to [0,1]$ be a non-increasing smooth function such that $f(a\epsilon_1)=1$ and $f(\epsilon_1)=0$. Then the function $\zeta(x,y) \coloneqq f(\dd(x,y))$ meets our requirements.

If $\theta > 1$ then the existence of $\zeta$ is an immediate consequence of the fact that the algebra of $\theta$-H\"older functions on $X \times X$ is normal (\cref{l.square}).
\end{proof}

\begin{proof}[Proof of \cref{t.extremal}]
As in \eqref{e.hyp_rate}, let $a \coloneqq \exp(-\min \lambda_\uu) \in (0,1)$.
Let $\zeta$ be given by \cref{l.bump}.
For each $u \in \E$, let
\begin{equation}\label{e.formula_extremal}
\tribar{u} \coloneqq \limsup_{n \to \infty} e^{-\beta(\Phi) n} \sup_{v \in \W^\uu_{\epsilon_1}(u)} \zeta(\pi(u),\pi(v))  \,  \| \Phi^n (v) \| \, .
\end{equation}
We will check that formula \eqref{e.formula_extremal} defines an extremal norm with the additional properties stated in \cref{t.extremal}.
To simplify writing, we assume from now on that $\beta(\Phi) = 0$.

Since $0 \le \zeta \le 1$ and $\Phi$ is relatively product bounded (thanks to \cref{p.bounded}), the quantity \eqref{e.formula_extremal} is always finite, and therefore defines a seminorm on each fiber of $\E$.

Take arbitrary nonzero $u \in \E$.
Since $\zeta(x,y) = 1$ whenever $y \in W^\uu_{a \epsilon_1}(x)$, we have:
$$
\limsup_{n \to \infty} \sup_{v \in \W^\uu_{a\epsilon_1}(u)} \| \Phi^n (v) \| \le \tribar{u} \, .
$$
Recalling the hyperbolicity property \eqref{e.W_contraction}, we have $\W^\uu_{a \epsilon_1}(u) \supseteq \Phi^{-1}\big(\W^\uu_{\epsilon_1}(\Phi(u))\big)$, and so:
\begin{equation}\label{e.nice1}
\limsup_{n \to \infty} \sup_{v \in \W^\uu_{\epsilon_1}(\Phi(u))} \| \Phi^n (v) \| \le \tribar{u} \, .
\end{equation}
So, letting $\tilde u \coloneqq \Phi(u)$, we have:
$$
\frac{1}{\|\tilde u\|}
\limsup_{n \to \infty} \sup_{v \in \W^\uu_{\epsilon_1}(\tilde u)} \| \Phi^n (v) \| \le \frac{\tribar{u}}{\|\tilde u\|} \le K \frac{\tribar{u}}{\|u\|} \, ,
$$
using the bound \eqref{e.bound_Phi}.
This allows us to apply \cref{l.key} to $\tilde u$ and conclude that, for some constant $C_3>1$ that only depends on $\Phi$, 
\begin{equation}\label{e.nice2}
\limsup_{n \to \infty}  \sup_{y \in X} \|\Phi^n_y \| \le C_3 K \frac{\tribar{u}}{\|u\|} \, ,  \quad\text{for all } u \in \E^\times \, .
\end{equation}
The left-hand side is at least $1$; indeed by \eqref{e.beta_other}, for every $n>0$ there exists $y \in X$ such that  $\|\Phi^n_y \| \ge e^{n\beta(\Phi)} = 1$.
Therefore:
\begin{equation}\label{e.comparison_below}
\tribar{u} \ge K^{-1} C_3^{-1} \|u\| \, , \quad\text{for all } u \in \E \, .
\end{equation}
In particular, the seminorm $\tribar{\mathord{\cdot}}$ is actually a norm.

Since $0 \le \zeta \le 1$ and $\zeta(x,y) = 0$ whenever $y \in W^\uu_{\epsilon_1}(x)$, inequality \eqref{e.nice1} implies:
$$
\tribar{\Phi(u)} \le \tribar{u}\, , \quad\text{for all } u \in \E \, ,
$$
that is, $\tribar{\mathord{\cdot}}$ is an extremal norm.

Now consider the vector $u_* \in \E^\times$ given by \cref{l.input}.
It satisfies $\tribar{u_*} \le C_2 \|u_*\|$, where $C_2>1$ is a constant depending only on $\Phi$.
Applying \eqref{e.nice2} to this vector we obtain:
$$
\limsup_{n \to \infty}  \sup_{y \in X} \|\Phi^n_y \| \le K C_2 C_3 \, .
$$
Therefore, for all $u \in \E$,
$$
\tribar{u} \le \limsup_{n \to \infty} \sup_{v \in \W^\uu_{\epsilon_1}(u)} \| \Phi^n (v) \|
\le K C_2 C_3 \sup_{v \in \W^\uu_{\epsilon_1}(u)} \| v \|
\le K C_1 C_2 C_3  \|u\| \, ,
$$
where $C_1 > 1$ is the constant from \eqref{e.bound_H}. 
So, letting $C_4 \coloneq K C_1 C_2 C_3$ and recalling the lower bound \eqref{e.comparison_below}, we obtain \eqref{e.eccentricity}: the extremal norm is uniformly comparable to the original norm by a factor $C_4$ that works not only for $\Phi$ but also for its $C^0$ perturbations in $\Aut^\theta_K(\E,T)$.

\medskip

Before proving regularity properties of the extremal norm, let us establish a few auxiliary facts.
For all $u \in \E$, $v \in \W^\uu_{\epsilon_1}(u)$, and $n\ge 0$, using \eqref{e.eccentricity}, extremality, and \eqref{e.bound_H}, we obtain:
\begin{equation}\label{e.nice3}
\| \Phi^n (v) \| \le C_4 \tribar{\Phi^n(v)} \le C_4 \tribar{v} \le C_4^2 \| v\| \le C_1 C_4^2 \|u\| \le C_1 C_4^3 \tribar{u} \, . 
\end{equation}
Fix a constant $b<1$ sufficiently close to $1$ so that:
$$
\dd(x,y) \ge b\epsilon_1  \quad \Rightarrow \quad
\zeta(x,y) < \tfrac{1}{2} C_1^{-1} C_4^{-3} \, .
$$
Then:
$$
v \in \W^\uu_{\epsilon_1}(u) \setminus \W^\uu_{b\epsilon_1}(u) \quad \Rightarrow \quad
\zeta(\pi(u),\pi(v)) \| \Phi^n (v) \| \le \tfrac{1}{2}\tribar{u} \, .
$$
So vectors $v$ outside $\W^\uu_{b\epsilon_1}(u)$ do not contribute in formula \eqref{e.formula_extremal}, which therefore can be rewritten as:
\begin{equation}\label{e.cutoff}
\tribar{u} \coloneqq \limsup_{n \to \infty} \sup_{v \in \W^\uu_{b\epsilon_1}(u)} \zeta(\pi(u),\pi(v)) \, \| \Phi^n (v) \| \, .	
\end{equation}

\medskip

We will prove property (\ref{i.norm_u_Holder}) first, and use it later in the proof of property (\ref{i.norm_Holder}).
In order to simplify writing, let us use the $O$ notation to denote constants that depend only on $\Phi$ and can be taken uniform on a $C^0$-neighborhood of $\Phi$ in $\Aut^\theta_K(\E,T)$.
In order to prove property (\ref{i.norm_u_Holder}), we need to show:
$$
x \in X, \ x' \in W^\uu_{\epsilon_0}(x), \ u \in \E_x, \  \|u\| = 1 \quad \Rightarrow \quad
\big| \tribar{H^\uu_{x' \gets x}(u)} - \tribar{u} \big| = O \big( \dd(x,x')^{\bar \theta}  \big)\, .
$$
It is sufficient to consider $x'$ very close to $x$, so assume $\dd(x',x)\le (1-b)\epsilon_1$.
Fix a unit vector $u \in \E_x$ and let $u' \coloneqq H^\uu_{x' \gets x}(u)$.
For all $v \in \W^\uu_{\epsilon_1}(u)$, and $n\ge 0$, using \eqref{e.nice3} and the fact that
$\zeta$ is $\bar\theta$-H\"older, we estimate:
$$
\Big| \zeta(x', \pi(v)) \|\Phi^n(v)\| - \zeta(x,\pi(v)) \|\Phi^n(v)\| \Big| = O \big( \dd(x,x')^{\bar \theta}  \big)\, .
$$
Noting that $\W^\uu_{b\epsilon_1}(u') \subseteq \W^\uu_{\epsilon_1}(u)$,  we have:
$$
\sup_{v \in \W^\uu_{b\epsilon_1}(u')} \zeta(x', \pi(v)) \|\Phi^n(v)\| \le 
\sup_{v \in \W^\uu_{ \epsilon_1}(u )} \zeta(x,  \pi(v)) \|\Phi^n(v)\| + O \big( \dd(x,x')^{\bar\theta} \big) \tribar{u}  \, .
$$
Using \eqref{e.cutoff} and \eqref{e.formula_extremal}, we obtain:
$$
\tribar{u'} \le \tribar{u} + O \big( \dd(x,x')^{\bar \theta}  \big)\, .
$$
On the other hand, using $\W^\uu_{b\epsilon_1}(u) \subseteq \W^\uu_{\epsilon_1}(u')$, a similar argument shows that:
$$
\tribar{u} \le \tribar{u'} + O \big( \dd(x,x')^{\bar \theta}  \big)\, .
$$
This completes the proof of property (\ref{i.norm_u_Holder}).

\medskip

We are left to check $\theta_\uu$-H\"olderness of the norm, that is, property (\ref{i.norm_Holder}). 
In fact, it is sufficient to prove $\theta_\uu$-H\"olderness along stable sets, that is: 
\begin{equation}\label{e.last_thing}
x \in X, \ x' \in W^\ss_{\epsilon_0}(x), \ u \in \E_x, \  \|u\|=1 \ \Rightarrow \ 
\big| \tribar{H^\ss_{x' \gets x}(u) } - \tribar{u} \big| = O \big( \dd(x,x')^{\theta_\uu}  \big)\, .
\end{equation}
Since we have already proven $\bar\theta$-H\"olderness of the norm along unstable sets, and $\theta_\uu \le \bar\theta$, 
property (\ref{i.norm_Holder}) will follow from \eqref{e.last_thing}: just mimic the proof of (\ref{i.irred2})~$\Rightarrow$~(\ref{i.irred1}) in \cref{c.irred}.

In order to prove \eqref{e.last_thing}, it is sufficient to consider $x'$ very close to $x$. 
Fix a unit vector $u \in \E_x$ and let $u' \coloneqq H^\ss_{x' \gets x}(u)$.
Consider arbitrary $v \in \W^\uu_{\epsilon_1}(u)$, and write $y \coloneqq \pi(v)$.
Since $\dd(x',y) \le \epsilon_1 + \dd(x,x') < 2\epsilon_1$, the bracket $[x',y] \eqcolon y'$ is well-defined. 
Let also $w \coloneqq H^\ss_{y' \gets y}(v)$, and $v' \coloneqq H^\uu_{y' \gets x'}(u')$: see \cref{f.coffee}.

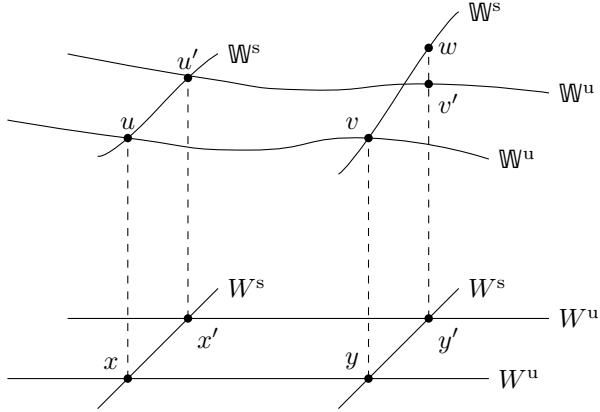
\begin{figure}[hbt]
	\begin{center}
		\begin{tikzpicture}[scale=.8]
			\draw(0,0)--(8,0) node[right]{$W^\uu$};
			\fill(2,0) circle[radius=2pt] node[above left]{$x$};
			\fill(6,0) circle[radius=2pt] node[above left]{$y$};
			\draw plot [smooth, tension=1] coordinates { (0,4.3) (2,4) (4,3.8) (6,4) (8,3.65)} node[right]{$\W^\uu$};
			\draw[dashed] (2,0)--(2,4);
			\fill(2,4) circle[radius=2pt] node[above]{$u$};
			\fill(6,4) circle[radius=2pt] node[above left]{$v$};
			\draw[dashed] (6,0)--(6,4);

			\draw(1,1)--(9,1) node[right]{$W^\uu$};
			\fill(3,1) circle[radius=2pt] node[below right]{$x'$};
			\fill(7,1) circle[radius=2pt] node[below right]{$y'$};
			\draw plot [smooth, tension=1] coordinates { (1,5.4) (3.25,5) (5,4.8) (7,4.9) (9,4.75)} node[right]{$\W^\uu$};
			\draw[dashed] (3,1)--(3,5);
			\fill(3,5) circle[radius=2pt] node[above]{$u'$};

			\draw(1.5,-.5)--(3.5,1.5) node[right]{$W^\ss$};
			\draw plot [smooth, tension=1] coordinates { (1.5,3.7) (2,4) (3,5) (3.5,5.4)} node[right]{$\W^\ss$};
			\draw(5.5,-.5)--(7.5,1.5) node[right]{$W^\ss$};
			\draw plot [smooth, tension=1] coordinates { (5.5,3.4) (6,4) (7,5.5) (7.5,6.1)} node[right]{$\W^\ss$};

			\fill(7,4.9) circle[radius=2pt] node[below right]{$v'$};
			\fill(7,5.5) circle[radius=2pt] node[right]{$w$};
			\draw[dashed] (7,1)--(7,5.5);
		\end{tikzpicture}
	\caption{Proof of property (\ref{i.norm_Holder}).}\label{f.coffee}
	\end{center}
\end{figure}

Then for each $n \ge 0$ we estimate:
\begin{gather*}
\big| \zeta(x,y) \|\Phi^n (v)\| - \zeta(x',y') \|\Phi^n (v')\| \big| \le
\circled{1} + \circled{2} + \circled{3} \, , 
\qquad \text{where}
\\
\begin{aligned}
\circled{1} &\coloneqq | \zeta(x,y) - \zeta(x',y')| \,  \|\Phi^n (v)\|  \, , \\
\circled{2} &\coloneqq \big|  \|\Phi^n (v)\| - \|\Phi^n (w)\| \big| \, , \\
\circled{3} &\coloneqq \| \Phi^n(w - v') \| \, .
\end{aligned}
\end{gather*}

In order to estimate $\circled{1}$, recall that by \eqref{e.nice3}, $\|\Phi^n(v)\| = O(1)$.
On the other hand, by H\"older-continuity of $\zeta$,
$$
| \zeta(x,y) - \zeta(x',y')| = O \big(\max\big\{\dd(x,x')^{\bar \theta} , \dd(y,y')^{\bar \theta} \big\} \big)\, .
$$
Using \cref{p.regularity_base} for $T^{-1}$, we have $\dd(y,y') = O(\dd(x,x')^{\kappa_\uu})$, where the exponent $\kappa_\uu$ is at most $1$.
So:
$$
| \zeta(x,y) - \zeta(x',y')| = O(\dd(x,x')^{\kappa_\uu \bar\theta} ) \, .
$$
Note that $\kappa_\uu \bar{\theta} \ge \kappa_\uu \theta \ge \theta_\uu$, so the weaker estimate 
$\circled{1} = O(\dd(x,x')^{\theta_\uu})$ holds.

The next term is estimated as follows:
\begin{align*}
\circled{2}
&=   \left|  \|  H^s_{T^n y' \gets T^n y} \Phi^n (v) \| - \| \Phi^n (v) \| \right | \\
&\le \left| \| H^s_{T^n y' \gets T^n y} - I_{T^n y' \gets T^n y}\| + \|I_{T^n y' \gets T^n y} \| - 1 \right| \|\Phi^n(v)\| \, . 
\end{align*}
Since $\dd(T^n y', T^n y) = o(1)$ (i.e., it tends to $0$ as $n \to \infty$), using regularity of holonomies \eqref{e.holonomy_Holder} and of the transport maps (\cref{p.norm_Holder}) together with product boundedness \eqref{e.nice3}, we conclude that $\circled{2} = o(1)$.

In order to estimate the last term, we use \cref{p.regularity_above} applied to $ T^{-1} $:
$$
\circled{3} = O(\| v' - w \|)
= O\big( \| H^\uu_{y'\gets x'} \circ H^\ss_{x' \gets x} - H^\ss_{y' \gets y} \circ H^\uu_{y \gets x} \|  \big)
= O\big( \dd(x,x')^{\theta_\uu} \big) \, ,
$$

Summing the three estimates, 
\begin{equation}\label{e.3parts}
\big| \zeta(x,y) \|\Phi^n (v)\| - \zeta(x',y') \|\Phi^n (v')\| \big| =
O \big( \dd(x,x')^{\theta_\uu}  \big) + o(1) \, .
\end{equation}

As in the proof of the previous property (\ref{i.norm_u_Holder}), we need to use the cutoff property \eqref{e.cutoff} to conclude. 
If we are careful enough to take $\dd(x,x')$ sufficiently small then $\dd(y,y') = O(\dd(x,x')^{\kappa_\uu})$ is also small and therefore the following two implications are correct: 
\begin{align*}
	\dd(x,y) \le b \epsilon_1 \ &\Rightarrow \  \dd(x',y') \le \epsilon_1 \, , \\
	\dd(x',y') \le b \epsilon_1 \  &\Rightarrow \  \dd(x,y) \le \epsilon_1 \, .
\end{align*}
That is,
\begin{align*}
	v \in \W^\uu_{b \epsilon_1}(u) \  &\Rightarrow \  v' \in \W^\uu_{\epsilon_1}(u') \, , \\
	v' \in \W^\uu_{b \epsilon_1}(u') \  &\Rightarrow \  v \in \W^\uu_{\epsilon_1}(u) \, .
\end{align*}
Then, using \eqref{e.3parts}, \eqref{e.cutoff}, and \eqref{e.formula_extremal}, we obtain:
$$
\big| \tribar{u} - \tribar{u'} \big| = O \big( \dd(x,x')^{\theta_\uu}  \big) \, ,
$$
proving \eqref{e.last_thing} and the \lcnamecref{t.extremal}.
\end{proof}

\subsection{Application: Lipschitz continuity of the maximal Lyapunov exponent}\label{ss.Wirth}

As a simple application of \cref{t.extremal}, let us establish a local regularity result for the maximal Lyapunov exponent. A similar property for the joint spectral radius (under the assumption of irreducibility) was established by Wirth \cite[Corol.~4.2]{Wirth}, also using extremal norms; see also \cite{Kozyakin} for a more precise result.

Let $\cS_K$ denote the set of spannable automorphisms in $\Aut^\theta_K(\E,T)$, which by \cref{c.open_span} is relatively $C^0$-open. 

\begin{proposition}\label{p.Wirth}
The maximal Lyapunov exponent $\beta(\mathord{\cdot})$ is a locally Lipschitz function on the set $\cS_K$, with respect to the $C^0$-norm \eqref{e.C0_norm}.
\end{proposition}

\begin{proof}
Let $\Phi \in \cS_K$. 
Let $\cU \subset \cS_K$ be a $C^0$-neighborhood of $\Phi$ where \cref{t.extremal} applies with uniform constants.
Take any two automorphisms $\Phi_1$ and $\Phi_2$ in $\cU$, and let $\tribar{\mathord{\cdot}}_1$ and $\tribar{\mathord{\cdot}}_2$ be the corresponding extremal norms provided by \cref{t.extremal}. 
Then, using the bound \eqref{e.eccentricity}, we obtain:
\begin{multline*}
e^{\beta(\Phi_2)}	\le \sup_x \tribar{\Phi_{2x}}_1 
					\le \sup_x \tribar{\Phi_{1x}}_1 + \sup_x \tribar{\Phi_{1x} - \Phi_{2x}}_1 \\
					\le e^{\beta(\Phi_1)} + C_4 \sup_x \| \Phi_{1x} - \Phi_{2x} \| 
					=   e^{\beta(\Phi_1)} + C_4 \|\Phi_1 - \Phi_2\|_0 \, ,
\end{multline*}
where $\| \mathord{\cdot} \|_0$ is the $C^0$-norm \eqref{e.C0_norm}.
By symmetry, we obtain $|e^{\beta(\Phi_1)} - e^{\beta(\Phi_1)}| \le C_4 \|\Phi_1 - \Phi_2\|_0$.
This shows that the function $e^{\beta(\mathord{\cdot})}$ is Lipschitz on the neighborhood $\cU$, with respect to the $C^0$-norm.
Since the function $\beta(\mathord{\cdot})$ is uniformly bounded on $\cU$ (and in the whole set $\Aut^\theta_K(\E,T)$, in fact), it is Lipschitz as well.
\end{proof}

\begin{remark}\label{r.regularity_beta}
For reducible automorphisms, it is clear that $\beta$ is not locally Lipschitz: see e.g.\ \cite[p.~27]{Wirth}.
Nevertheless, $\beta$ is continuous on the whole space of $\theta$-H\"older automorphisms: indeed, upper semicontinuity is automatic from \eqref{e.beta_other}, while lower semicontinuity follows by a theorem of Kalinin \cite[Theorem~1.4]{Kalinin} that allows one to approximate $\beta$ by the Lyapunov exponents of periodic orbits. Let us also remark that if $T$ is no longer hyperbolic, then $\beta$ becomes discontinuous with respect to the $C^0$ topology. 
For example, the cocycle from \cref{ex.Herman} can be $C^0$-perturbed so that $\beta$ drops to $0$, as it follows e.g.\ from the the result of \cite{AB}.
\end{remark}

\subsection{Barabanov-like norms for linear cocycles over shifts}\label{ss.Barabanov}

Let us consider subshifts of finite type, that is, $X$ is the set of two-sided sequences $(x_n)_{n\in \Z}$ in an alphabet $\{0,1,\dots, N-1\}$ whose neighboring pairs are those allowed by a fixed 0-1 matrix, and $T \colon X \to X$ is the (left) shift map. 
As usual, we consider on $X$ the (ultra)metric
\begin{equation}\label{e.ultrametric}
\dd(x,y) \coloneqq e^{-\lambda k}\, \quad \text{where } k = \min \{|n| \st x_n \neq y_n \},
\end{equation}
and $\lambda>0$ is a fixed parameter.
Then $T$ is a hyperbolic homeomorphism.
Indeed letting $\epsilon_0 \coloneqq e^{-\lambda}$, the corresponding local unstable and stable sets at $x=(x_n) \in X$ are:
\begin{align*}
W_\mathrm{loc}^\uu (x) &\coloneqq W_{\epsilon_0}^\uu (x) = \big\{ (y_n) \in X \st y_n = x_n \text{ for all } n\le 0\big\} \, , \\
W_\mathrm{loc}^\ss (x) &\coloneqq W_{\epsilon_0}^\ss (x) = \big\{ (y_n) \in X \st y_n = x_n \text{ for all } n\ge 0\big\} \, .
\end{align*}
and so hyperbolicity property (\ref{i.lambdas}) holds with $\lambda_\uu = \lambda_\ss = \lambda$,
property (\ref{i.bracket}) holds with $2\epsilon_1 = \epsilon_0$,
and property (\ref{i.bounded_angles}) holds with $C=1$.
Also note that \cref{p.regularity_base} holds with $\kappa_\ss = 1 = C$.

We will consider $\theta$-H\"older automorphisms covering the subshift $T \colon X \to X$.
Since $X$ is a Cantor set, every $\theta$-H\"older vector bundle is trivial, i.e., $\theta$-H\"older isomorphic to the product bundle.
So we are actually dealing with $\theta$-H\"older linear cocycles; nevertheless, we will keep using the vector bundle terminology.

\begin{example}\label{ex.one-step}
As mentioned in \cref{ss.known},
the \emph{one-step cocycle} determined by a $N$-tuple of matrices $(A_0,\dots,A_{N-1}) \in \GL(d,\R)^N$ is the pair $(T,F)$ where $T$ is the full shift on $N$ symbols and $F \colon X \to \GL(d,\R)$ is given by $F(x) \coloneqq A_{x_0}$.
Let $\Phi$ the associated automorphism \eqref{e.cocycle}. 
Then $e^{\beta(\Phi)}$ is joint spectral radius of the set $\{A_0,\dots,A_{N-1}\}$.
Since $F$ is locally constant, it is $\theta$-H\"older for any $\theta \in (0,+\infty)$.
Choosing $\theta$ large enough, the automorphism $\Phi$ becomes fiber-bunched.
(Alternatively, we can take $\theta=1$, say, and then take the parameter $\lambda$ large enough.)
The holonomies are locally trivial:
\begin{equation}\label{e.trivial_holonomies}
\star \in \{\uu, \ss\}, \ 
y \in W^\star_\mathrm{loc}(x) \quad \Rightarrow \quad H^\star_{y \gets x} = \id \, .
\end{equation}

A useful generalization of one-step cocycles are the \emph{sofic cocycles} from \cite[\S~5.1]{BPS}; the same concept appears in \cite{PEDJ} under the terminology \emph{constrained switching systems}.
\end{example}

Let us present an improved version of \cref{t.extremal} for subshifts of finite type. 
We obtain an extremal norm with an additional Barabanov-like property: 
given any vector $u \in \E $, there always exists a vector in its local unstable set $\W^\uu_\mathrm{loc}(u) \coloneqq \W^\uu_{\epsilon_0}(u)$ whose expansion factor in a single
iterate equals the maximum asymptotic expansion rate $ e^{\beta(\Phi)}  $.
Furthermore, the norm is invariant under local unstable holonomies.
Therefore, for the case of one-step cocycles, we reobtain the Barabanov property \eqref{e.Barabanov}.

\begin{theorem}\label{t.Barabanov}
Let $ T $ be a two-sided subshift of finite type.
Let $ \E $ be a $d$-dimensional $\theta$-H\"older vector bundle.
Let $ \Phi $ be a spannable automorphism of $ \E $ covering $ T $.
Then
\begin{equation}\label{e.formula_Barabanov}
\tribar{u} \coloneqq \limsup_{n \to \infty} e^{-\beta(\Phi) n} \sup_{v \in \W^{\uu}_{\mathrm{loc}}(u)} \| \Phi^n (v) \| 
\end{equation}
is a well-defined \emph{Barabanov norm} on $ \E $, namely, an extremal norm satisfying, for all  $ u \in \E$,
\begin{enumerate}[label={\upshape(\arabic*)},ref=\arabic*]
\item\label{i.Barabanov_invariance} 
\emph{local $H^{\uu}$-invariance:} $ \tribar{u} = \tribar{v}$ for all $ v \in \W^{\uu}_{\mathrm{loc}}(u) $; 
\item\label{i.Barabanov_calibration} 
\emph{calibration:} there exists $ v \in \W^{\uu}_{\mathrm{loc}}(u) $ such that 
$ \tribar{\Phi(v)} = e^{\beta(\Phi)} \tribar{v}$.
\end{enumerate}
Furthermore, $\tribar{\mathord{\cdot}}$ satisfies the other properties stated in \cref{t.extremal}.
\end{theorem}

Let us comment on the hypotheses.
Given $\theta>0$, \cref{l.needed_strength} holds with the value $\eta_0 = \theta/3$; indeed, this follows from formula~\eqref{e.needed_strength}, recalling that $\lambda_\uu = \lambda_\ss = \lambda$, and noting that we can also take $\Lambda_\uu = \lambda$ in \eqref{e.Lambda_u}.
Therefore $\theta$-H\"older automorphism $\Phi \colon \E \to \E$ covering $T$ is strongly bunched 
if $\Phi$ is fiber-bunched and the fibers of $\E$ have dimension $d \le 2$, or $\Phi$ is a $(\theta/3,\theta)$-bunched.
In that case, it follows from \cref{t.irr_to_span} that $\Phi$ is spannable, provided it is irreducible and $T$ is transitive.

\begin{proof}[Proof of \cref{t.Barabanov}]
Since $X$ is a Cantor set, we can simplify the construction in \cref{t.extremal} and dispense with the bump function $\zeta$ from \cref{l.bump}. Ultimately, we can replace the definition \eqref{e.formula_extremal} of the extremal norm by the simpler formula \eqref{e.formula_Barabanov}. 
Note that in the latter formula we maximize over $\W^{\uu}_{\mathrm{loc}}(u) \coloneqq \W^\uu_{\epsilon_0}(u)$ instead of $\W^\uu_{\epsilon_1}(u) = \W^\uu_{\epsilon_0/2}(u)$; this is possible because the metric on $X$ is an ultrametric. 
It is straightforward to check that the proof in \cref{t.extremal} applies, with simplifications.
It is immediate from its definition that the norm satisfies local $H^{\uu}$-invariance, that is property (\ref{i.Barabanov_invariance}), which of course subsumes property (\ref{i.norm_u_Holder}) from \cref{t.extremal}.
		
We only left to check the calibration property (\ref{i.Barabanov_calibration}). 
Given $ u \in \E$, by definition, there exist sequences $n_i \nearrow \infty$ and $ v_i \in \W^{\uu}_\mathrm{loc}(u) $
such that
$$
\tribar{u} = \lim_{i\to\infty} e^{-\beta(\Phi) n_i} \| \Phi^{n_i}v_i \| \, .
$$
Denote $ y_i = \pi(v_i) $. By compactness, we may suppose that $ y_i \to y \in W^{\uu}_{\mathrm{loc}}(x) $ and $ v_i \to v \in \W^{\uu}_\mathrm{loc}(u) $. 
For $ i $ large enough, $y'_i \coloneqq T(y_i) \in W^{\uu}_{\mathrm{loc}}(T(y))$. We can assume that this property is true for all $i$.
Let $ v'_i \coloneqq \Phi(v_i) $. Thus, we have:
\begin{align*}
\tribar{\Phi(v)}
&= \limsup_{n\to \infty}  e^{- \beta(\Phi) n} \sup_{v' \in  \W^{\uu}_{\mathrm{loc}}(\Phi(v))} \| \Phi^n v' \| \\
&\ge \limsup_{i \to \infty}  e^{-(n_i-1) \beta(\Phi) } \| \Phi^{n_i-1}v'_i \| \\
&=   \lim_{i \to \infty}  e^{-(n_i-1) \beta(\Phi) } \| \Phi^{n_i}v_i \| \, = \, e^{\beta(\Phi)} \tribar{u} \, = \, e^{\beta(\Phi)} \tribar{v} \, .
\end{align*}
By extremality, the inequality is actually an equality.
This proves calibration.
\end{proof}

\section{Mather sets} \label{s.Mather}

In traditional ergodic optimization, that is, the optimization of Birkhoff averages (see
\cite{Jenkinson_survey,Jenkinson_survey_new,Garibaldi_book}), a \emph{maximizing set} is a closed subset such that an invariant probability is maximizing if and only if its support lies on this subset. The existence of such sets is guaranteed in any context where a Ma\~n\'e Lemma holds.
The \emph{Mather set} is the smallest maximizing set: it is defined as the union of the supports of all maximizing measures. The nomenclature is borrowed from Lagrangian dynamics, where  the concept of minimizing measures proved to be a useful generalization of the notion of action minimizing orbits: see \cite{Mat}. There are other canonical maximizing sets, such as the Aubry set: see e.g.\ \cite{Garibaldi_book}. Some of these concepts have been already considered in the optimization of the top Lyapunov exponent: see \cite{Morris_Mather,GG}.

In this \lcnamecref{s.Mather}, we are going to study some notions of Mather sets for  continuous vector bundle automorphisms, not necessarily with any H\"older or hyperbolicity structures. Our approach was profoundly influenced by some works of Morris \cite{Morris_rapidly,Morris_Mather}.

\medskip

Throughout the \lcnamecref{s.Mather}, we assume that $X$ is a compact metric space, $T \colon X \to X$ is a homeomorphism, $\E$ is a $d$-dimensional vector bundle over $X$, and $\Phi$ is an automorphism covering $T$.

\subsection{The first Mather set}\label{ss.first_Mather}

By a \emph{Lyapunov maximizing measure} we mean any $T$-invariant probability $\mu$ whose upper Lyapunov exponent $ \chi_1(\Phi, \mu)$ equals $\beta(\Phi) $. 
Following Morris \cite{Morris_Mather}, we define the \emph{(first) Mather set} $M(\Phi) \subseteq X$ as the union of the supports of all Lyapunov maximizing measures.

\begin{proposition}\label{p.Mather_support}
The Mather set $M(\Phi)$ is the support of some Lyapunov maximizing measure and, a fortiori, it is a nonempty, compact, and $T$-invariant set.
\end{proposition}

\begin{proof}
The argument is quite standard, but we add it for completeness.
For simplicity, write $M = M(\Phi)$.
As explained at the introduction, at least one Lyapunov maximizing measure exists, so $M \neq \emptyset$.
Given a countable basis $\{B_j\}_{j\in\N}$ for the topology of $X$, consider the subset of indices $ J \coloneqq \{ j \in \N \st B_j \cap M \neq \emptyset  \} $.
For each $ j \in J $, we assign a Lyapunov maximizing measure $\mu_j$ such that $B_j \cap \supp \mu_j \neq \emptyset $, which means $\mu_j(B_j)>0$.
Define then $ \mu\coloneqq\sum_{j \in J} \alpha_j \mu_j$, where $ \alpha_j>0$ and $ \sum_{j \in J} \alpha_j = 1 $. As a convex combination of Lyapunov maximizing measures, $ \mu$ is also  Lyapunov maximizing. 
Now consider an arbitrary $k \in \N$ such that $B_k \cap \supp  \mu = \emptyset $.
Then for all $ j \in J $ we have $B_k \cap \supp \mu_j = \emptyset $ and therefore  $\mu_j(B_k)=0$.
This implies that $k\not\in J$, and so $B_k \cap M = \emptyset$.
We have shown that $ X \setminus \supp  \mu \subseteq X \setminus M$, which yields $M(\Phi) = \supp  \mu$.
\end{proof}

\begin{proposition}\label{p.Mather_calibration}
If $\Phi$ admits an extremal norm $ \tribar{\cdot}$ then
\begin{equation*} 
\tribar{\Phi_x^n} = e^{n \beta(\Phi)}, \qquad  \forall \, x \in M(\Phi) \, , \forall \, n \ge 1 \, ,
\end{equation*}
and, in particular, every $T$-invariant probability measure whose support is contained in $M(\Phi)$ is Lyapunov-maximizing.
\end{proposition}

\begin{proof}
By extremality,
$$ 
f_n(x) \coloneqq \frac{1}{n}\log \tribar{\Phi_x^n} \le \beta(\Phi), \qquad \forall \, x \in X, \forall \, n \ge 1. 
$$
On the other hand, if $\mu$ is the Lyapunov maximizing measure with $\supp \mu = M(\Phi)$ constructed in \cref{p.Mather_support} then, by Kingman's subadditive theorem,
$ \inf_n \frac{1}{n} \int f_n \, d\mu = \beta(\Phi) $.
It follows that $\frac{1}{n} \int f_n \, d\mu = \beta(\Phi)$ for each $n \ge 1$.
Since the functions $f_n$ are continuous, they must be identically equal to $\beta(\Phi)$ over $\supp  \mu = M(\Phi)$,
as we wanted to show. 
\end{proof}

\subsection{Mather sets of higher index} \label{ss.other_Mathers}

Up to here we have only considered the first Lyapunov exponent $\chi_1$, but now we will need to consider the full Lyapunov spectrum. Let us recall the definitions and main properties, referring to \cite{Arnold} for details.

If $\mu$ is a $T$-invariant probability measure then the \emph{Lyapunov exponents} of the automorphism $\Phi$ with respect to $\mu$ are the numbers 
\begin{equation}\label{e.Lyapunov_spectrum}
\chi_1(\Phi,\mu) \ge \chi_2(\Phi,\mu) \ge \cdots \ge \chi_d(\Phi,\mu) 
\end{equation}
uniquely defined by the following equations: for every $p \in \ldbrack 1,d \rdbrack$,
\begin{equation}\label{e.Lyapunov_wedge}
\sum_{i=1}^p \chi_i(\Phi,\mu) = 
\chi_1(\WEDGE^p \Phi, \mu) \, ,
\end{equation}
where the automorphism $\WEDGE^p \Phi \colon \WEDGE^p \E \to \WEDGE^p \E$ is the $p$-fold exterior power of the automorphism ${\Phi \colon \E \to \E}$.

Suppose $\mu$ is ergodic, and that $\lambda$ is a Lyapunov exponent with respect to $\mu$ of \emph{multiplicity} $k$, in the sense that it appears $k$ times in the list \eqref{e.Lyapunov_spectrum}.
Then Oseledets' theorem says that for $\mu$-a.e.\ $x\in X$, there exists a $k$-dimensional subspace $\mathbb{O}_x(\lambda)$ of the fiber $\E_x$, called a \emph{Oseledets space}, such that:
$$
u \in \mathbb{O}_x(\lambda) \setminus \{0\} \quad \text{if and only if} \quad
\lim_{n \to \pm \infty} \frac{1}{n} \log \|\Phi^n(u)\| = \lambda \, .
$$
Moreover, Oseledets spaces form a splitting of $\E_x$, depend measurably on the point $x$, and are $\Phi$-equivariant.

\medskip

We now consider other Mather sets that take multiplicity into account.
Define a chain of sets
\begin{equation}\label{e.chain_Mathers}
M(\Phi) = M_1(\Phi) \supseteq M_2(\Phi) \supseteq \cdots \supseteq M_d(\Phi)
\end{equation}
as follows: the \emph{$p$-th Mather set} $M_p(\Phi)$ is the union of the supports of all $T$-invariant probabilities $ \mu $ whose $p$ first Lyapunov exponents are all maximal, that is,
$$
\chi_1(\Phi,\mu) = \chi_2(\Phi,\mu) = \cdots = \chi_p(\Phi,\mu) = \beta(\Phi) \, .
$$
Repeating the argument of the proof of \cref{p.Mather_support}, we see that if the set $M_p(\Phi)$ is nonempty then there exists a measure $\mu$ with $p$ maximal Lyapunov exponents and whose support is exactly $M_p(\Phi)$; in particular, $M_p(\Phi)$ is compact and $T$-invariant.

The following properties follow immediately from the definition of Mather sets and relations \eqref{e.Lyapunov_spectrum} and \eqref{e.Lyapunov_wedge}:

\begin{proposition}\label{p.Mathers}
For any $p \in \ldbrack 1,d \rdbrack$, we have:
$$
\beta(\WEDGE^p \Phi) \le p \beta(\Phi)
\quad \text{and} \quad
M_p(\Phi)\subseteq M_1(\WEDGE^p \Phi) \, .
$$
Furthermore, these two relations become equalities if and only if $M_p(\Phi) \neq \emptyset$.
\end{proposition}

Let us extend the chain of sets \eqref{e.chain_Mathers}.
Let $M_{d+1}(\Phi) \coloneqq \emptyset$
and let $M_0(\Phi)$ be defined as the union of the supports of all $T$-invariant probability measures (which is also the support of one of them). So $M_0(\Phi)$ only depends on $T$, and is in fact the classical \emph{minimal center of attraction} of $T$: see \cite{Sigmund}, \cite[p.~164]{Akin}.

\subsection{Dominated splittings over the Mather sets}\label{ss.Mather_dom}

We want to provide more information about the action of $\Phi$ on the fibers above the Mather sets, assuming the existence of a extremal norm. 
We will use the notion of dominated splitting, which is very useful in Differentiable Dynamics (see \cite{BDV_book}).
It appears in the celebrated book \cite{HPS} as \emph{relative pseudo hyperbolicity}.
It also appears in ODE and Control Theory under the terminology \emph{exponentially separated splitting} (see \cite{CK_book}), and is intimately related to the concept of Anosov representations in Geometric Group Theory (see \cite{BPS}).

\medskip

Let $Y \subseteq X$ be a nonempty $T$-invariant compact set, and let $\E_Y \coloneqq \pi^{-1}(Y)$ be the restricted vector bundle. (Recall that $\pi \colon \E \to X$ denotes the bundle projection.)
Let $\| \mathord{\cdot} \|$ be a Finsler norm on $\E$.
Suppose that the bundle $\E_Y$ splits as a direct sum $\F \oplus \G$ of two (continuous) subbundles whose fibers $\F_x$, $\G_x$ have constant dimensions, and are equivariant in the sense that $\Phi_x(\F_x) = \F_{Tx}$, $\Phi_x(\G_x) = \G_{Tx}$. 
We say that $\F \oplus \G$ is a \emph{dominated splitting} with \emph{dominating} bundle $\F$ and \emph{dominated} bundle $\G$ if there are positive constants $c$ and $\tau$ such that for each point $x \in Y$, if $u \in \F_x$, $v \in \G_x$ are unit vectors then 
\begin{equation}\label{e.def_DS}
\| \Phi_x^n(v) \| \le c e^{-\tau n} \| \Phi_x^n(u) \| \quad \text{for all } n \ge 0.
\end{equation}

An equivalent definition is to say that there exists an \emph{adapted norm} $\| \mathord{\cdot} \|$ for which relation \eqref{e.def_DS} holds with $c=1$ (and therefore only needs to be checked for $n=1$): see \cite{Gourmelon}. Dominated splittings are unique given the dimensions: see \cite[Prop.~2.2]{CroPo}. Continuity of the subbundles actually follows from the uniform estimates \eqref{e.def_DS} and therefore could be removed from the definition: see \cite[Prop.~2.5]{CroPo}. Actually, if $\Phi$ is H\"older then the bundles of a dominated splitting are always H\"older (with a smaller exponent): see \cite[Thrm.~4.11]{CroPo}. 
Domination can be characterized in terms of existence of invariant cone fields: see \cite[Thrm.~2.6]{CroPo}.
This implies strong robustness properties: see \cite[Corol.~2.8]{CroPo}.

\medskip

We will use another criterion for the existence of dominated splittings, expressed in terms of singular values.
Recall that if $L \colon E \to F$ is a linear map between $d$-dimensional inner product spaces, then the \emph{singular values}  $\sigma_1 (L) \ge \cdots \ge \sigma_d (L)$ are the eigenvalues of the symmetric operator $(L^* L)^{1/2}$. So $\sigma_1(L)$ coincides with the Euclidean operator norm $\|L\|$.
Endowing the exterior power spaces with the induced inner products, the exterior powers of $L$ have norm:
\begin{equation}\label{e.singular_wedge}
\| \WEDGE^p L \| = \sigma_1(L) \sigma_2(L) \cdots \sigma_p(L) \, ;
\end{equation}
see e.g.\ \cite[p.~120]{Arnold}.
Another useful characterization of the singular values is:
\begin{equation}\label{e.maxmin}
\sigma_p(L) = \max_{V \in \cG_p(E)} \min_{u \in V} \frac{\|L u\|}{\|u\|} \, ;
\end{equation}
see e.g.\ \cite[p.~68]{Stewart}.

\medskip

A theorem from \cite{BG} 
says that the domination is equivalent to a uniform exponential gap between singular values of the powers of $\Phi$ (computed with respect to a Riemannian norm fixed a priori).
More precisely:

\begin{theorem}[Bochi--Gourmelon]\label{t.BG}
The bundle $\E_Y$ admits a dominated splitting with a dominating bundle of dimension $p$ if and only if there exist positive constants $c$ and $\tau$ such that
$$
\sigma_{p+1}(\Phi_x^n) \le c e^{-\tau n} \sigma_p(\Phi_x^n)  \quad \text{for all $x \in Y$ and $n \ge 0$.} 
$$
\end{theorem}

We now come back to the Mather sets:

\begin{theorem}\label{t.dom}
Suppose $\Phi$ admits an extremal norm $\tribar{\cdot}$.
Let $p \in \ldbrack 1, d \rdbrack$.
Suppose $Y$ is a nonempty compact $T$-invariant set contained in $M_p(\Phi) \setminus M_{p+1}(\Phi)$.
Then the restricted bundle $\E_Y$ admits a dominated splitting $\F \oplus \G$ where the dominating bundle $\F$ has fibers of dimension $p$ and is calibrated in the sense that $\tribar{\Phi(u)} = e^{\beta(\Phi)} \tribar{u}$ for every $u \in \F$.
\end{theorem}

In particular, if exactly one of the sets $M_p(\Phi) \setminus M_{p+1}(\Phi)$ is nonempty then we obtain a dominated splitting over the whole Mather set $M(\Phi)$.

Related results were previously obtained by Morris: \cite[Theorem~2.1]{Morris_rapidly} produces a dominated splitting under the weaker assumption of relative product boundedness, but with the strong hypothesis that the set $Y$ is minimal (i.e., all orbits in $Y$ are dense). Assuming existence of an extremal norm, Morris also proves the calibration property of the dominating bundle in his Theorem~2.2. 

For a complement to \cref{t.dom}, see \cref{p.Riem_weak} in \cref{ss.Riemann}.

\subsection{Proof of \texorpdfstring{\cref{t.dom}}{the domination theorem}}

Consider the set of vectors whose bi-infinite orbits under $\Phi$ are \emph{calibrated} with respect to the extremal norm $\tribar{\mathord{\cdot}}$:
\begin{equation}\label{e.calibrated}
\K \coloneqq \big\{ u \in \E \st \tribar{\Phi^n(u)} = e^{n \beta(\Phi)}\tribar{u} \text{ for all } n \in \Z \big\}.
\end{equation}
This is a closed, $\Phi$-invariant subset of $\E$.
Denote its fibers by $\K_x \coloneqq \E_x \cap \K$.

\begin{proposition}\label{p.Oseledets}
There exists a $T$-invariant Borel set $R \subseteq M(\Phi)$ such that:
\begin{itemize}
\item $\mu ( M(\Phi) \setminus R ) = 0$ for every $T$-invariant probability measure $\mu$;
\item for all $x \in R$, the Oseledets space corresponding to the Lyapunov exponent $\beta(\Phi)$ exists and coincides with~$\K_x$.
\end{itemize}
\end{proposition}

\begin{proof}
Without loss of generality, assume that $\beta(\Phi)=0$.
Let $R_0 \subseteq X$ be the Borel set of 
points that satisfy the conclusions of Oseledets theorem. 
For each $x \in R_0 \cap M(\Phi)$, the Oseledets space $\mathbb{O}_x = \mathbb{O}_x(0) \subseteq \E_x$ is well-defined and has positive dimension, say $p(x)$.
Calibrated vectors have zero Lyapunov exponent, so $\K_x \subseteq \mathbb{O}_x$.
Consider the unit balls on these Oseledets spaces, i.e., $\cB_x \coloneqq \{ u \in \mathbb{O}_x \st \tribar{u} \le 1\}$.
Then $\Phi_x(\cB_{x}) \subseteq \cB_{Tx}$, and therefore the following function is non-positive:
$$
\psi(x) \coloneqq \log \frac{\vol(\Phi_x(\cB_{x}))}{\vol(\cB_{Tx})} \, .
$$
Here $\vol$ means $p(x)$-dimensional volume with respect to a fixed Riemannian norm on the bundle $\E$; of course, the choice of this metric does not affect the function $\psi$.
Then $\psi$ is cohomologous to the function
$\tilde \psi(x) \coloneqq \log \det \Phi(x)|_{\mathbb{O}_x}$, where $\det$ denotes the signless determinant induced by the Riemannian metric (see \cite[p.~213]{Arnold}); indeed $\psi = \tilde \psi + \phi - \phi\circ T$ where $\phi(x) \coloneqq \log \vol(\cB_{x})$.
All these functions are Borel measurable and bounded.
Let $\mu$ be any $T$-invariant probability measure supported on $M(\Phi)$, that is, any Lyapunov maximizing measure.
As a consequence of Oseledets theorem, we have $\int \tilde\psi \, d\mu = 0$ (see \cite[p.~214]{Arnold}). Since $\psi$ is cohomologous to $\tilde \psi$, its integral is zero as well. But $\psi \le 0$, so $\psi = 0$ $\mu$-a.e.
Let $R_1 \coloneqq \{ x \in R_0 \cap M(\Phi) \st \psi(x)=0 \}$ and $R \coloneqq \bigcap_{n \in \Z} T^{-n}(R_1)$; then $\mu (R) = 1$.
Noting that $\psi(x) = 0$ if and only if $\Phi(x)|_{\mathbb{O}_x}$ preserves $\tribar{\mathord{\cdot}}$, we see that if $x \in R$ then $\mathbb{O}_x \subseteq \K_x$.
As remarked before, the reverse inclusion is automatic, so $\mathbb{O}_x = \K_x$ for every $x \in R$.
Since set $R$ has full measure with respect to any Lyapunov maximizing measure, the set $M(\Phi) \setminus R$ has zero measure with respect to any $T$-invariant probability measure, as we wanted to show.
\end{proof}

\begin{corollary}\label{c.dimension}
Suppose $\Phi$ admits an extremal norm $ \tribar{\cdot}$.
For each $p \in \ldbrack 1, d \rdbrack$ and $x \in M_p(\Phi)$, the set $\K_x$ contains a vector space of dimension $p$.
\end{corollary}

\begin{proof}
Let $\mu$ be a measure whose $p$ first Lyapunov exponents equal $\beta(\Phi)$ and whose support equals $M_p(\Phi)$.
Given $x \in M_p(\Phi)$, take a sequence of neighborhoods $U_i$ converging to $x$.
Since $\mu(U_i)>0$, by \cref{p.Oseledets} we can find $x_i \in U_i$ such that the Oseledets space corresponding to the Lyapunov exponent $\beta(\Phi)$ exists and coincides with~$\K_{x_i}$. Moreover, these spaces have dimensions at least $p$.
Passing to subsequences, we can assume that these dimensions are constant equal to some $q \ge p$, and that $\K_{x_i}$ converges to some $q$-dimensional space $V$. As $\K$ is a closed subset of $\E$, we conclude that $V \subseteq \K_x$, completing the proof.
\end{proof}

\begin{remark}\label{r.exceptional_fibers}
It is not necessarily the case that $\K_x$ is a subspace: see \cref{ex.non_space} in \cref{ss.examples_calibrated}.
On the other hand, if $\K_x$ is a subspace, then by \cref{c.dimension} its dimension is at least the number $p$ such that $x \in M_p(\Phi) \setminus M_{p+1}(\Phi)$. 
However, it is not necessarily true that $\dim \K_x  = p$: see \cref{ex.bad_dim} in \cref{ss.examples_calibrated}.
\end{remark}

\begin{proof}[Proof of \cref{t.dom}]
As usual, it is sufficient to consider $\beta(\Phi) = 0$.

In the case $p=d$, we have $Y \subseteq M_d(\Phi)$ and so by \cref{c.dimension} the extremal norm is preserved along the bundle~$\E_Y$. So the trivial splitting $\E_Y \oplus 0$ has the required properties.
So let us suppose that $p<d$.

Fix a Riemannian norm $\| \mathord{\cdot} \|$ on $\E$.
For each $x \in Y$, by \cref{c.dimension} the fiber $\E_x$ contains a $p$-dimensional subspace formed by vectors $u$ such that for every $n \ge 0$, we have $\tribar{\Phi^n(u)} = \tribar{u}$ and therefore 
$c_1^{-1} \|u\| \le \| \Phi^n(u) \| \le c_1 \|u\|$, for some constant $c_1 \ge 1$.
Recalling the maxmin characterization of singular values \eqref{e.maxmin}, we conclude that:
\begin{equation}\label{e.almost_calibration}
c_1^{-1} \le \sigma_p(\Phi^n_x) \le \cdots \le \sigma_1(\Phi^n_x) \le c_1 \quad \text{for all $x \in Y$ and $n \ge 0$.} 
\end{equation}
On the other hand, note that the set $M_{p+1} (\Phi|_{\E_Y})$ is contained in $Y \cap M_{p+1}(\Phi)$ and therefore is empty.
So \cref{p.Mathers} yields
$$
\beta \left( \WEDGE^{p+1} (\Phi|_{\E_Y}) \right) < (p+1) \beta(\Phi|_{\E_Y}) = 0 \, .
$$
Recalling that $\beta(\mathord{\cdot})$ can also be characterized by \eqref{e.beta_other}, we conclude that there exist positive constants $c_2$ and $\tau$ such that for all $x \in Y$ and $n \ge 0$,
$$
\| \WEDGE^{p+1} \Phi^n_x \| \le c_2 e^{-\tau n} \, .
$$
So, using \eqref{e.singular_wedge} and \eqref{e.almost_calibration}, we have:
$$
\sigma_{p+1} (\Phi^n_x) = \frac{\| \WEDGE^{p+1} \Phi^n_x \|}{\| \WEDGE^p \Phi^n_x \|} \le c_1^p c_2 e^{-\tau n}
$$
and $\sigma_p (\Phi^n_x) \ge c_1^{-1}$. 
So we have a uniform exponential gap between the $p$-th and $p+1$-th singular values.
By \cref{t.BG}, the bundle $\E_Y$ admits a dominated splitting $\F \oplus \G$ with a dominating bundle $\F$ of dimension $p$.

To conclude, we need to check that $\Phi$ preserves the extremal norm along the bundle~$\F$.
We will actually show that, in terms of notation \eqref{e.calibrated}, $\K_x = \F_x$ for every $x \in Y$.
Since $\Phi$ is product bounded, domination implies that vectors in $\G$ are uniformly contracted in the future, and therefore uniformly expanded in the past.
Furthermore, any vector in $u \in \E_x \setminus \F_x$ is uniformly expanded in the past, since we can write $u = v + w$ with $v \in \F_x$, $w \in \G_x \setminus \{0\}$ and then
$$
\tribar{\Phi^{-n}(u)} \ge \tribar{\Phi^{-n}(w)} \left( 1 - \frac{\tribar{\Phi^{-n}(v)}}{\tribar{\Phi^{-n}(w)}} \right) \to \infty \quad \text{as $n \to +\infty$ . }
$$
In particular, vectors in $\E_x \setminus \F_x$ cannot be calibrated; that is, $\K_x \subseteq \F_x$. This inclusion cannot be strict, thanks to \cref{c.dimension}.
So $\K_x = \F_x$, as claimed.
\end{proof}

\section{Further applications of extremal norms and Mather sets}\label{s.app}

\subsection{Subordination}\label{ss.subordination}

By definition, the Mather set $M(\Phi)$ contains the support of every Lyapunov maximizing measure. Let us see that the converse holds under the hypothesis of strong fiber-bunching, regardless of reducibility:

\begin{theorem}\label{t.subordination}
Let $T$ be a transitive hyperbolic homeomorphism.
Let $\Phi$ be strongly bunched automorphism covering $T$.
Then every $T$-invariant probability measure whose support is contained in the Mather set $M(\Phi)$ is Lyapunov maximizing.
\end{theorem}

\begin{proof}
Let $\Phi \colon \E \to \E$ be a strongly bunched automorphism and let $\nu$ be a $T$-invariant probability measure whose support is contained in $M(\Phi)$; we want to prove that $\chi_1(\Phi, \nu) = \beta(\Phi)$.
By ergodic decomposition, it is sufficient to consider the case of ergodic $\nu$.

If $\Phi$ is irreducible (which is certainly the case if $d=1$) then $\Phi$ is spannable by \cref{t.irr_to_span}, and so $\Phi$ admits an extremal norm by \cref{t.extremal}. Then \cref{p.Mather_calibration} yields the desired conclusion.

From now on, assume that $\Phi$ is reducible, that is, there exists a $\theta$-H\"older $\Phi$-invariant nontrivial subbundle $\F \subset \E$.
By induction on dimension, we can assume that the \lcnamecref{t.subordination} holds for the restricted automorphism $\Phi|_{\F}$ and the quotient automorphism $\nicefrac{\Phi}{\F}$.
Recall from \eqref{e.split_beta} that $\beta(\Phi) = \max\{ \beta(\Phi|_{\F}) , \beta(\nicefrac{\Phi}{\F}) \}$.
 
As a first case, suppose that $\beta(\Phi) = \beta(\Phi|_{\F}) > \beta(\nicefrac{\Phi}{\F})$.
Then it follows from \eqref{e.split_chi} that an ergodic measure $\mu$ is Lyapunov maximizing for $\Phi$ if and only if it is Lyapunov maximizing for $\Phi|_{\F}$.
Therefore the Mather sets coincide: $M(\Phi) = M(\Phi|_{\F})$.
The measure $\nu$ fixed at the beginning is supported on this set; so, by the induction hypothesis, it is Lyapunov maximizing for $\Phi|_{\F}$, that is, it is Lyapunov maximizing for $\Phi$, as we wanted to show.

The second case where $\beta(\Phi) = \beta(\nicefrac{\Phi}{\F}) > \beta(\Phi|_{\F})$ is entirely analogous.

In the last case, we have  $\beta(\Phi) = \beta(\Phi|_{\F}) = \beta(\nicefrac{\Phi}{\F})$.
Then it follows from \eqref{e.split_chi} that an ergodic measure $\mu$ is Lyapunov maximizing for $\Phi$ if and only if it is Lyapunov maximizing for $\Phi|_{\F}$ or for $\nicefrac{\Phi}{\F}$.
Therefore $M(\Phi) = M(\Phi|_{\F}) \cup M(\nicefrac{\Phi}{\F})$.
So the measure $\nu$ fixed at the beginning has a support contained in the union of the two closed $T$-invariant sets 
$M(\Phi|_{\F})$ and $M(\nicefrac{\Phi}{\F})$.
By ergodicity, 
this support must be contained in one of the two sets. 
By the induction hypothesis, $\nu$ is Lyapunov maximizing for $\Phi|_{\F}$ or for $\nicefrac{\Phi}{\F}$.
In either case, it is Lyapunov maximizing for $\Phi$, as we wanted to show.
\end{proof}

\subsection{Lyapunov almost-maximizing periodic orbits of low period}

Let $\Phi$ be a $\theta$-H\"older automorphism covering a hyperbolic homeomorphism.
For each integer $n \ge 1$, let
$$
\beta_n(\Phi) \coloneqq \max \big\{\chi_1(\Phi, \mu) \st \text{$\mu$ is supported on a periodic orbit of period $\le n$} \big\} \, .
$$
This is a bounded non-decreasing sequence, 
and so it is convergent.
Actually, the limit is:
\begin{equation}\label{e.BergerWang}
\lim_{n \to \infty} \beta_n(\Phi)  = \beta(\Phi) \, .
\end{equation}
Indeed, this follows from a much more general result of Kalinin \cite[Theorem 1.4]{Kalinin} on the approximation of Lyapunov exponents using measures supported on periodic orbits.
In the case of one-step cocycles, formula \eqref{e.BergerWang} is known as the \emph{Berger--Wang theorem}, and it was first proved in \cite{BWang}. 
For other extensions of Berger--Wang theorem, see \cite{Oregon,BreuF}.

It is quite possible that the limit \eqref{e.BergerWang} is attained for some finite $n$ (and indeed this is expected to be the typical situation).
On the other hand, in the worst-case scenario, what can we say about the speed of the approximation in formula \eqref{e.BergerWang}?
A result of Morris \cite{Morris_rapidly} says that for one-step cocycles, this speed is always superpolynomial.
Here we show that the same is true for strongly bunched automorphisms:

\begin{theorem}\label{t.super_pol}
If $\Phi$ is a strongly bunched automorphism then for every $\tau>0$, 
$$
\beta(\Phi) - \beta_n(\Phi) = O(n^{-\tau}) \, .
$$
\end{theorem}

The first result of superpolynomial approximation was actually obtained in the context of ergodic optimization of Birkhoff averages by Bressaud and Quas \cite{BQuas}, who also showed that this type of bound is essentially sharp. 
The key ingredient is a quantitative version of Anosov Closing Lemma, also due to Bressaud and Quas \cite{BQuas}, which we state as follows:

\begin{theorem}[Bressaud--Quas Closing Lemma]\label{t.BQ}
Let $T \colon X \to X$ be a hyperbolic homeomorphism.
Let $Y \subseteq X$ be a nonempty compact $T$-invariant set.
Then for every $\tau>0$ and every sufficiently large $n$, there exists a periodic orbit of period at most $n$ entirely contained in the $n^{-\tau}$-neighborhood of $Y$.
\end{theorem}

This result is proved in \cite{BQuas} for the one-sided full shift; as remarked in that paper, one can use standard techniques to reduce to that case. Alternatively, one can prove \cref{t.BQ} directly, and we do so in \cref{ss.BQ}.

\begin{proof}[Proof of \cref{t.super_pol}]
Recall from \eqref{e.split_beta} that if $\Phi$ is reducible then we can replace it by either a restricted or a quotient automorphism with the same maximal Lyapunov exponent.
Repeating this procedure a finite number of times, we eventually find a irreducible automorphism with the same maximal Lyapunov exponent; this induced automorphism will also be strongly bunched.
So, without loss of generality, we assume that $\Phi$ is irreducible.
By \cref{t.irr_to_span}, $\Phi$ is spannable, and by \cref{t.extremal}, $\Phi$ admits a H\"older extremal norm  $\tribar{\mathord{\cdot}}$. 
Let $p \in \ldbrack 1, d \rdbrack$ be maximal such that the $p$-th Mather set $M_p(\Phi) \eqcolon Y$ is nonempty.
By \cref{t.dom}, the restricted bundle $\E_Y$ admits a dominated splitting $\F \oplus \G$ where the dominating bundle $\F$ has fibers of dimension $p$ and is calibrated in the sense that for every $x \in Y$ and $u \in \F_x$, we have
$\tribar{\Phi(u)} = e^{\beta(\Phi)} \tribar{u}$. 
	
By robustness of dominated splittings \cite[Corol.~2.8]{CroPo}, there exists a closed neighborhood $U$ of $Y$ such that 
if $Z \coloneqq \bigcap_{k \in \Z} T^{-k}(U) \supseteq Y$ is the maximal invariant set in this neighborhood, 
then the restricted bundle $\E_Z$ over the compact invariant set  admits a dominated splitting $\F \oplus \G$, extending the previously found dominated splitting on $\E_Y$.
Recall that the bundles of a dominated splitting are H\"older-continuous \cite[Thrm.~4.11]{CroPo}.
Furthermore, the extremal norm is also H\"older-continuous.
It follows that there exist $\rho>0$ and $C_0>0$ such that for every $x \in Z$ and every $u \in \F_x$,
$$
e^{\beta(\Phi) - C_0 \dd(x,Y)^\rho} \tribar{u}  \le 
\tribar{\Phi(u)}  \le e^{\beta(\Phi)} \tribar{u} \, .
$$

Fix $\tau>0$.
Let $C$ be given by \cref{t.BQ}; so for all sufficiently large $n$, there exists a periodic orbit of period at most $n$ supported on the $C n^{-\tau}$-neighborhood of $Y$, and so contained in $Z$.
Let $\nu_n$ be the invariant probability measure supported on that orbit.
The bound obtained before implies: 
$$
\chi_1(\Phi,\nu) \ge \beta(\Phi) - C_0 (C n^{-\tau})^\rho \, .
$$
So $\beta(\Phi) - \beta_n(\Phi) = O(n^{-\rho \tau})$.
Since $\tau>0$ is arbitrary, the \lcnamecref{t.BQ} is proved.
\end{proof}

\appendix

\section{Appendix: Proof of some technical results}
\label{s.technical}

\subsection{Basic constructions on \texorpdfstring{$\theta$}{theta}-H\"older bundles}\label{ss.basic}

Recall our assumption from \cref{ss.theta} that the algebra of $\theta$-H\"older functions on $X$ is normal. Let us metrize the product $X \times X$ by $\dd\big( (x,y), (x',y') \big) \coloneq \max\big\{ \dd(x,y), \dd(x',y') \big\}$.

\begin{lemma}\label{l.square}
The algebra of $\theta$-H\"older functions on $X \times X$ is normal.
\end{lemma}

\begin{proof}
Let $K_0$, $K_1 \subset X \times X$ be two disjoint nonempty compact sets.
Let $\epsilon>0$ be a lower bound for the distance between a point in $K_0$ and a point in $K_1$. Let $\{B_i\}$ be a finite cover of $X$ by open sets of diameter less than $\epsilon$. Let $\{\rho_i\}$ be a partition of unity subordinated to this cover and formed by $\theta$-H\"older functions. Define a function $f \colon X \times X \to \R$ by:
$$
f(x,y) \coloneq \sum_{\substack{(i,j) \text{ such that} \\ (B_i \times B_j)\cap K_1 \neq \emptyset}} \rho_i(x) \rho_j(y) \, .
$$
Then $f$ is $\theta$-H\"older, takes values in the interval $[0,1]$, equals~$0$ on $K_0$, and equals~$1$ on $K_1$. This proves normality.
\end{proof}

\begin{proof}[Proof of \cref{p.transport} (existence of transport maps $I_{y \gets x}$)]\footnote{A different construction that provides the additional property $(I_{y \gets x})^{-1} = I_{x \gets y}$ (for sufficiently close $x$, $y$) can be found in \cite[p.~169]{KalSad}; however, we will not need that property.}
Consider the finite cover of $X \times X$ formed by the following open sets:
$$
V_{k,\ell} \coloneqq
\begin{cases}
	U_k \times U_k &\text{if $k=\ell$;} \\
	U_k \times U_\ell \setminus \Delta  &\text{if $k\neq\ell$,} 
\end{cases}
$$
where $\Delta \subseteq X \times X$ is the diagonal.
Consider a partition of unity subordinate to this cover, composed of $\theta$-H\"older functions $\rho_{k,\ell}$; its existence is a consequence of \cref{l.square}. 
Given any pair of points $x$, $y \in X$, define a linear map from $\E_x$ to $\E_y$ by: 
$$
I_{y \gets x} \coloneqq \sum_{(k, \ell)} \rho_{k,\ell}(x,y)  \, h_\ell(y) \circ [h_k(x)]^{-1}  \, ,
$$
where the sum is taken over the indices $(k, \ell)$ such that $V_{k,\ell} \ni (x,y)$.
If $(x,y) \in U_i \times U_j$ then the matrix:
$$
[h_j(y)]^{-1} \circ I_{y \gets x} \circ h_i(x) =
 \sum_{(k, \ell)} \rho_{k,\ell}(x,y)  \, g_{j \gets \ell}(y) \circ g_{k \gets i}(x)
$$
is $\theta$-H\"older continuous as a function of $(x,y)$,
and equals the identity when $x = y$.
\end{proof}

For the following proofs, it is convenient to fix another open cover $\{V_i\}$ of $X$ such that $\overline{V_i} \subset U_i$ for each $k$.
Note that for any Finsler norm $\| \mathord{\cdot} \|$ on $\E$, we have:
\begin{equation}\label{e.h_i_bounded}
\max_i \sup_{x \in V_i} \max \big\{ \|h_i(x)\| , \|[h_i(x)]^{-1}\| \big\} < \infty \, ,
\end{equation}
where these operators norms are relative to the norm $\| \mathord{\cdot} \|_x$ on $\E_x$ and the Euclidean norm  $\| \mathord{\cdot} \|_\mathrm{eucl}$ on $\R^d$.

\begin{proof}[Proof of \cref{p.transport_groupoid} (composition of transport maps)]
In order to prove the assertion, 
it is sufficient to consider triples of points $x$, $y$, $z$ that are close enough so that they belong to a common coordinate neighborhood $V_i$. Consider the matrix 
\begin{equation}\label{e.tilde_I}
\tilde I_{y \gets x} \coloneqq [h_i(y)]^{-1} \circ I_{y \gets x} \circ h_i(x) \, , 
\end{equation}
which by \cref{p.transport} is $O(\dd(x,y)^\theta)$-close to the identity matrix. 
Using a similar notation for the other points, we have:
$$
\| \tilde I_{z \gets x} - \mathrm{Id} \|_\mathrm{eucl} = O(\dd(x,z)^\theta) \quad \text{and} \quad 
\| \tilde I_{y \gets z} - \mathrm{Id} \|_\mathrm{eucl} = O(\dd(y,z)^\theta)  \, .
$$
Therefore: 
$$
\| \tilde I_{y \gets z} \circ \tilde I_{z \gets x} - \tilde I_{y \gets x} \|_\mathrm{eucl} = O \big( \max\{ \dd(x,z)^\theta, \dd(y,z)^\theta \} \big) \, .
$$
Since $I_{y \gets z} \circ I_{z \gets x} - I_{y \gets x}  = h_i(y) \circ \big( \tilde I_{y \gets z} \circ \tilde I_{z \gets x} - \tilde I_{y \gets x} \big) \circ [h_i(x)]^{-1}$,  using the boundedness property \eqref{e.h_i_bounded} we obtain 
$$
\| I_{y \gets z} \circ I_{z \gets x} - I_{y \gets x} \|_\mathrm{eucl} = O \big( \max\{ \dd(x,z)^\theta, \dd(y,z)^\theta \} \big) \, ,
$$
as we wanted to show.
\end{proof}

\begin{proof}[Proof of \cref{p.norm_Holder} (characterization of $\theta$-H\"older norms)]
Let $\|\mathord{\cdot}\|$ be a $\theta$-H\"older Finsler norm.
In order to prove the desired estimate, it is sufficient to consider pairs of points $x$, $y$ that are close enough so that they belong to a same set $V_i$.
By definition, for every $u \in \R^d$, the map $x \in V_i \mapsto \|h_i(x) u \|$ is $\theta$-H\"older, and so there is a constant $C > 0$ such that, for all $x$, $y \in V_i$,
\begin{equation}\label{i.intermediate_Holder}
\big| \| h_i(y) u \|  - \| h_i(x) u \|  \big| \le C \|u\|_\mathrm{eucl} \, \dd(x,y)^\theta \, .
\end{equation}
Using the boundedness property \eqref{e.h_i_bounded} and compactness of the unit sphere, we can find a uniform $C$ so that the estimate above holds for every $u \in \R^d$.

Recall that that the matrix defined in \eqref{e.tilde_I} satisfies $\| \tilde I_{y \gets x} -\mathrm{Id}\|_\mathrm{eucl} = O(\dd(x,y)^\theta)$.
Now, given $v \in \E_x$, consider $u \coloneqq [h_i(x)]^{-1} v$.
Then:
\begin{align*}
\big| \| I_{y \gets x} v\|  - \| v \|  \big| 
&= 
\big| \| h_i(y) \tilde I_{y \gets x} u \|  - \| h_i(x) u \|  \big| \\
&\le
\big| \| h_i(y) \tilde I_{y \gets x} u \|  - \| h_i(y) u \|  \big|  + 
\big| \| h_i(y) u\|  - \| h_i(x) u \|  \big| 
\end{align*}
Using \eqref{e.h_i_bounded} and \eqref{i.intermediate_Holder}, we conclude that $\big| \| I_{y \gets x} v\|  - \| v \|  \big| = O(\|v\| \dd(x,y)^\theta)$, that is, $\big| \| I_{y \gets x} \|  - 1  \big| = O(\dd(x,y)^\theta)$, as claimed.

The proof of the converse is entirely analogous. 
\end{proof}

\begin{proof}[Proof of \cref{p.endo_Holder} (characterization of $\theta$-H\"older endomorphisms)]
Suppose $\Phi$ is $\theta$-H\"older.
In order to prove the desired estimate, it is sufficient to consider pairs of points $x$, $y$ that are close enough so that they belong to a same set $V_i \cap T^{-1}(V_j)$. 
Let $\tilde \Phi_x \coloneqq [h_j(Tx)]^{-1} \circ \Phi_x \circ h_i(x)$ and similarly define $\tilde \Phi_y$.
Let $\tilde I_{y \gets x}$ be defined by \eqref{e.tilde_I}, and similarly define $\tilde I_{Tx \gets Ty}$;
by \cref{p.transport} these matrix-valued maps are $\theta$-H\"older as functions of $(x,y)$.
So the map $(x,y) \mapsto \tilde I_{Ty \gets Tx} \circ  \tilde\Phi_x - \tilde \Phi_y \circ \tilde I_{y \gets x}$
is also $\theta$-H\"older, and since it vanishes on $(x,x)$ we conclude that:
\begin{equation}\label{e.tildes}
\big\| \tilde I_{Ty \gets Tx} \circ \tilde \Phi_x - \tilde \Phi_y \circ \tilde I_{y \gets x} \big\| = O(\dd(x,y)^\theta ) \, .
\end{equation}
Using the boundedness property \eqref{e.h_i_bounded} we obtain:
$$
\big\| I_{Ty \gets Tx} \circ \Phi_x - \Phi_y \circ I_{y \gets x} \big\| = O(\dd(x,y)^\theta) \, ,
$$
as desired.

Conversely, assume that such an estimate holds; then \eqref{e.tildes} follows from \eqref{e.h_i_bounded}.
By \cref{p.transport}, the matrices $\tilde I_{y \gets x}$ and  $\tilde I_{Ty \gets Tx}$ are $O(\dd(x,y)^\theta)$-close to the identity.
It follows that the matrices $\tilde \Phi_x$ and $\tilde \Phi_y$ are $O(\dd(x,y)^\theta)$-close.
This means that  $\Phi$ is $\theta$-H\"older.
\end{proof}


\subsection{Existence of holonomies}\label{ss.holonomies}

We begin with a straightforward estimate:

\begin{lemma}\label{l.bol}
Let $\Phi \in \Aut^\theta_K(\E,T)$.
For every $x$, $y \in X$ and $n \ge 0$ we have:
$$
\| (\Phi^n_y)^{-1} \| \, \| \Phi^n_x \| \le \prod_{j=0}^{n-1} e^{K_1 \dd(T^j x, T^j y)^\theta} \bol(\Phi_{T^j y}) \, ,
$$
where $K_1$ depends only on $K$. 
\end{lemma}

\begin{proof}
By submultiplicativity of norms and the definition of bolicity, we have:
$$
\| (\Phi^n_y)^{-1} \| \, \| \Phi^n_x \|
\le \prod_{j=0}^{n-1} \| (\Phi_{T^j y})^{-1} \| \, \| \Phi_{T^j x} \|
= \prod_{j=0}^{n-1} \| \Phi_{T^j y} \|^{-1} \, \| \Phi_{T^j x} \| \, \bol(\Phi_{T^j y}) \, ,
$$
and so the claimed inequality holds with $K_1$ being the $\theta$-H\"older constant of $\log \|\Phi\|$.
This constant can be estimated in terms of $K$, using \eqref{p.norm_Holder}.
\end{proof}

\begin{proof}[Proof of \cref{p.holonomies} (existence of holonomies)]
By symmetry, it is sufficient to consider $\star = \ss$.
The stable holonomy is defined as:
\begin{equation}\label{e.def_holonomy}
H^\ss_{y \gets x} \coloneqq \lim_{n \to + \infty} \underbrace{(\Phi_y^n)^{-1} \circ I_{T^n y \gets T^n x} \circ \Phi_x^n}_{H_n} \, ,
\end{equation}
where $x$ and $y$ are in a same stable set.
Let us establish convergence.
Assume first that $y \in W^\ss_{\epsilon_0}(x)$.
We have: 
$$
H_{n+1} - H_n = (\Phi_y^{n+1})^{-1} \circ \underbrace{\big( I_{T^{n+1} y \gets T^{n+1} x} \circ \Phi_{T^n x} - \Phi_{T^n y} \circ I_{T^n y \gets T^n x} \big)}_{\Delta_n}  \circ \Phi_x^n  \, ,
$$
and so, using the definition \eqref{e.bounded_set} of the set $\Aut^\theta_K(\E,T)$ and \cref{l.bol},
\begin{align*}
\|H_{n+1} - H_n\| 
&\le K \|\Delta_n\| \,  \| (\Phi_y^n)^{-1}\| \,  \|\Phi_x^n\| \\ 
&\le K^2 \, \dd(T^n x, T^n y)^\theta \, \prod_{j=0}^{n-1} e^{K_1 \dd(T^j x, T^j y)^\theta} \bol(\Phi_{T^j y}) \, .
\end{align*}
By property \eqref{e.lambda_s} in the definition of hyperbolicity, for every $j \ge 0$ we have
$$
\dd(T^j x, T^j y) \le e^{-\lambda_\ss^{(j)}(y)} \dd(x,y) \, , \quad \text{where} \quad
\lambda_\ss^{(j)}(y) \coloneqq \sum_{i=0}^{j-1}\lambda_\ss(T^i y) \, . 
$$
Since $\lambda_\ss$ is strictly positive, the series $\sum_{j=0}^\infty \dd(T^j x, T^j y)^\theta$ is convergent.
Therefore
$$
\|H_{n+1} - H_n\| 
\le K_2 \, e^{-\theta \lambda_\ss^{(n)}(y)} \left( \prod_{j=0}^{n-1} \bol(\Phi_{T^j y}) \right) \dd(x,y)^\theta \, ,
$$
where $K_2>0$ is another constant.
Take a small constant $\eta>0$ such that the fiber-bunching condition \eqref{e.def_fiber_bunched} still holds if the right hand side is multiplied by $1-\eta$.
In particular, $\bol(\Phi_{T^j y})  < e^{(1-\eta)\theta\lambda_\ss(T^j y)}$ and so
\begin{equation}\label{e.exp_convergence}
\|H_{n+1} - H_n\| 
\le K_2 \, e^{-\eta \theta \lambda_\ss^{(n)}(y)} \dd(x,y)^\theta \, ,
\end{equation}
This establishes uniform exponential convergence in formula \eqref{e.def_holonomy} when $y \in W^\ss_{\epsilon_0}(x)$.
Using \eqref{e.longW} we see that convergence holds whenever $y \in W^\ss(x)$.
The groupoid properties (\ref{i.groupoid_1}) and (\ref{i.groupoid_2}) are the equivariance property (\ref{i.equivariance}) are automatic from the definition.
The H\"olderness property (\ref{i.holonomy_Holder}) follows by summing \eqref{e.exp_convergence} for $n=0$ to $\infty$, and noting that $H_0 = I_{y\gets x}$.
The joint continuity property (\ref{i.holonomy_cont}) also follows from the uniformity of our estimates.
Finally, if we consider a small $C^0$ perturbation of $\Phi$ in the set $\Aut^\theta_K(\E,T)$, then we can use the same constants $K_2$ and $\eta$ in \eqref{e.exp_convergence}, and so the remaining assertions of the \lcnamecref{p.holonomies} follow.
\end{proof}

\subsection{Regularity estimates}\label{ss.regularity}

In this \lcnamecref{ss.regularity}, we prove \cref{p.regularity_base,p.regularity_above}.
Before going into the proofs, let us state our estimates for H\"older exponents.

Since $T$ is Lipschitz, we can find $\epsilon_1 \in (0,\epsilon_0)$ and a continuous strictly positive function $\Lambda_\uu$ such that for all $x$, $x'$, $x''\in X$,
\begin{equation}\label{e.Lambda_u}
x',x''\in W^\uu_{\epsilon_1}(x) \quad \Rightarrow \quad
\left\{
\begin{array}{l}
	Tx', Tx'' \in W^\uu_{\epsilon_0}(Tx)  \, , \\
	\dd(Tx', Tx'') \le e^{\Lambda_\uu(x)} \dd(x', x'') \, .
\end{array}
\right.
\end{equation}
We will show that the conclusion of \cref{p.regularity_base} holds for any $\kappa_\ss$ in the range:
\begin{equation}\label{e.regularity_estimate_base}
0 < \kappa_{\ss} < \inf_{X} \frac{\lambda_\ss+\lambda_\uu}{\lambda_\ss+\Lambda_\uu} \, .
\end{equation}

Let
\begin{equation}\label{e.def_phi}
\phi(x) \coloneqq \log \bol(\Phi_x);
\end{equation}
We will show that the conclusion of \cref{p.regularity_above} holds for any $\kappa_\ss$ in the range:
\begin{equation}\label{e.regularity_estimate_above}
\theta_{\ss} < \inf_{X} \frac{\theta \lambda_\ss - \phi}{\lambda_\ss+\Lambda_\uu} \, .
\end{equation}
(Note that the numerator is positive by fiber-bunching.)

\medskip

The idea of the proof of \cref{p.regularity_base} is roughly as follows.
Without loss of generality, we can assume that the ``quadrilateral'' in \cref{f.rectangle} has a ``base'' $\dd(x,x')$ much smaller than the two ``legs'' $\dd(x,x')$, $\dd(x,y)$.
We take $N \ge 0$ as big as possible for which we can guarantee that the $T^N$-image of that quadrilateral has a base smaller than the legs.
We estimate the ``summit'' $\dd(T^N y, T^N y')$ using the triangle inequality, and finally we iterate backwards to obtain the desired estimate for $\dd(y, y')$.
The proof of \cref{p.regularity_above} uses the same ``there and back again'' idea.
Formal proofs follow.

Let us denote the  Birkhoff sums of a function $f \colon X \to \R$ as:
$$
f^{(n)} \coloneqq f + f \circ T + \cdots + f \circ T^{n-1} \, ,
\qquad f^{(0)} \coloneqq 0.
$$

\begin{lemma}\label{l.loss}
For any strictly positive continuous function $f$ on $X$ and any $a \in (0,1)$, there exists $b(f,a)>0$ such that for any $z \in X$ and any $n \ge 0$,
$$
z' \in W^\ss_{\epsilon_0}(z) \cup T^{-n}(W^\uu_{\epsilon_0}(T^n z)) \quad \Rightarrow \quad
f^{(n)}(z') \ge a f^{(n)}(z) - b(f,a) \, .
$$
\end{lemma}

\begin{proof}
It is sufficient to consider the case $z' \in W^\ss_{\epsilon_0}(z)$, since the case $z' \in T^{-n}(W^\uu_{\epsilon_0}(T^n z))$ follows by reversing the time.
By uniform continuity of $f$ and uniform contraction on local stable sets, we can find an integer $k = k(f,a) \ge 0$ such that if $z' \in W^\ss_{\epsilon_0}(z)$ then $f(T^j z') \ge a f(T^j z)$ for every $j \ge k$.
Letting $b(f,a) \coloneqq ak \sup_X f$, we obtain the desired conclusion.
\end{proof}

We start the proofs of \cref{p.regularity_base,p.regularity_above} with some estimates that are common to them.
Fix $\kappa_\ss$ and $\theta_\ss$ satisfying \eqref{e.regularity_estimate_base} and \eqref{e.regularity_estimate_above}, respectively.
Fix a number $a\in(0,1)$ sufficiently close to $1$ such that:
\begin{align}
\kappa_{\ss} &< \inf_{X} \frac{a (\lambda_\ss + \lambda_\uu)}{a \lambda_\ss + \Lambda_\uu} \, ,
\label{e.folga_base}
\\
\theta_{\ss} &< \inf_{X} \frac{a (\theta \lambda_\ss - \phi)}{a \lambda_\ss + \Lambda_\uu} \, .
\label{e.folga_above}
\end{align}

Fix four points $x$, $x'$, $y$, $y'$ satisfying \eqref{e.rectangle}.
Let $\delta \coloneqq \dd(x,x')$.
Note that to prove \cref{p.regularity_base,p.regularity_above}, it is sufficient to consider $\delta$ smaller than a fixed positive constant, say $\epsilon_1$ from \eqref{e.Lambda_u}.
Let $N$ be the largest nonnegative integer such that:
\begin{equation}\label{e.optimal}
a \lambda_\ss^{(N)}(x) + \Lambda_\uu^{(N)}(x) <  \log(\epsilon_1/\delta) \, ;
\end{equation}
Then:
\begin{equation}\label{e.optimal_other_side}
a \lambda_\ss^{(N)}(x) + \Lambda_\uu^{(N)}(x) \ge \log(\epsilon_1/\delta) - c \, ,
\end{equation}
for some constant $c$, namely $c \coloneqq \sup_X (a \lambda_\ss + \Lambda_\uu)$.
In particular, assuming that $\delta$ is small enough, $N$ will be large and so the following inequality will hold:
\begin{equation}\label{e.WLOG}
e^{a \lambda_\ss^{(N)}(x)} > 2 + e^{b(\lambda_\ss,a)}
\end{equation}
(where $b$ comes from \cref{l.loss}).

Using \eqref{e.Lambda_u} and \eqref{e.optimal}, one checks by induction that
the following chain of inequalities hold for each $n \in \ldbrack 0, N \rdbrack$:
\begin{equation}\label{e.base}
\dd(T^n x, T^n x')
\le \delta e^{\Lambda_\uu^{(n)}(x)}
\le \epsilon_1 e^{- a \lambda_\ss^{(n)}(x)}
\le \epsilon_1 \, .
\end{equation}
This gives estimates for the base of the ``quadrilateral'' obtained as the $T^n$-image of that of \cref{f.rectangle}.
Let us estimate the other sides; the ``legs'' are:
\begin{align}
\dd(T^n x , T^n y)  &\le \epsilon_0 e^{- \lambda_\ss^{(n)}(x)} \, ,  \label{e.leg1} \\
\dd(T^n x', T^n y') &\le \epsilon_0 e^{- \lambda_\ss^{(n)}(x')}
                   \le \epsilon_0 e^{- a\lambda_\ss^{(n)}(x) + b(\lambda_\ss,a)}  \, , \label{e.leg2}
\end{align}
where in the last inequality we used \cref{l.loss}.
Therefore the ``summit'' is: 
\begin{align}
\dd(T^n y, T^n y')
&\le \dd(T^n x, T^n y) + \dd(T^n x, T^n x') + \dd(T^n x', T^n y') \\
&\le (2 + e^{b(\lambda_\ss,a)}) \epsilon_0 e^{-a\lambda_\ss^{(n)}(x)} \label{e.summit1} \\
&\le \epsilon_0 \, , \label{e.summit2}
\end{align}
where in the last step we used assumption \eqref{e.WLOG}.

We estimate the base of the original quadrilateral by iterating backwards:
\begin{alignat*}{2}
\dd(y,y')
&\le e^{-\lambda_\uu^{(N)}(y)} \dd(T^N y, T^N y')
&\qquad&\text{(by \eqref{e.summit2})} \\
&= O \big( e^{-a \lambda_\uu^{(N)}(x)} \dd(T^N y, T^N y') \big)
&\qquad&\text{(by \cref{l.loss})} \\
&= O \big( e^{-a [\lambda_\ss^{(N)}(x) + \lambda_\uu^{(N)}(x)]} \big)
&\qquad&\text{(by \eqref{e.summit1})}  \\
&= O \big( e^{- \kappa_\ss [a \lambda_\ss^{(N)}(x) + \Lambda_\uu^{(N)}(x)]} \big)
&\qquad&\text{(by \eqref{e.folga_base})}  \\
&= O \big( \delta^{\kappa_\ss} \big)
&\quad&\text{(by \eqref{e.optimal_other_side})} \, . 
\end{alignat*}
This proves \cref{p.regularity_base}.

\medskip

We proceed to the proof of \cref{p.regularity_above}.
In what follows, the constants implicit in $O$ can be taken uniform on a a $C^0$-neighborhood of $\Phi$ in $\Aut^\theta_K(\E,T)$.
For $n \ge 0$, define
$$
\Gamma_n \coloneqq H^\uu_{T^n y' \gets T^n y} \circ H^\ss_{T^n y \gets T^n x} - H^\ss_{T^n y' \gets T^n x'} \circ H^\uu_{T^n x' \gets T^n x} \, .
$$
Let us estimate the norm of these linear maps.
First,
$$
\| \Gamma_0 \| \le \left\| H^\uu_{y' \gets y} \circ H^\ss_{y \gets x} - I_{y' \gets x} \right\|
+ \left\| H^\ss_{y' \gets x'} \circ H^\uu_{x'\gets x} - I_{y' \gets x} \right\|
\eqcolon \circled{1} + \circled{2} \, .
$$
We estimate the first term:
$$
\circled{1} \le
\left\| H^\uu_{y' \gets y} \circ (H^\ss_{y\gets x} - I_{y \gets x}) \right\| +
\left\| (H^\uu_{y' \gets y} - I_{y' \gets y}) \circ I_{y \gets x} \right\| +
\left\| I_{y' \gets y} \circ I_{y \gets x} - I_{y' \gets x} \right\| \, .
$$
Using \eqref{e.holonomy_Holder} and \cref{p.transport_groupoid}, we conclude that
$$
\circled{1} = O \big( \max \big\{ \dd(x,y)^\theta, \dd(y,y')^\theta \big\} \big) \, .
$$
An analogous reasoning yields:
$$
\circled{2} = O \big(  \max \big\{ \dd(x,x')^\theta, \dd(x',y')^\theta \big\} \big) \, .
$$
So we obtain:
$$
\|\Gamma_0 \| = O \big( \max \big\{ \dd(x,x')^\theta, \dd(y,y')^\theta, \dd(x,y)^\theta, \dd(x',y')^\theta \big\} \big) \, .
$$
Any of these four distances, say $\dd(y,y')$, is less than the sum of the other three; so:
$$
\|\Gamma_0 \| = O \big(  \max \big\{ \dd(x,x')^\theta, \dd(x,y)^\theta, \dd(x',y')^\theta \big\} \big) \, .
$$
Now if $n \in \ldbrack 0, N \rdbrack$, the corresponding quadrilateral has sides are bounded by $\epsilon_0$ (estimates \eqref{e.base}--\eqref{e.summit2}), and
the exact same argument yields:
$$
\|\Gamma_n \| = O \big( \max \big\{ \dd(T^n x,T^n x')^\theta, \dd(T^n x,T^n y)^\theta, \dd(T^n x',T^n y')^\theta \big\} \big) \, .
$$
Then, using estimates \eqref{e.base}, \eqref{e.leg1}, and \eqref{e.leg2} we obtain:
\begin{equation}\label{e.Delta}
\|\Gamma_n \| = O \big( e^{-\theta a \lambda_\ss^{(n)}(x)} \big) \, .
\end{equation}

As a consequence of estimates  \eqref{e.base} and \eqref{e.leg2}, if $j \in \ldbrack 0, N \rdbrack$ then $\dd(T^j y' , T^j x) = O \big( e^{-a\lambda_\ss^{(j)}(x)} \big)$. 
So it follows from \cref{l.bol} that
\begin{equation}\label{e.new_bol}
\| (\Phi^n_{y'})^{-1} \| \, \| \Phi^n_x \| = O\big(e^{\phi^{(n)}(y')}\big) , \quad \text{for } n \in \ldbrack 0, N \rdbrack \, .
\end{equation}

Now we want to iterate backwards to obtain a finer estimate for $\| \Gamma_0 \|$.
By the groupoid properties of holonomies,
$\Gamma_0 = (\Phi^n_{y'})^{-1} \circ \Gamma_n \circ \Phi^n_x$.
Therefore:
\begin{alignat*}{2}
\|\Gamma_0\|
&= O \big( e^{\phi^{(N)}(y')} \, \| \Gamma_N \| \big)
&\qquad&\text{(by \eqref{e.new_bol})} \\
&= O \big( e^{-\theta a \lambda_\ss^{(N)}(x) + a \phi^{(N)}(x)} \big)
&\qquad&\text{(by \eqref{e.Delta} and \cref{l.loss})} \\
&= O \big( e^{-\theta_\ss [a \lambda_\ss^{(N)}(x) + \Lambda_\uu^{(N)}(x)]} \big)
&\qquad&\text{(by \eqref{e.folga_above})} \\
&= O \big( \delta^{\theta_\ss} \big)
&\qquad&\text{(by \eqref{e.optimal_other_side})} \, .
\end{alignat*}
\cref{p.regularity_above} is proved.

\begin{proof}[Proof of \cref{l.needed_strength} (the strong bunching constant)]
Fix any positive
\begin{equation}\label{e.needed_strength}
\eta_0 \le \inf_X \frac{\theta \lambda_\ss}{\lambda_\ss + \lambda_\uu + \Lambda_\uu} \, .
\end{equation}
Suppose $\Phi$ is a $(\eta_0,\theta)$-bunched automorphism.
This means that the function $\phi$ defined by \eqref{e.def_phi} is less than $ \eta_0 \lambda_\uu$.
Therefore we have pointwise inequalities:
$$
\frac{\theta \lambda_\ss - \phi}{\lambda_\ss+\Lambda_\uu} > \frac{\theta \lambda_\ss - \eta_0 \lambda_\uu}{\lambda_\ss+\Lambda_\uu} \ge \eta_0. 
$$
So there exists $\theta_\ss \ge \eta_0$ that satisfies \eqref{e.regularity_estimate_above}.
\end{proof}

\subsection{The metric on the Grassmannian}\label{ss.Grass}

Let $E$ be an inner product space of dimension $d$.
If $V_1$, $V_2 \subseteq E$ are subspaces of the same dimension $p>0$, we define:
\begin{equation}\label{e.def_metric}
\dd(V_1,V_2) \coloneqq \inf_{F_1, F_2} \| F_1 - F_2 \| \, ,
\end{equation}
where each $F_i$ runs over all linear isomorphisms $F_i \colon \R^p \to V_i$ such that $\|F_i^{-1}\| \le 1$.
(We consider $\R^p$ endowed with the canonical inner product, and $\| \mathord{\cdot} \|$ always denotes the operator norm.)  

\begin{proposition}\label{p.metric}
$\dd$ is a metric on the Grassmannian $\Gr_p(E)$.
\end{proposition}

\begin{proof}
Symmetry and the triangular inequality are trivially satisfied, so let us check non-degeneracy.
Suppose $V_1 \neq V_2 \in \Gr_p(E)$.
Take a unit vector $v_1$ in $V_1$ but not in $V_2$.
Then there exists $\delta>0$ such that $\|v_1-v_2\| \ge \delta$ for every $v_2 \in V_2$.
For each $i \in \{1,2\}$, let $F_i \colon \R^p \to V_i$ be a linear isomorphism such that $\|F_i^{-1}\| \le 1$.
Then:
$$
\|F_1 - F_2\| \ge \frac{\|v_1 - F_2(F_1^{-1}(v_1))\|}{\|F_1^{-1}(v_1))\|} \ge \|v_1 - F_2(F_1^{-1}(v_1))\| \ge \delta \, .
$$
This shows that $\dd(V_1,V_2) \ge \delta > 0$.
\end{proof}

\begin{proof}[Proof of \cref{p.Lip_bol} (linearly induced maps are Lispchitz)]
Consider a linear isomorphism $L \colon E \to F$ between $d$-dimensional inner product spaces.
For each $i \in \{1,2\}$, let $V_i \in \Gr_p(E)$, and let $W_i \coloneqq L(V_i)$.
Let $F_i \colon \R^p \to V_i$ be a linear isomorphism such that $\|F_i^{-1}\| \le 1$.
Define  $G_i \colon \R^p \to W_i$ by $G_i \coloneqq \|L^{-1}\| \, L \circ F_i$.
Then $G_i$ is a linear isomorphism and $\|G_i^{-1}\| \le 1$.
So 
$$
\dd(W_1,W_2) \le \| G_1 - G_2 \| =  \|L^{-1}\| \,  \|L \circ F_1 - L \circ F_2 \| \le \bol(L)  \| F_1 - F_2 \|  \, .
$$
Taking infimum over the $F_i$'s, we obtain $\dd(W_1,W_2) \le \bol(L) \dd(V_1,V_2)$.
This proves that the map induced by $L$ has Lipschitz constant $\bol(L)$.
\end{proof}

\begin{proof}[Proof of \cref{p.close_id} (maps close to the identity)]
Suppose $L \colon E \to \E$ satisfies $\|L - \id \|\le \delta \le \tfrac{1}{2}$.
Note that $\|L^{-1}\| \le (1-\delta)^{-1}$.
Fix an arbitrary $V \in \Gr_p(E)$. 
Let $F_1 \colon \R^p \to V$ be an isometry, and let $F_2 \coloneqq (1-\delta)^{-1} L \circ F_1$.
Then $\|F_1^{-1}\|=1$, $\|F_2^{-1}\| \le 1$, and so
\begin{equation*}
\dd(V, LV) \le \|F_1 - F_2\| = \| \id - (1-\delta)^{-1} L\| \le \| \id - L\| +  \frac{\delta}{1-\delta} \|L\| \le \frac{2\delta}{1-\delta} \le 4 \delta.
\qedhere
\end{equation*}
\end{proof}

\begin{proof}[Proof of \cref{p.span_Lip} (span is locally Lipschitz)]
This is an easy consequence of the definition \eqref{e.def_metric}, and details are left to the reader.
\end{proof}

\begin{remark}
It can be shown that our metric \eqref{e.def_metric} coincides with the metric used in \cite[\S~A.1]{BPS}.
\end{remark}

\subsection{Typical fiber-bunched automorphisms are irreducible}\label{ss.irred_typical}

Recall from \cref{ss.auto} that $\End^\theta(\E,T)$ denotes the vector space of $\theta$-H\"older endomorphisms, which becomes a Banach space with the  $\theta$-H\"older norm \eqref{e.Holder_norm}. 
The set $\Aut^\theta(\E,T)$ of $\theta$-H\"older automorphisms 
and the subset $\cB \subset \Aut^\theta(\E,T)$ of fiber-bunched automorphisms 
are both open subsets of $\End^\theta(\E,T)$ (actually they are $C^0$-open).

A subset of a Banach space is said to be of \emph{infinite codimension} if it is locally contained in the union of finitely many closed submanifolds of arbitrarily large codimension.

\begin{proposition}\label{p.irred_typical}
Suppose $X$ is infinite, $T \colon X \to X$ is a transitive hyperbolic homeomorphism, and $\E$ is a $\theta$-H\"older vector bundle over $X$.
Then there exists an open and dense subset $\cI$ of the set $\cB \subset \Aut^\theta(\E,T)$ of fiber-bunched automorphisms such that every $\Phi \in \cI$ is irreducible.
Furthermore, the set $\cB \setminus \cI$ has infinite codimension.
\end{proposition}

The proof is an obvious adaptation of arguments from \cite{BGV,Viana}, so we will make it concise.

\begin{proof}
As a consequence of shadowing and expansivity, the hyperbolic homeomorphism $T$ has infinitely many periodic points (see e.g.\ \cite[p.~228]{Akin}). 
Select one of these, say a point $p$ of period $k$, and a homoclinic point $q$ associated to $p$. 
If $\Phi$ is reducible then it admits a non-trivial $\theta$-H\"older $\Phi$-invariant subbundle $\F$ which by \cref{c.irred} is both $H^\uu$- and $H^\ss$-invariant.
Then the subspace $\F_p \subseteq \E_p$ is invariant under two different linear maps, namely
$\Phi^k_p$ and $H^\uu_{p \gets q} \circ H^\uu_{q \gets p}$.
On the other hand, we claim that the property that these two maps admit a common nontrivial invariant subspace is \emph{atypical} in the space $\Aut^\theta(\E,T)$; more precisely, it has positive codimension and, a fortiori, empty interior.

First note that the property that an element of $\GL(d,\R)$ admits infinitely many invariant subspaces is atypical (because it implies the existence of a complex eigenvalue of geometric multiplicity bigger than $1$). So for typical $\Phi \in \Aut^\theta(\E,T)$, the collection of $\Phi^k_p$-invariant subspaces is finite.

On the other hand, choose a closed neighborhood $U$ of $q$ that is disjoint from the future and past iterates of $q$.
If we perturb the automorphism in this neighborhood (or rather in $\pi^{-1}(U)$) then the maps $\Phi^k_p$ and  $H^\uu_{q \gets p}$ are unaffected, but $H^\ss_{p \gets q}$ changes, and actually any small perturbation of $H^\ss_{p \gets q}$ can be realized with a perturbation of $\Phi$ supported in $U$.
In particular, with a well-chosen perturbation, the composition of holonomies sends each of the (finitely many) $\Phi^k_p$-invariant nontrivial subspaces of $\E_p$ into something transverse to it. Such an automorphism $\Phi$ cannot be reducible.
This shows that irreducibility has dense interior in $\Aut^\theta(\E,T)$. 

The argument actually shows that if $\Phi$ is reducible then it must satisfy infinitely many independent conditions of positive codimension, at least one for each homoclinic orbit.
Therefore reducibility has infinite codimension. See \cite[{\S}4]{Viana} for full details. 
\end{proof}

\subsection{Bressaud--Quas Closing Lemma}\label{ss.BQ}

Here we will prove \cref{t.BQ}.
Though our formulation is different, the key ideas come from \cite{BQuas}.

Let $f \colon Y \to Y$ be \emph{any} homeomorphism of a compact metric space $(Y,\dd)$.
For $\epsilon>0$, an \emph{$(\epsilon, \dd, f)$-pseudoorbit} is a string of points $(x_0,x_1, \dots, x_{n-1})$ such that
$\dd(f(x_i), x_{i+1})< \epsilon$ for every $i \in \ldbrack 0, n-2\rdbrack$.
If additionally $\dd(f (x_{n-1}), x_0)< \epsilon$ then we say that the pseudo-orbit is \emph{periodic}, with \emph{period} $n$; in that case indices can be taken as integers mod $n$ instead. 
Let $R(\epsilon, \dd, f)$ denote the minimal period of a periodic $(\epsilon, \dd, f)$-pseudoorbit.
Note that:
\begin{equation}\label{e.R_power}
R(\epsilon, \dd, f) \le n R(\epsilon, \dd, f^n) \quad \text{for every $n \ge 1$.}
\end{equation}

A set $E \subseteq Y$ is called \emph{$(\epsilon, \dd)$-separated} if $\dd(x,y)\ge \epsilon$ for every pair of distinct points $x$, $y\in E$.
Let $S(\epsilon,\dd)$ be the maximal cardinality of a $(\epsilon, \dd)$-separated set.

Define a sequence of metrics by:
$$
\dd_{n,f}(x,y) \coloneqq \max_{i \in \ldbrack 0, n-1\rdbrack} \dd(f^i(x), f^i(y)) \, .
$$

\begin{lemma}\label{l.BQ_estimate}
Let $\epsilon>0$.
Suppose that $R(\epsilon,\dd,f) > m > 0$.
Then:
$$
\log m \le \log S(\tfrac{\epsilon}{2}, \dd)  - \tfrac{1}{m} \log S(\epsilon, \dd_{m,f})  + 1 \, .
$$
\end{lemma}

\begin{proof}	
Let $E$ be a $(\dd, \tfrac{\epsilon}{2})$-separated set of maximal cardinality. 
Note that the $\dd$-balls of radius $\tfrac{\epsilon}{2}$ and centers at the points of $E$ cover $Y$, because otherwise we could enlarge $E$ by adding any point not covered. 

Let $F$ be a $(\dd_{m,f}, \epsilon)$-separated set of maximal cardinality. 
For each $y \in F$, choose a $m$-tuple $(x_0, \dots, x_{m-1})$ of points in $E$ such that $\dd(x_j, f^j(y)) < \tfrac{\epsilon}{2}$ for each $j \in \ldbrack 0, m-1 \rdbrack$.
First, we claim that these $x_j$'s are all distinct. Indeed, if $x_j = x_k$ with $j < k$, then $\dd(f^j(y), f^k(y)) < \epsilon$, so $\big( f^i(y) \big)_{i \in \ldbrack j, k-1 \rdbrack}$ is a periodic $(\epsilon,\dd)$-pseudoorbit of period $k-j \le m-1 < R(\epsilon,\dd)$, contradiction.

Second, we claim that if $y \neq y' \in F$ then the corresponding $m$-tuples $(x_0, \dots, x_{m-1})$ and $(x'_0, \dots, x'_{m-1})$ are distinct.
Indeed, if the two $m$-tuples coincide then for each $j \in \ldbrack 0, m-1\rdbrack$ we have $\dd(f^j(y), f^j(y')) < \epsilon$. This means that $\dd_{m,f}(y,y') < \epsilon$. Since the set $F$ is $(\dd_{m,f}, \epsilon)$-separated, we conclude that $y=y'$.

Third, we claim that if $y \neq y' \in F$ then the sets $\{x_0, \dots, x_{m-1}\}$ and $\{x'_0, \dots, x'_{m-1}\}$ are distinct.
Indeed, if the two sets coincide then $x'_i = x_{\sigma(i)}$ for some permutation $\sigma$ of $\ldbrack 0, m-1\rdbrack$.
By the previous claim, this permutation is not the identity; therefore there exists $\ell \in \ldbrack 0, m-2\rdbrack$ such that $k \coloneqq \sigma(\ell) > \sigma(\ell+1) \eqcolon j$. Then $\dd(f^{\ell}(y'), f^{k}(y)) < \epsilon$ and $\dd(f^{\ell+1}(y'), f^{j}(y)) < \epsilon$.
Therefore $\big(f^{j}(y), f^{j+1}(y), \dots, f^{k-1}(y), f^{\ell}(y') \big)$
is a periodic $(\epsilon,\dd)$-pseudoorbit of period $k-j+1 \le m < R(\epsilon,\dd)$, contradiction.

We conclude that the number of elements of the set $F$ cannot exceed the number of subsets of the set $E$ with exactly $m$ elements, that is,
$$
|F| \le \binom{|E|}{m} \le \frac{|E|^m}{m!} \le \left( \frac{e |E|}{m} \right)^m \, .
$$
Taking $\log$'s, recalling that $|E| = S(\dd, \tfrac{\epsilon}{2})$ and $|F| = S(\dd_{m,f}, \tfrac{\epsilon}{2})$, and rearranging, we obtain the inequality stated in the \lcnamecref{l.BQ_estimate}.
\end{proof}

\medskip

Recall that a homeomorphism $f \colon Y \to Y$ is called \emph{expansive} if there is a uniform separation between every pair of distinct orbits. In that case, the topological entropy $h_\mathrm{top}(f)$ is finite; furthermore, for every sufficiently small $\epsilon>0$, the limit
\begin{equation}\label{e.htop}
\lim_{n \to \infty} \frac{\log S(\epsilon,\dd_{n,f})}{n} \quad \text{exists and equals } h_\mathrm{top}(f); \nopagebreak 
\end{equation}
see \cite[p.~174, 177]{Walters}.

\begin{lemma}\label{p.subexp}
If $f$ is expansive then for every sufficiently small $\epsilon>0$ we have:
$$
\lim_{n \to \infty} \frac{\log R(\epsilon,\dd_{n,f},f)}{n} = 0 \, .
$$
\end{lemma}

\begin{proof}
Fix a small $\epsilon>0$ and a large integer $n$.
Write $R_n \coloneqq R(\epsilon,\dd_{n,f},f)$ and  $m_n \coloneqq \lfloor (R_n-1)/n \rfloor$.
Assume that $m_n>0$, otherwise $(\log R_n)/n$ is already small.
Using \eqref{e.R_power}, 
we have $R(\epsilon,\dd_{n,f},f^n) > m_n$.
Applying \cref{l.BQ_estimate}, 
we obtain that 
$$
\frac{\log m_n}{n} \le \frac{\log S(\epsilon/2, \dd_{n,f})}{n} - \frac{\log S(\epsilon, \dd_{nm_n,f})}{nm_n} + \frac{1}{n} \, .
$$
By \eqref{e.htop}, the right-hand side is small: the first two terms essentially cancel each other.
It follows that $(\log R_n)/n$ is small.
\end{proof}

\begin{proof}[Proof of \cref{t.BQ}]
Given the hyperbolic homeomorphism $T$ and the compact $T$-invariant set $Y \neq \emptyset$, let 
$f$ be the restriction of $T$ to $Y$.
Hyperbolic homeomorphisms are expansive (recall \cref{r.hyperb}), so $f$ is expansive as well.

Fix $\epsilon>0$ small enough so that \cref{p.subexp} applies.
Note that if $(x_i)_{i \in \Z/k\Z}$ is a periodic $(\epsilon, \dd_n,f)$-pseudoorbit then, letting $y_i \coloneqq f^{\lceil n/2 \rceil}(x_i)$, we have, for all $i \in \Z/k\Z$,
$$
\max_{j \in \left\ldbrack -\lceil n/2 \rceil, \lceil n/2 \rceil -1 \right\rdbrack}
\dd ( f^{j+1}(y_i), f^j(y_{i+1}) ) < \epsilon \, .
$$
Hyperbolicity implies that 
$\dd ( f(y_i), y_{i+1}) ) < C e^{-\lambda n} \epsilon \eqcolon \epsilon_n$,
where $C$ and $\lambda$ are positive constants.
That is, $(y_i)_{i \in \Z/k\Z}$ is a periodic $(\epsilon_n, \dd, f)$-pseudoorbit. 
So $R(\epsilon, \dd_{f,n}, f) \ge R(\epsilon_n, \dd, f) \eqcolon N_n$.
In particular, 
$$
\frac{\log N_n}{n} \quad\text{and}\quad
\frac{\log N_n}{\log \epsilon_n^{-1}}
\quad\text{also tend to $0$ as $n \to \infty$.}
$$
Therefore, for any given $\tau>0$, if $n$ is large enough then $\epsilon_n < N_n^{-\tau}$.
By definition, there exists a periodic $(\epsilon_n, \dd, T)$-pseudoorbit of period $N_n$ in the set $Y$.
By the Lipschitz shadowing lemma \cite[Thrm.~2]{Sakai}, there exist a periodic orbit for $T$ of period $N_n$ within distance $O(\epsilon_n) = O(N_n^{-\tau})$.
This proves the \lcnamecref{t.BQ}.
\end{proof}

\section{Appendix: Examples}\label{s.examples}

Here we present examples that show some of the limits of our results. 

\subsection{Examples of sets of calibrated vectors with exceptional behavior}\label{ss.examples_calibrated}

The following two examples show that the set $\K$ of calibrated vectors defined by \eqref{e.calibrated} can have exceptional fibers, justifying \cref{r.exceptional_fibers}.

\begin{example}\label{ex.non_space}
Let $T \colon X \to X$ be a hyperbolic homeomorphism having a fixed point $x_0$.
Let $f$ be a non-negative H\"older function vanishing only at $x_0$.
Consider the cocycle
$$
A(x) \coloneqq \begin{pmatrix} 1 & 0 \\ 0 & e^{-f(x)} \end{pmatrix} \, .
$$
Then the corresponding automorphism $\Phi$ on the trivial bundle $\E \coloneqq X \times \R^2$ has $\beta(\Phi) = 0$, and its Mather sets are $M_1(\Phi) = X$ and $M_2(\Phi) = \{x_0\}$.
The max norm $\tribar{(u_1,u_2)} \coloneqq \max \{|u_1|,|u_2|\}$ is extremal.
Consider the corresponding set $\K$ of calibrated vectors, defined by \eqref{e.calibrated}.
If $x \in W^\uu(x_0) \setminus \{x_0\}$ then the fiber
$$
\K_x = \big\{ (u_1,u_2) \in \R^2 \st |u_2| \le e^{\sum_{n=1}^\infty f(T^{-n} x)} |u_1| \big\}
$$
is not a subspace.
\end{example}

\begin{example}\label{ex.bad_dim}
Suppose $T \colon X \to X$ is a homeomorphism admitting two nonempty compact invariant sets $X_1$, $X_2$ such that:
\begin{itemize}
	\item each $X_i$ equals the support of some $T$-invariant probability measure $\mu_i$;
	\item $X_1 \cup X_2 = X$;
	\item $X_1 \cap X_2 = \{x_0\} \cup \{T^n y_0 \st n \in \Z\}$ where $x_0$ is a fixed point and $y_0 \neq x_0$ is an homoclinic point.
\end{itemize}
Let $f$ be a non-negative continuous function vanishing only at $X_1 \cap X_2$.
Consider the cocycle
$$
A(x) \coloneqq \begin{pmatrix} 1 & 0 \\ 0 & e^{-f(x)} \end{pmatrix} \quad \text{if $x \in X_1$,} \qquad
A(x) \coloneqq \begin{pmatrix} e^{-f(x)} & 0 \\ 0 & 1 \end{pmatrix} \quad \text{if $x \in X_2$.} 
$$
Then the corresponding automorphism $\Phi$ on the trivial bundle $\E \coloneqq X \times \R^2$ has $\beta(\Phi) = 0$, and its Mather sets are $M_1(\Phi) = X$ and $M_2(\Phi) = \{x_0\}$.
The Euclidean norm is extremal. 
Consider the corresponding set $\K$ of calibrated vectors, defined by \eqref{e.calibrated}.
Then $\K_{y_0} = \R^2$ despite the fact that $y_0 \in M_1(\Phi) \setminus M_2(\Phi)$.
\end{example}

\begin{remark}
One may contend that ``correctly'' defined Mather sets  should not lie in the base $X$, but instead in the bundle $\E$, or in its 
projectivization $\hat \E$.  
Fix a norm $\| \mathord{\cdot} \|$ on $\E$ and define a function $f \colon \hat\E \to \R$ by $f([u]) \coloneqq \log(\|\Phi u\| / \|u\|)$. Let $\hat \M$ be the union of the supports of all probability measures on $\hat \E$ that are invariant under the automorphism $\hat{\Phi}$ and that maximize the integral of the function $f$. Let $\M \coloneqq \{u \in \E \st u=0 \text{ or } [u] \in \hat{\M} \}$. This is a closed subset of $\E$ that projects down on the Mather set $M(\Phi) \subseteq X$. Given an extremal norm, it is clear that the fibers of $\M$ are calibrated, i.e.\ $\M_x \subseteq \K_x$ for every $x \in M(\Phi)$. A stronger property actually holds: $\mathrm{span}(\M_x) \subseteq \K_x$ for every $x \in M(\Phi)$; we omit the proof. However, $\M_x$ may fail to be a subspace. Indeed, in \cref{ex.bad_dim} the set $\M_{y_0}$ is a union of two lines.
\end{remark}

\subsection{On Riemannian extremal norms}\label{ss.Riemann}

After having established the existence of extremal \emph{Finsler} norms (under appropriate hypotheses), one naturally wonders about the existence of extremal \emph{Riemannian} norms.
Let us begin with a weak positive result:

\begin{proposition}\label{p.Riem_weak}
In the situation of \cref{t.dom}, there exists a Riemannian norm $\| \mathord{\cdot} \|'$ such that for all $x \in Y$, the spaces $\F_x$ and $\G_x$ are orthogonal, and 
$$
\| \Phi(v) \|' = e^{\beta(\Phi)} > \|\Phi(w)\|'
\quad \text{for all unit vectors $v \in \F_x$, $w \in \G_x$.}
$$
\end{proposition}

\begin{proof}
As usual, assume $\beta(\Phi) = 0$.
For each $x \in Y$, consider the restriction of the extremal norm to the space $\F_x$, and let $\cB_x \subseteq \F_x$ be the unit ball.
Let $\cE_x$ be the John ellipsoid of $\cB_x$, namely the unique ellipsoid of maximal volume contained in $\cB_x$ (see e.g.\ \cite{Ball}). 
This field of ellipsoids is continuous, since finding the John ellipsoid is a continuous operation (as a consequence of its uniqueness).
Consider the Riemannian norm on the bundle $\F$ whose unit balls are the $\cE_x$'s. Since $\Phi_x (\cB_x) = \cB_{Tx}$ and the John ellipsoid is equivariant with respect to linear isomorphisms, we obtain $\Phi_x (\cE_x) = \cE_{Tx}$. This means that the Riemannian norm just constructed on the bundle $\F$ is preserved by $\Phi$.\footnote{Incidentally, note that if the John ellipsoid were monotonic with respect to set inclusion, then we could use it to ``Riemannize'' any given Finsler extremal norm. However, monotonicity fails: consider for instance a pair of rectangles as in \cref{f.MO}.}

In the bundle $\G$, we use the standard construction of Lyapunov norms (see e.g.\ \cite[p.~667]{KH}).
Fix a small positive $\epsilon$, and for each $x \in M_p$ and $w \in \G_x$, let:
$$
\| w \|'_x \coloneqq \left( \sum_{n=0}^\infty e^{2 \epsilon n} \|\Phi_x^n (v) \|^2 \right)^{1/2} \, .
$$
As a consequence of domination, the series converges exponentially, so the formula yields a well-defined continuous Riemannian norm on the bundle $\G$.
It is immediate that $\|\Phi_x(w)\|' \le e^{-\epsilon} \|w\|'$, so the norm along $\G$ is uniformly contracted.
Finally, we extend the Riemannian norm to the fibers $\E_x$ for $x \in Y$ by declaring $\F_x$ and $\G_x$ to be orthogonal.
This completes the construction. 
\end{proof}

The Riemannian norm $\| \mathord{\cdot} \|'$ provided by \cref{p.Riem_weak} is extremal over the restricted subbundle $\E_Y = \pi^{-1}(Y)$. 
Can one extend this Riemannian norm to the whole bundle, keeping it extremal?
The answer is no, as we will see next.

\medskip

We will present an example of an irreducible fiber-bunched automorphism in dimension $2$ that admits no \emph{Riemannian} extremal norm.

Consider the following two matrices: 
\begin{equation}\label{e.two_matrices}
A_0 \coloneqq \begin{pmatrix} 0 & -1 \\ 1 & 0 \end{pmatrix} , \qquad
A_1 \coloneqq \begin{pmatrix} 0.8 & -0.1 \\ 0.8 & 0.1 \end{pmatrix} .
\end{equation}
Let $(T,F)$ be the corresponding one-step cocycle (see \cref{ex.one-step}),
and let $\Phi$ be the corresponding automorphism of the trivial vector bundle $\E = X \times \R^2$. 
Consider a H\"older exponent $\theta = 1$, and take the parameter $\lambda$ in the metric \eqref{e.ultrametric} large enough so that $\Phi$ becomes fiber-bunched. 
Consider the fixed point $p=(p_n)$ where each $p_n \coloneqq 0$.
Since $F(p) = A_0$ has non-real eigenvalues, the automorphism $\Phi$ is irreducible: there can be no nontrivial $\Phi$-invariant subbundle.
Then \cref{t.extremal} yields the existence of an extremal norm.
Actually, the max norm in $\R^2$, defined by 
$\tribar{(u_1,u_2)} \coloneqq \max\{|u_1|,|u_2|\}$, is an extremal norm.
Indeed, the operator norms of our two matrices are: 
$$
\tribar{A_0} = 1 \, , \quad \tribar{A_1} = 0.9 \, .
$$
Since the spectral radius of $A_0$ is $1$, it follows that $\beta(\Phi) = 0$, and so $\tribar{\mathord{\cdot}}$ is a (constant) extremal norm, as claimed. Also note that $\delta_p$ is the unique Lyapunov-maximizing measure.

\begin{proposition}\label{p.no_Riemann}
The automorphism $\Phi$ admits no Riemannian extremal norm.
\end{proposition}

\begin{proof}
Assume for a contradiction that $\Phi$ admits a Riemannian extremal norm $\{\| \mathrm{\cdot} \|_x\}_{x \in X}$.
Since $e^{\beta(\Phi)} = 1$ is the spectral radius of $F(p) = A_0$, we must have:
$$
\| A_0 \|_{p \gets p} \le 1 \, ,
$$
in the operator norm notation \eqref{e.def_operator_norm}.
This means that if $D\subset \R^2$ denotes the unit ball in the norm $\| \mathord{\cdot} \|_p$, we have $A_0(D) \subseteq D$.
Since the norm is assumed to be Riemannian, $D$ is a (filled) ellipse, and since $A_0$ is a rotation, this ellipse must be a disk. Rescaling the norm if necessary, we can assume that $D$ is the unit disk. Equivalently, $\| \mathord{\cdot} \|_{p \gets p}$ is the usual Euclidean operator norm, which for emphasis we will write $\| \mathord{\cdot} \|_{\mathrm{eucl}}$.

Consider the homoclinic point $q \coloneq (\dots 0,0,\underdot{1},0,0\dots)$, i.e., the sequence that has a unique symbol $1$ at position $0$.
Note that for any $k > 0$ we have the identity:
\begin{equation}\label{e.loop_identity}
H^\ss_{p \gets q} \circ H^\uu_{q \gets p} = \Phi_p^{-k} \circ H^\ss_{p \gets T^k q} \circ \Phi^{2k}_{T^{-k} q} \circ H^\uu_{T^{-k} q \gets p} \circ \Phi_p^{-k} \, .
\end{equation}
In particular, taking $k=1$, by triviality of local holonomies \eqref{e.trivial_holonomies} we obtain:
\begin{align*}
H^\ss_{p \gets q} \circ H^\uu_{q \gets p} &= A_0^{-1} \circ \id \circ A_1A_0 \circ \id \circ A_0^{-1} \\ &= A_0^{-1} A_1 \, .
\end{align*}
Using that $A_0$ preserves Euclidean norm, applying the extremal Riemannian norm to \eqref{e.loop_identity}:
$$
\|A_1 \|_{\mathrm{eucl}} = \|A_0^{-1} A_1\|_{\mathrm{eucl}}  = \|H^\ss_{p \gets q} \circ H^\uu_{q \gets p}\|
\le
\underbrace{\|\Phi_p^{-k}\|}_{=1} \ \underbrace{\|H^\ss_{p \gets T^k q}\|}_{\to 1} \ \underbrace{\|\Phi^{2k}_{T^{-k} q}\|}_{\le 1} \ \underbrace{\| H^\uu_{T^{-k} q \gets p}\|}_{\to 1} \ \underbrace{\|\Phi_p^{-k}\|}_{=1} \, .
$$
Taking $k \to \infty$ we obtain that
$\|A_1 \|_{\mathrm{eucl}} \le 1$.
This is a contradiction: actually $\|A_1 \|_{\mathrm{eucl}} = 0.8 \sqrt{2} > 1$ (see \cref{f.MO}). 
\end{proof}

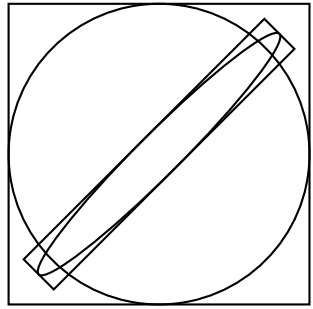
\begin{figure}[hbt]
	\begin{center}
		\begin{tikzpicture}[scale=2]
			\draw[thick] (1,1)--(-1,1)--(-1,-1)--(1,-1)--cycle;
			\draw[thick] (.7,.9)--(-.9,-.7)--(-.7,-.9)--(.9,.7)--cycle;
			\draw[thick] (0,0) circle (1) ;
			\draw[thick, rotate=45] (0,0) ellipse (1.1314 and 0.1414);
		\end{tikzpicture}
		\caption{The unit balls with respect the two norms and their images under $A_1$.}
		\label{f.MO}
	\end{center}
\end{figure}

Let us comment on other properties of our example.
We claim that there are perturbations $\tilde \Phi$ of $\Phi$ for which the measure $\delta_p$ ceases to be Lyapunov-maximizing; so the ``locking property'' (\emph{verrouillage}) is not satisfied.
Indeed, let $k \gg 1$ be an integer, let $m \coloneqq 4k+2$, and let $\tilde A_0$ be rotation matrix of angle $\frac{\pi}{2} - \frac{\pi}{4m}$.
Then:
$$
\tilde{A}_0^{m} A_1 = \begin{pmatrix} -0.8 \sqrt{2} & 0 \\ 0 & -0.1 \sqrt{2}\end{pmatrix} \, .
$$
Consider the associated one-step cocycle $\tilde F$, and the associated automorphism $\tilde \Phi$.
Then the probability measure $\tilde \mu$ supported on the orbit of the periodic point
$$
\tilde p \coloneqq
(\dots, \underdot{1}, \underbrace{0, \dots, 0}_m , 1, \underbrace{0, \dots, 0}_m , \dots) , \quad \text{i.e., $\tilde p_n = 1$ iff $m+1$ divides $n$,}
$$
has Lyapunov exponent
$$
\chi_1(\tilde\Phi, \tilde \mu) = \frac{\log(0.8 \sqrt{2})}{m+1} > 0 = \chi_1(\tilde\Phi, \delta_p) \, ,
$$
showing that $\delta_p$ was ``unlocked''.
Therefore the argument of the proof of \cref{p.no_Riemann} does not apply to the perturbation $\tilde\Phi$, and it is possible that these perturbations $\tilde \Phi$ admit Riemannian extremal norms (though there is no obvious candidate).
So the main property of our example $\Phi$, namely not to possess Riemannian extremal norms, may be fragile.
Going beyond this specific example, we ask:

\begin{question}
Let $T \colon X \to X$ be a hyperbolic automorphism.
Let $\E$ be a $2$-dimensional $\theta$-H\"older vector bundle over $X$.
Let $\cB \subset \Aut^\theta(\E,T)$ be the set of fiber-bunched irreducible automorphisms, endowed with the $\theta$-H\"older topology.
Let $\cR \subset \cB$ be the subset of automorphisms that admit a Riemannian extremal norm.
Is $\cR$ dense in $\cB$?
Is the interior of $\cR$ dense in $\cB$?
\end{question}

\bigskip 
\begin{ack}
We are very much indebted to Rafael Potrie for numerous illuminating and influential conversations. 
We also thank Clark Butler and Kiho Park for interesting discussions, and the referee for corrections and suggestions.
\end{ack}


\bigskip

{\footnotesize

\noindent Jairo Bochi

\noindent Facultad de Matem\'aticas, Pontificia Universidad Cat\'olica de Chile

\noindent \href{http://www.mat.uc.cl/~jairo.bochi}{www.mat.uc.cl/$\sim$jairo.bochi}

\noindent \href{mailto:jairo.bochi@mat.uc.cl}{jairo.bochi@mat.uc.cl}

\medskip

\noindent Eduardo Garibaldi

\noindent IMECC, Unicamp

\noindent \href{http://www.ime.unicamp.br/~garibaldi/}{www.ime.unicamp.br/$\sim$garibaldi}

\noindent \href{mailto:garibaldi@ime.unicamp.br}{garibaldi@ime.unicamp.br}

}
	
\end{document}